\documentclass[12pt]{article}
\usepackage{stmaryrd}
\usepackage{amsmath,latexsym,amssymb,amsfonts,amsbsy}
\usepackage{tipa}
\usepackage{mathrsfs}

\begin{document}

\author{MO, Guoduan \\
%EndAName
Dept.of Appl.Math.,Tongji Univ., Shanghai, China, 200092}
\title{ A PROOF ON HYPOTHESIS OF DIRICHLET DIVISOR PROBLEM}
\date{}
\maketitle

\begin{abstract}
We have proved that $\triangle(X)=X^{\frac{1}{4}+\varepsilon(X)}$, $%
\varepsilon(X)=\frac{c}{\log\log X}$.
\end{abstract}

$\underset{}{}$\newtheorem{them}{Theorem}[section] %
\newtheorem{lem}[them]{Lemma} \newtheorem{mylem}{Lemma}[section] %
\renewcommand{\themylem}{\arabic{section}.\arabic{mylem}} %
\newtheorem{coro}[them]{Corollary} \newtheorem{mypro}{Proposition}[section] %
\newtheorem{pro}[them]{Proposition} \newenvironment{proof}{{\noindent\it
Proof:\/}}{\hfill{$\Box$}} \numberwithin{equation}{section}

\textbf{Key Words:}\ \ Dirichlet Divisor Problem, 1/4-Hypothesis.

\textbf{2000 Mathematics Subject Classification: }11N37

\baselineskip 0.2in

\ \

\section{RESULTS AND THOUGHT OF PROOF}

Let
\begin{equation}
D(X)=\sum\limits_{n\leq X}d(n),
\end{equation}%
where $d(n)$ denotes the number of divisors of $n$ as usual, let

\begin{equation}
\triangle (X)=D(X)-X\log X-(2\gamma -1)X,
\end{equation}%
and let $\theta =\inf \{\theta ^{\ast }:\triangle (X)\ll X^{\theta ^{\ast
}}\}$. \ Then the following bounds for $\theta $ have been obtained:

\begin{equation*}
\begin{array}{lllll}
1. & \theta\leq\frac{1}{2} & =0.50000\cdots & \mbox {Dirichlet} & 1849 . \\
\mbox{} & \mbox{} & \mbox{} & \mbox{} &  \\
2. & \theta\leq\frac{1}{3} & =0.33333\cdots & \mbox {Voronoi} & 1903 \\
\mbox{} & \mbox{} & \mbox{} & \mbox{} &  \\
3. & \theta\leq\frac{33}{100} & =0.33000\cdots & \mbox {Vander  Corput} \cite%
{[6]} & 1922 \\
\mbox{} & \mbox{} & \mbox{} & \mbox{} &  \\
4. & \theta\leq\frac{27}{82} & =0.32926\cdots & \mbox {Van der  Corput} \cite%
{[7]} & 1928 \\
\mbox{} & \mbox{} & \mbox{} & \mbox{} &  \\
5. & \theta\leq\frac{15}{46} & =0.32608\cdots & \mbox {Chih,T. T.} \cite{[8]}%
(1950), \mbox{Richert} \cite{[9]} & 1953 \\
\mbox{} & \mbox{} & \mbox{} & \mbox{} &  \\
6. & \theta\leq\frac{13}{40} & =0.32500\cdots & \mbox {Yin  Wenlen} \cite%
{[10]} & 1959 \\
\mbox{} & \mbox{} & \mbox{} & \mbox{} &
\end{array}%
\end{equation*}

\begin{equation*}
\begin{array}{lllll}
7. & \theta\leq\frac{12}{37} & =0.32432\cdots & \mbox {Kolesnik}
\cite{[11]}
& 1969 \\
\mbox{} & \mbox{} & \mbox{} & \mbox{} &  \\
8. & \theta\leq\frac{346}{1067} & =0.32427\cdots & \mbox {Kolesnik} \cite%
{[12]} & 1973\\ 9. & \theta\leq\frac{35}{108} & =0.32407\cdots &
\mbox {Kolesnik} \cite{[13]}
& 1982 \\
\mbox{} & \mbox{} & \mbox{} & \mbox{} &  \\
10. & \theta\leq\frac{139}{429} & =0.32400\cdots & \mbox {Kolesnik} \cite%
{[14]} & 1985 \\
\mbox{} & \mbox{} & \mbox{} & \mbox{} &  \\
11. & \theta\leq\frac{7}{22} & =0.31818\cdots & \mbox {Iwniec and Mozzochi}
\cite{[15]} & 1988 \\
\mbox{} & \mbox{} & \mbox{} & \mbox{} &  \\
12. & \theta\leq\frac{131}{416} & =0.31490\cdots & \mbox
{M.N.Huxley}\cite{[16]} & 2003%
\end{array}%
\end{equation*}

\noindent The hypothesis value is $\theta =1/4$. Moreover, in 1916, Hardy
\cite{[15]} proved that $\theta \geq 1/4$.

In this paper we are going to prove the hypothesis is true. For this
purpose, we will prove the following theorem:

\ \

{\bf Theorem} Let $X_{1}<X_{2}$ be large positive number, $X_{1}\leq
X\leq X_{2}$ and
\begin{equation}
L=\log X_{2}, \ X_{1}\asymp X_{2}.
\end{equation}
Then
\begin{equation}
\triangle(X)=X^{\frac{1}{4}}\Lambda(X)+\delta(X),
\end{equation}
where $\delta(X)\ll X^{\varepsilon}$, $\varepsilon>0$ arbitrary
small, and
$$\Lambda(X)=\frac{X}{(m_{0}+1)^{2}}\sum\limits_{a=1}^{4K} C_{1}(a)
\sum\limits_{n\leq\frac{X_{2}L^{2}}{V}}
d(n)n^{-\frac{5}{4}}\cos(4\pi\sqrt{n}(\sqrt{X}+\frac{am_{0}}{2\sqrt{X}})
-\frac{3\pi}{4}) \ \ \ \ \ \ \ \ \ \ \ \ $$
\begin{equation}+\frac{\sqrt{X}}{m_{0}+1}\sum\limits_{a=1}^{4K+2}C_{2}(a)
\sum\limits_{n\leq\frac{X_{2}L^{2}}{V}}d(n)n^{-\frac{7}{4}}\cos(4\pi\sqrt{n}(\sqrt{X}+\frac{am_{0}}{2\sqrt{X}})-\frac{5\pi}{4}),\end{equation}
with
\begin{equation}
K=[\frac{C_{0}L}{\log L}],\ C_{0}\geq 200,
\end{equation}
$C_{1}(a),C_{2}(a)$ being independent of $X$, and
\begin{equation}C_{1}(a),\  C_{2}(a)\ll \exp(O(\frac{L}{\log L})),
\ m_{0}=2[\sqrt{X_{2}}L^{-2}].\end{equation}

\ \

Clearly, $L$, $K$ and $m_{0}$ are independent of $X$ by (1.3) and (1.6). The
following Corollary is obtained immediately by this Theorem.

\ \

{\bf Corollary}
\begin{equation}
\triangle(X)\ll X^{\frac{1}{4}+\varepsilon(X)},\
0<\varepsilon(X)=O(\frac{1}{\log\log X}).
\end{equation}

\begin{proof}
In fact,
$$
\Lambda(X)\ll\exp(O(\frac{L}{\log
L}))(\sum\limits_{n=1}^{\infty}
d(n)n^{-\frac{5}{4}}+\sum\limits_{n=1}^{\infty}
d(n)n^{-\frac{7}{4}})\ll\exp(O(\frac{L}{\log L}))\ll
X^{(O(\frac{1}{\log L}))},
$$

\noindent where the last inequality is given by $L\asymp\log
X_{1}\asymp\log X$.
\end{proof}

\ \

Clearly, the hypothesis is true by this Corollary.\\

To prove Theorem 1.1, suppose that
\begin{equation}
U=[X],
\end{equation}

\begin{equation}
A=C_{0}\sqrt{UL},\ C_{0}\geq 200,
\end{equation}

\begin{equation}
B=\sqrt{U}+\frac{k+j}{2\sqrt{U}}+\eta,
\end{equation}

\begin{equation}
k=0, 1, \cdots , m\ll m_{0}=[\sqrt{X_{2}}L^{-2}],
\end{equation}

\begin{equation}
j=0, \pm 1, \cdots , \pm j_{0},\ j_{0}\leq\sqrt{V}L,
\end{equation}

\begin{equation}
V=X_{2}^{\varepsilon_{0}}, \ \varepsilon_{0}=\frac{1}{\log L},
\end{equation}

\begin{equation}
\eta=\eta_{1}+\cdots +\eta_{K},\ K=[\frac{C_{0}L}{\log L}],
\end{equation}

\begin{equation}
\eta_{\lambda}=\frac{k_{1}}{\sqrt{U}},\ \lambda=1, 2, \cdots , K,
\end{equation}

\begin{equation}
k_{1}=0, 1, \cdots, m_{1}\ll m_{0}=[\sqrt{X_{2}}L^{-2}],
\end{equation}

\begin{equation}
\xi_{1}=\frac{1}{16\sqrt{U}},\ \xi_{2}=\frac{3}{16\sqrt{U}},
\end{equation}

\begin{equation}
S(B)=A\int\limits_{-\frac{5\sqrt{L}}{A}}^{\frac{5\sqrt{L}}{A}} e^{-\pi
A^{2}\theta^{2}}d\theta\sum\limits_{\xi_{1}<\sqrt{n}-B-\theta<%
\xi_{2}}d(n)n^{-\frac{1}{4}}.
\end{equation}

Clearly,
\begin{eqnarray}
(B+\theta +\xi _{2})^{2}-(B+\theta +\xi _{1})^{2} &=&(\xi _{2}-\xi
_{1})(2B+2\theta +\xi _{1}+\xi _{2})  \notag \\
&=&\frac{1}{8\sqrt{U}}(2\sqrt{U}+O(1))  \notag \\
&<&\frac{1}{3}
\end{eqnarray}%
and
\begin{eqnarray*}
\{(B+\theta +\xi )^{2}\} &=&\{(\sqrt{U}+\frac{k+j}{2\sqrt{U}}+\eta +\theta
+\xi )^{2}\} \\
&=&\{U+k+j+2\sqrt{U}\eta +2\sqrt{U}\xi +2\sqrt{U}\theta +(\frac{k+j}{2\sqrt{U%
}}+\eta +\theta +\xi )^{2}\} \\
&=&\{2\sqrt{U}\xi +2\sqrt{U}\theta +O(L^{-1})\}.
\end{eqnarray*}%
By (1.15) and (1.16) we know that $2\sqrt{U}\eta $ is an integer,
thus
\begin{equation*}
2\sqrt{U}(\xi +\theta )\geq 2\sqrt{U}(\frac{1}{16\sqrt{U}}-\frac{5\sqrt{L}}{A%
})=\frac{1}{8}-\frac{10}{C_{0}}\geq \frac{1}{8}-\frac{1}{20}=\frac{3}{40}
\end{equation*}%
and
\begin{equation*}
2\sqrt{U}(\xi +\theta )\leq 2\sqrt{U}(\frac{3}{16\sqrt{U}}+\frac{5\sqrt{L}}{A%
})=\frac{3}{8}+\frac{10}{C_{0}}\leq \frac{3}{8}+\frac{1}{20}=\frac{17}{40},
\end{equation*}%
therefore,
\begin{equation}
\frac{1}{20}<\{(B+\theta +\xi )^{2}\}<\frac{9}{20}
\end{equation}%
for $\frac{1}{16\sqrt{U}}\leq \xi \leq \frac{3}{16\sqrt{U}}$, $-\frac{5\sqrt{%
L}}{A}\leq \theta \leq \frac{5\sqrt{L}}{A}$. It follows from (1.20) and
(1.21) that there doesn't exist any integer in the interval
\begin{equation*}
\lbrack (B+\theta +\xi _{1})^{2},(B+\theta +\xi _{2})^{2}].
\end{equation*}%
Hence
\begin{equation}
\sum\limits_{\xi _{1}<\sqrt{n}-B-\theta \leq \xi
_{2}}=\sum\limits_{(B+\theta +\xi _{1})^{2}<n\leq (B+\theta +\xi
_{2})^{2}}=0,
\end{equation}%
and by (1.19),
\begin{equation}
S(B)=0.
\end{equation}%
Denote
\begin{equation}
f(\eta )\mbox {\LARGE \textbar
}_{\eta _{K}}=f(\eta _{1}+\eta _{2}+\cdots +\eta _{K})%
\mbox {\LARGE
\textbar }_{\eta _{1}=0}^{\eta _{1}=\frac{k_{1}}{\sqrt{U}}}\cdots
\mbox {\LARGE \textbar
}_{\eta _{K}=0}^{\eta _{K}=\frac{k_{1}}{\sqrt{U}}}.
\end{equation}%
By (1.23),
\begin{equation}
S(B)\mbox {\LARGE \textbar }_{\eta _{K}}=0.
\end{equation}%
Therefore,
\begin{equation}
\Omega =\frac{1}{\sqrt{V}}\sum\limits_{-\sqrt{V}L\leq j\leq \sqrt{V}L}e^{%
\frac{-\pi j^{2}}{V}}\sum_{k_{1}=m_{0}+1}^{2m_{0}}\sum_{m=m_{0}+1}^{2m_{0}}%
\sum_{k=0}^{m}S(B)\mbox
{\LARGE \textbar }_{\eta _{K}}=0.
\end{equation}%
On the other hand, in the following we are going to prove that
\begin{equation}
\Omega =\frac{(-1)^{K+1}(m_{0}+1)^{2}}{4X^{\frac{1}{4}}}(\triangle (X)-X^{%
\frac{1}{4}}\Lambda (X)+O(X^{\varepsilon })).
\end{equation}%
This and (1.26) lead to (1.4).

\ \

For convenience in the proof we are going to use the following definition.

\ \

\textbf{Definition 1.1} If
\begin{equation*}
f(\eta )\mbox {\LARGE \textbar
}_{\eta _{K}}=f(\eta _{1}+\eta _{2}+\cdots +\eta _{K}){%
\mbox {\LARGE
\textbar }}_{\eta _{1}=0}^{\eta _{1}=\frac{k_{1}}{\sqrt{U}}}\cdots {%
\mbox {\LARGE \textbar
}}_{\eta _{K}=0}^{\eta _{K}=\frac{k_{1}}{\sqrt{U}}}\ll Y,
\end{equation*}%
then we write
\begin{equation}
f(\eta )=O_{\eta }(Y),\ \mathrm{or}\ f(\eta )\ll _{\eta }Y.
\end{equation}%
Moreover, we denote

\begin{equation*}
f(\eta )\mid \mbox {\LARGE \textbar }_{\eta _{K}}\mid =\sum\limits_{\eta
_{1}=0,\frac{k_{1}}{\sqrt{U}}}\cdots \sum\limits_{{\eta _{1}=0,\frac{k_{K}}{%
\sqrt{U}}}}\mbox {\LARGE
\textbar }f(\eta _{1}+\eta _{2}+\cdots +\eta _{K})\mbox {\LARGE
\textbar }.
\end{equation*}%
Clearly,
\begin{equation}
f(\eta )\mbox {\LARGE \textbar }_{\eta _{K}}\leq f(\eta )|%
\mbox
{\LARGE \textbar }_{\eta _{K}}|.
\end{equation}%
Particularly, if $f(\eta )\ll M$, then
\begin{equation}
f(\eta )\mbox {\LARGE \textbar }_{\eta _{K}}\leq f(\eta )|%
\mbox {\LARGE
\textbar }_{\eta _{K}}|\ll M2^{K}.
\end{equation}%
For arbitrary small $\varepsilon >0$, or $\varepsilon =\varepsilon
(U)=O(1/\log L)$, we are going to use expressions, such as
\begin{equation}
U^{2\varepsilon }\ll U^{\varepsilon },
\end{equation}
without explanation.

\ \

\section{SOME LEMMAS (I)}

Let
\begin{equation*}
S^{\prime }=\sum\limits_{a\leq n\leq b}{}^{\prime }d(n)f(n),
\end{equation*}%
where $0<a<b<\infty $, $f(x)\in C^{2}[a,b)$, $\sum^{\prime }$ is that if $a$
or $b$ is an integer, then $\frac{1}{2}d(a)n(a)$ or $\frac{1}{2}d(b)n(b)$ is
to be counted instead of $d(a)n(a)$ or $d(b)n(b)$, then (for example, (3.2),
(3.15) in \cite{[2]})
\begin{eqnarray}
S^{\prime } &=&\sum\limits_{a\leq n\leq b}{}^{\prime }d(n)f(n)  \notag \\
&=&\int_{a}^{b}(\log x+2\gamma )f(x)dx+\sum\limits_{n=1}^{\infty
}d(n)\int_{b}^{a}f(x)(\Theta (nx)+O((nx)^{\frac{-5}{4}}))dx,
\end{eqnarray}%
where
\begin{equation}
\Theta (nx)=\sqrt{2}(nx)^{-1/4}(\cos (4\pi \sqrt{nx}+\pi /4)-\frac{1}{32\pi
\sqrt{nx}}\cos (4\pi \sqrt{nx}+3\pi /4)).
\end{equation}

\ \

In the following we are going to evaluate general sum $\sum
d_{\lambda}(n)f(n)$ in $\Gamma-$ function with more accurate $O$-term, and
indicate the results in the finite term. Denote
\begin{equation}
S_{\lambda}=\frac{1}{h_{0}^{\lambda-1}}\int_{0}^{h_{0}}\cdots
\int_{0}^{h_{0}} dh_{1}\cdots dh_{\lambda-1}\sum\limits_{a+h<n^{\frac{1}{%
\lambda}}\leq b+h}d_{\lambda}(n)f(n),
\end{equation}
where
\begin{equation}
d_{\lambda}(n)=\sum\limits_{n_{1}\cdots n_{\lambda}=n}1, \ \lambda\geq2,
\end{equation}

\begin{equation}
f(x)\in C^{\lambda}[a^{\lambda},(b+(\lambda-1)h_{0})^{\lambda}], \ 0<h_{0}<%
\frac{a}{2},
\end{equation}

\begin{equation*}
h=h_{1}+\cdots+h_{\lambda-1}.
\end{equation*}

\ \

\begin{lem}
Let
$$
\beta_{0}=\sum\limits_{j=1}^{\lambda} \left
(\begin{array}{clr}\lambda \\ j\end{array}\right
)((\lambda-1)h_{0})^{j}b^{\lambda-j},
$$

\noindent if  $h_{0}$ is sufficient small such that
$$
\beta_{0}<\min(1-\{a^{\lambda}\},\ 1-\{b^{\lambda}\}),
$$

\noindent then
\begin{equation}
S_{\lambda}=\sum\limits_{a<n^{\frac{1}{\lambda}}\leq
b}d_{\lambda}(n)f(n).
\end{equation}
\end{lem}

In fact, in this case, there are not any integer in $(a^{\lambda},(a+h)^{%
\lambda}]$ and $(b^{\lambda},(b+h)^{\lambda}]$,

\noindent hence,
\begin{equation*}
\sum\limits_{a^{\lambda}<n\leq(a+h)^{\lambda}}=0,\sum\limits_{b^{\lambda}<n%
\leq(b+h)^{\lambda}}=0.
\end{equation*}

\ \

\begin{lem}
For  $0.9<a<b<\infty$,
\begin{equation}
S_{\lambda}=S_{\lambda l}+\sum\limits_{n\leq
N}\frac{d_{\lambda}(n)}{h_{0}^{\lambda-1}}\int_{0}^{h_{0}}\cdots\int_{0}^{h_{0}}
dh_{1}\cdots
dh_{\lambda-1}\int_{(a+h)^{\lambda}}^{(b+h)^{\lambda}}f(x)\Theta_{\lambda}(nx)dx+O(\varepsilon_{1})
\end{equation}
holds for any $\varepsilon_{1}>0$ as $N\geq N(\varepsilon_{1})$,
where
\begin{equation}
S_{\lambda l}=\frac{1}{h_{0}^{\lambda-1}}\int_{0}^{h_{0}}\cdots
\int_{0}^{h_{0}} dh_{1}\cdots dh_{\lambda-1}\int_{(a+h)^{\lambda}}^{(b+h)^{\lambda}}f(n)p_{\lambda}(x)dx,
\end{equation}
with
$$
p_{\lambda}(x)=\frac{1}{2\pi
i}\int_{|S|=1/2}x^{s}\varsigma^{\lambda}(s+1)ds=
\frac{1}{(\lambda-1)!}(\log x)^{\lambda-1}+C_{1}(\log
x)^{\lambda-2}+\cdots+C_{\lambda},
$$
\begin{eqnarray}
\Theta_{\lambda}(nx)
&=&\frac{2}{\sqrt{\lambda}}(nx)^{-\frac{1}{2}+\frac{1}{2\lambda}}\sum\limits_{\alpha=0}^{\alpha_{\lambda}}
\frac{\gamma_{\alpha}^{(\lambda)}(nx)^{\frac{-\alpha}{\lambda}}}{(2\pi\lambda)^{\alpha}}\cos(2n\lambda(nx)^{\frac{1}{\lambda}}
+\frac{(\lambda-1)\pi}{4}+\frac{\alpha\pi}{2})\notag\\
&&+O(nx)^{\frac{-1}{2}+\frac{1}{2\lambda}-\frac{\alpha_{\lambda}+1}{\lambda}}
\end{eqnarray}
and
$$
\gamma_{0}^{(\lambda)}=1,\
\gamma_{1}^{(\lambda)}=\frac{1-\lambda^{2}}{12},
$$
for any natural number $\alpha_{\lambda}$.
\end{lem}

\ \

(In this paper we need the case $\lambda=2$ only.)

\ \

\begin{lem}
$$
S=S_{2}=\frac{1}{h_{0}}\int_{0}^{h_{0}}
dh\sum\limits_{a+h<\sqrt{n}\leq b+h}d(n)f(n) \ \ \ \ \ \ \ \ \ \ \ \
\ \ \ \ \ \ \ \ \ \ $$
\begin{equation}
=S_{l}+\sum\limits_{n\leq N}\frac{d(n)}{h_{0}}\int_{0}^{h_{0}}dh \int_{(a+h)^{2}}^{(b+h)^{2}}f(x)\Theta(nx)dx+O(\varepsilon_{1})
\end{equation}
holds for any $\varepsilon_{1}>0$ as $N\geq N(\varepsilon_{1})$,
where

$$
0.9<a<b<\infty, \ 0<h_{0}<a/2,\  f(x)\in C^{2}[a,b+h_{0}],
$$
\begin{equation}
S_{l}=\frac{1}{h_{0}}\int_{0}^{h_{0}}dh
\int_{(a+h)^{2}}^{(b+h)^{2}}f(x)(\log x+2\gamma)dx,
\end{equation}
\begin{equation}
\Theta(nx)=\sqrt{2}(nx)^{-\frac{1}{4}}\sum\limits_{\alpha=0}^{\alpha_{0}}
\frac{\gamma_{\alpha}(nx)^{\frac{-\alpha}{2}}}{(4\pi)^{\alpha}}\cos(4\pi
\sqrt{nx}+\frac{\pi}{4}+
\frac{\alpha\pi}{2})+O((nx)^{\frac{-3}{4}-\frac{\alpha_{0}}{2}}),
\end{equation}
with
\begin{equation}
\gamma_{0}=1,\ \gamma_{1}=-\frac{1}{8},\ \gamma_{2}=\frac{9}{128},\
\gamma_{3}=-\frac{75}{1024},\gamma_{4}=\frac{3675}{2^{15}},\
\gamma_{5}= -\frac{59535}{2^{18}},\
\gamma_{6}=\frac{2401245}{2^{22}}\cdots
\end{equation}
 for any natural number $\alpha_{0}$.
\end{lem}

(in this paper we need not the concrete number values of $\gamma_{j}$, $%
j\geq1$, we see $\gamma_{0}$, $\gamma_{1}$ be consistent for known results
in (2.2).)

A corollary of Lemma 2.3 is the following Lemma 2.4

\begin{lem}
Let
$$
S=\frac{1}{h_{0}}\int_{h_{1}}^{h_{2}}g(h)dh\sum\limits_{a+h<\sqrt{n}\leq
b+h}d(n)f(n),
$$
where $a+h_{1}>0.9$, $h_{1}<h_{2}$, $0<a<b<\infty$, $g(h)\in
C^{2}[h_{1},h_{2}]$, $f(x)\in C^{2}[(a+h_{1})^{2},(b+h_{2})^{2}]$,
$h_{0}=\int_{h_{1}}^{h_{2}}g(h)dh>0$, then

\begin{eqnarray*}
S&=&\frac{1}{h_{0}}\int_{h_{1}}^{h_{2}}g(h)dh
\int_{(a+h)^{2}}^{(b+h)^{2}}f(x)(\log x+2\gamma)dx \\
&&+\sum\limits_{n\leq N}\frac{d(n)}{h_{0}}\int_{h_{1}}^{h_{2}}g(h)dh
\int_{(a+h)^{2}}^{(b+h)^{2}}f(x)\Theta(nx)dx+O(\varepsilon_{1})
\end{eqnarray*}
holds for any $\varepsilon_{1}>0$ as $N\geq N(\varepsilon_{1})$,
where $\Theta(nx)$ is (2.12).
\end{lem}

Clearly, Lemma 2.2 has also similar corollary.

\ \

\textit{Proof} of Lemma 2.2: Let
\begin{equation}
D_{\lambda }(x)=\sum\limits_{n\leq x}d_{\lambda }(n),
\end{equation}%
then

\begin{equation*}
\sum\limits_{(a+h)^{\lambda }<n\leq (b+h)^{\lambda }}d_{\lambda
}(n)f(n)=\int_{(a+h)^{\lambda }}^{(b+h)^{\lambda }}f(x)dD_{\lambda }(x)\ \ \
\ \ \ \ \ \ \ \ \ \ \ \ \ \ \ \ \ \ \
\end{equation*}%
\begin{equation}
=f(x)D_{\lambda }(x){\mbox {\LARGE \textbar }}_{x=(a+h)^{\lambda
}}^{x=(b+h)^{\lambda }}-\int_{(a+h)^{\lambda }}^{(b+h)^{\lambda }}f^{\prime
}(x)D_{\lambda }(x)dx.
\end{equation}%
the last integral is%
\begin{equation*}
-\int_{(a+h)^{\lambda }}^{(b+h)^{\lambda }}f^{\prime }(x)D_{\lambda
}(x)dx=-\int_{(a+h)^{\lambda }}^{(b+h)^{\lambda }}f^{\prime
}(x)(\int_{0}^{x}D_{\lambda }(y)dy)_{x}^{\prime }dx\ \ \ \ \ \ \ \ \ \ \ \ \
\ \ \ \ \ \ \ \ \ \ \ \
\end{equation*}

\begin{eqnarray}
&=&-f^{\prime }(x)\int_{0}^{x}D_{\lambda }(y)dy{\mbox {\LARGE \textbar
}}_{x=(a+h)^{\lambda }}^{x=(b+h)^{\lambda }}+\int_{(a+h)^{\lambda
}}^{(b+h)^{\lambda }}f^{\prime \prime }(x)(\int_{0}^{x}D_{\lambda }(y)dy)dx
\notag \\
&=&\frac{-f^{\prime }(x)}{2\pi i}\int_{\sigma _{0}-iT}^{\sigma _{0}+iT}\frac{%
x^{s+1}\varsigma ^{\lambda }(s)ds}{s(s+1)}{\mbox {\LARGE
\textbar }}_{x=(a+h)^{\lambda }}^{x=(b+h)^{\lambda }}  \notag \\
&&+\frac{1}{2\pi i}\int_{\sigma _{0}-iT}^{\sigma
_{0}+iT}ds\int_{(a+h)^{\lambda }}^{(b+h)^{\lambda }}\frac{f^{\prime \prime
}(x)x^{s+1}\varsigma ^{\lambda }(s)}{s(s+1)}dx+\delta ,
\end{eqnarray}%
where $\sigma _{0}=1.1$ and

\begin{eqnarray*}
\delta &=&\frac{1}{2\pi i}(\int_{\sigma _{0}-i\infty }^{\sigma
_{0}-iT}ds+\int_{\sigma _{0}+iT}^{\sigma _{0}+i\infty }ds)(-\frac{f^{\prime
}(x)x^{s+1}\varsigma ^{\lambda }(s)}{s(s+1)}{\mbox {\LARGE \textbar }}%
_{x=(a+h)^{\lambda }}^{x=(b+h)^{\lambda }} \\
&&+\int_{(a+h)^{\lambda }}^{(b+h)^{\lambda }}\frac{f^{\prime \prime
}(x)x^{s+1}\varsigma ^{\lambda }(s)}{s(s+1)}dx) \\
&\ll& \int_{T}^{\infty }\frac{dt}{t^{2}}\ll \frac{1}{T}.
\end{eqnarray*}%
By (2.16),
\begin{equation}
-\int_{(a+h)^{\lambda }}^{(b+h)^{\lambda }}f^{\prime }D_{\lambda }(x)dx=-%
\frac{1}{2\pi i}\int_{(a+h)^{\lambda }}^{(b+h)^{\lambda }}f^{\prime
}\int_{\sigma _{0}-iT}^{\sigma _{0}+iT}\frac{x^{s}\varsigma ^{\lambda }(s)ds%
}{s}+O(\frac{1}{T}).
\end{equation}

Suppose $e$ is the set of $h_{\lambda-1}$ and $e$ satisfies $%
|x_{\lambda}+h_{\lambda-1}-m^{\frac{1}{\lambda}}|<\mu$, \noindent
where $m$
is an natural number, $x_{\lambda}=x+h_{1}+\cdots+h_{\lambda-2} $ fixed, $%
E=(0,h_{0})\setminus e$. We have

\begin{equation*}
\int_{e}f((x_{\lambda}+h_{\lambda-1})^{\lambda})D_{\lambda}((x_{\lambda}+h_{%
\lambda-1})^{\lambda})dh_{\lambda-1}\ll\mu,
\end{equation*}

\begin{equation*}
\int_{E}f((x_{\lambda}+h_{\lambda-1})^{\lambda})D_{\lambda}((x_{\lambda}+h_{%
\lambda-1})^{\lambda})dh_{\lambda-1} \ \ \ \ \ \ \ \ \ \ \ \ \ \ \ \ \ \ \ \
\ \ \ \ \ \ \ \ \
\end{equation*}

\begin{equation*}
\ \ \ \ \ \ \ =\int_{E}\frac{f((x_{\lambda }+h_{\lambda -1})^{\lambda })}{%
2\pi i}dh_{\lambda -1}\int_{\sigma _{0}-iT}^{\sigma _{0}+iT}\frac{%
(x_{\lambda }+h_{\lambda -1})^{\lambda s}\varsigma ^{\lambda }(s)}{s}%
ds+\delta _{E},
\end{equation*}%
where

\begin{eqnarray*}
\delta_{E}&=&\int_{E}\frac{f((x_{\lambda}+h_{\lambda-1})^{\lambda})}{2\pi i}%
dh_{\lambda-1}
(\int_{\sigma_{0}-i\infty}^{\sigma_{0}-iT}+\int_{\sigma_{0}+iT}^{%
\sigma_{0}+i\infty}) \frac{(x_{\lambda }+h_{\lambda-1})^{\lambda
s}\varsigma^{\lambda}(s)ds}{s} \\
&=&\int_{E}\frac{f((x_{\lambda}+h_{\lambda-1})^{\lambda})}{2\pi i}%
dh_{\lambda-1}\sum\limits_{n=1}^{\infty}
d_{\lambda}(n)(\int_{\sigma_{0}-i\infty}^{\sigma_{0}-iT}+\int_{%
\sigma_{0}+iT}^{\sigma_{0}+i\infty}) \frac{1}{s}(\frac{x_{\lambda}+h_{%
\lambda-1}}{n^{\frac{1}{\lambda}}})^{\lambda s}ds \\
&\ll&\frac{1}{T}\int_{E}dh_{\lambda-1}\sum\limits_{n=1}^{\infty}d_{%
\lambda}(n)(\frac{(x_{\lambda}+h_{\lambda-1})^{\lambda}}{n})^{\sigma_{0}}
\frac{1}{|\log\frac{x_{\lambda}+h_{\lambda-1}}{n^{\frac{1}{\lambda}}}|} \\
&\ll&\frac{1}{T}\int_{E}dh_{\lambda-1}(1+\sum\limits_{\frac{2}{3}\leq|n^{%
\frac{1}{\lambda}}-x_{\lambda}-h_{\lambda-1}|\leq \frac{3}{2}}\frac{1}{|n^{%
\frac{1}{\lambda}}-x_{\lambda}-h_{\lambda-1}|}) \\
&\ll&\frac{1}{T}(1+\sum\limits_{\frac{1}{12}\leq|n^{\frac{1}{\lambda}%
}-x_{\lambda}|\leq3}\int_{E} \frac{dh_{\lambda}}{|n^{\frac{1}{\lambda}%
}-x_{\lambda}-h_{\lambda-1}|}) \\
&\ll&\frac{1}{T}\log\frac{1}{\mu}
\end{eqnarray*}

\noindent for $0<h_{0}<a/2$, $0<\mu <1/2$, and the constant in $^{\prime
\prime }O^{\prime \prime }$ is independent of $T$. Furthermore,
\begin{equation*}
\int_{e}f((x_{\lambda }+h_{\lambda -1})^{\lambda })dh_{\lambda
-1}\int_{\sigma _{0}-iT}^{\sigma _{0}+iT}\frac{(x_{\lambda }+h_{\lambda
-1})^{\lambda s}\varsigma ^{\lambda }(s)}{s}dS\ll \int_{e}\log Tdh_{\lambda
-1}\ll \mu \log T.
\end{equation*}

\noindent Hence, taking $\mu =\frac{1}{T}$, we have
\begin{equation*}
\int_{0}^{h_{0}}\cdots \int_{0}^{h_{0}}dh_{1}\cdots dh_{\lambda
-1}f(x)D_{\lambda }(x){\mbox {\LARGE \textbar }}_{x=(a+h)^{\lambda
}}^{x=(b+h)^{\lambda }}=\ \ \ \ \ \ \ \ \ \ \ \ \ \ \ \ \ \ \ \ \ \ \ \ \ \
\ \ \ \ \ \ \ \ \ \ \ \ \ \
\end{equation*}%
\begin{equation}
=\int_{0}^{h_{0}}\cdots \int_{0}^{h_{0}}dh_{1}\cdots dh_{\lambda -1}f(x)%
\frac{1}{2\pi i}\int_{\sigma _{0}-iT}^{\sigma _{0}+iT}\frac{%
f(x)x^{s}\varsigma (s)ds}{s}{\mbox
{\LARGE \textbar }}_{x=(a+h)^{\lambda }}^{x=(b+h)^{\lambda }}+O(\frac{\log T%
}{T}).
\end{equation}%
It follows from (2.3), (2.15), (2.17) and (2.18) that%
\begin{eqnarray}
S_{\lambda } &=&\frac{1}{h_{0}^{\lambda -1}}\int_{0}^{h_{0}}\cdots
\int_{0}^{h_{0}}dh_{1}\cdots dh_{\lambda -1}\frac{1}{2\pi i}\int_{\sigma
_{0}-iT}^{\sigma _{0}+iT}\frac{\varsigma ^{\lambda }(s)ds}{s}%
(-\int_{(a+h)^{\lambda }}^{(b+h)^{\lambda }}f^{\prime }(x)x^{s}dx  \notag \\
&&+f(x)x^{s}{\mbox {\LARGE
\textbar }}_{x=(a+h)^{\lambda }}^{x=(b+h)^{\lambda }})+O(\frac{\log T}{T})
\notag \\
&=&\frac{1}{h_{0}^{\lambda -1}}\int_{0}^{h_{0}}\cdots
\int_{0}^{h_{0}}dh_{1}\cdots dh_{\lambda -1}\frac{1}{2\pi i}%
\int_{(a+h)^{\lambda }}^{(b+h)^{\lambda }}f(x)dx\int_{\sigma
_{0}-iT}^{\sigma _{0}+iT}x^{s-1}\varsigma ^{\lambda }(s)ds+O(\frac{\log T}{T}%
)  \notag \\
&=&\frac{1}{h_{0}^{\lambda -1}}\int_{0}^{h_{0}}\cdots
\int_{0}^{h_{0}}dh_{1}\cdots dh_{\lambda -1}\frac{1}{2\pi i}%
\int_{(a+h)^{\lambda }}^{(b+h)^{\lambda }}f(x)dx\int_{1-\sigma
_{0}-iT}^{1-\sigma _{0}+iT}x^{s-1}\varsigma ^{\lambda }(s)ds  \notag \\
&&+S_{\lambda l}+\delta _{+}+\delta _{-}+O(\varepsilon _{1})
\end{eqnarray}%
holds for any $\varepsilon _{1}>0$, as $T$ large, where $h=h_{1}+\cdots
+h_{\lambda -1},\ \lambda \geq 2,\ S_{\lambda l}$ is (2.8) and
\begin{equation*}
\delta _{\pm }=\mp \frac{1}{2\pi ih_{0}^{\lambda -1}}\int_{1-\sigma _{0}\pm
iT}^{\sigma _{0}\pm iT}\varsigma ^{\lambda }(s)I_{\delta }(s)ds,
\end{equation*}

\begin{equation}
I_{\delta }(s)=\int_{0}^{h_{0}}\cdots \int_{0}^{h_{0}}dh_{1}\cdots
dh_{\lambda -1}\int_{a^{\lambda }}^{b^{\lambda }}f(x+h)(x+h)^{s-1}dx.
\end{equation}%
Using the integrations by parts and (2.5), we see that $I_{\delta }\ll \frac{%
1}{T^{\lambda }}$. It is known that (for example, 2.1.8, 2.1.3 in \cite{[1]}%
)
\begin{equation}
\varsigma (s)=\frac{\varsigma (1-s)}{\chi (1-s)},\ \chi (s)\chi (1-s)=1,
\end{equation}%
\begin{equation}
\frac{1}{\chi (s)}=2(2\pi )^{-s}\Gamma (s)\cos \frac{\pi s}{2}=(\frac{t}{%
2\pi })^{\sigma -1/2}e^{it\log \frac{t}{2\pi e}}(1+O(\frac{1}{t})),
\end{equation}%
\begin{equation*}
t>1,\ s=\sigma +it,
\end{equation*}%
\begin{equation}
\varsigma (s)\ll |t|^{\frac{1}{6}},\ \sigma \geq 1/2,|t|>1.
\end{equation}%
Hence
\begin{eqnarray}
\delta _{\pm } &\ll &\frac{1}{T^{\lambda }}(\int_{1-\sigma _{0}}^{\frac{1}{2}%
}+\int_{\frac{1}{2}}^{\sigma _{0}})|\varsigma ^{\lambda }(\sigma \pm
iT)|d\sigma  \notag \\
&\ll &\frac{1}{T^{\lambda }}\int_{\frac{1}{2}}^{\sigma _{0}}(1+T^{\lambda
(\sigma -1/2)})T^{\frac{\lambda }{6}}d\sigma \ll T^{\lambda (\frac{1}{6}%
+\sigma _{0}-3/2)}\ll T^{\frac{-\lambda }{5}}
\end{eqnarray}%
for $\sigma _{0}=1.1$. Putting $s=1-s^{\prime }$ in the integral (2.19),
noticing (2.21 ) and (2.24),
\begin{eqnarray}
S_{\lambda } &=&\frac{1}{h_{0}^{\lambda -1}}\int_{0}^{h_{0}}\cdots
\int_{0}^{h_{0}}dh_{1}\cdots dh_{\lambda -1}\frac{1}{2\pi i}%
\int_{(a+h)^{\lambda }}^{(b+h)^{\lambda }}\frac{f(x)dx}{2\pi i}\int_{\sigma
_{0}-iT}^{\sigma _{0}+iT}\frac{x^{-s}\varsigma ^{\lambda }(s)ds}{\chi
^{\lambda }(s)}  \notag \\
&&+S_{\lambda l}+O(\varepsilon _{1})  \notag \\
&=&\sum\limits_{n\leq N}\frac{d_{\lambda }(n)}{h_{0}^{\lambda -1}}%
\int_{0}^{h_{0}}\cdots \int_{0}^{h_{0}}dh_{1}\cdots dh_{\lambda -1}\frac{1}{%
2\pi i}\int_{(a+h)^{\lambda }}^{(b+h)^{\lambda }}\frac{f(x)dx}{2\pi i}%
\int_{\sigma _{0}-iT}^{\sigma _{0}+iT}\frac{(nx)^{-s}ds}{\chi ^{\lambda }(s)}
\notag \\
&&+S_{\lambda \delta }+\delta _{N}+O(\varepsilon _{1}),
\end{eqnarray}%
where
\begin{eqnarray}
\delta _{N} &\ll &\sum\limits_{n>N}d_{\lambda }(n)\int_{0}^{h_{0}}\cdots
\int_{0}^{h_{0}}dh_{1}\cdots dh_{\lambda -1}\int_{(a+h)^{\lambda
}}^{(b+h)^{\lambda }}|f(x)|\int_{-T}^{T}\frac{(n\chi )^{-\sigma _{0}}dt}{%
|\chi ^{\lambda }(\sigma _{0}+it)|}  \notag \\
&\ll &\sum\limits_{n>N}d_{\lambda }(n)n^{-1.1}\int_{-T}^{T}(|t|+1)^{\lambda
(1.1-1/2)}dt\ll \varepsilon _{1})
\end{eqnarray}%
holds for any fixed $T$ as $N$ sufficient large.

Now estimate the integral
\begin{equation}
\delta _{\lambda }(n)=\int_{\sigma _{0}+iT}^{\sigma _{0}+iT_{1}}\frac{%
j_{\lambda }(s)n^{-s}ds}{\chi ^{\lambda }(s)},
\end{equation}%
\begin{equation}
j_{\lambda }(s)=\int_{0}^{h_{0}}\cdots \int_{0}^{h_{0}}dh_{1}\cdots
dh_{\lambda -1}\int_{(a+h)^{\lambda }}^{(b+h)^{\lambda }}f(x)x^{-s}dx
\end{equation}%
for $\sigma _{0}=1.1,\ T_{1}\rightarrow \infty ,\ n\leq N$. Using the
integrations by parts it is easy to obtain that $j_{\lambda }(s)\ll \frac{1}{%
t^{\lambda }}$. Thus, by (2.22),
\begin{equation}
\delta _{\lambda }(n)\leq n^{-1.1}\int_{T}^{T_{1}}t^{\lambda
(1.1-1/2)-\lambda }dt\leq n^{-1.1}\int_{T}^{T_{1}}t^{-0.4\lambda }dt\leq
n^{-1.1}T^{1-0.4\lambda }\leq T^{-0.2}n^{-1.1}
\end{equation}%
for $\lambda \geq 3$. If $\lambda =2$, then
\begin{eqnarray}
j_{2}(s) &=&\int_{0}^{h_{0}}dh\int_{(a+h)^{2}}^{(b+h)^{2}}f(x)x^{-s}dx
\notag \\
&=&\frac{2}{-2s+2}\int_{0}^{h_{0}}dhf((x+h)^{2})(x+h)^{-2s+2}{%
\mbox
{\LARGE \textbar }}_{x=a}^{x=b}  \notag \\
&&-\frac{2}{-2s+2}\int_{0}^{h_{0}}dh\int_{a}^{b}(f((x+h)^{2}))^{\prime
}(x+h)^{-2s+2}dx  \notag \\
&=&\frac{c_{1}}{t^{2}}f((x+h)^{2})(x+h)^{-2s+3}{\mbox {\LARGE \textbar
}}_{x=a}^{x=b}{\mbox {\LARGE \textbar }}_{h=0}^{h=h_{0}}  \notag \\
&&+\frac{c_{2}}{t^{2}}\int_{o}^{h_{0}}dh(f((x+h)^{2}))^{\prime }(x+h)^{-2s+3}%
{\mbox
{\LARGE \textbar }}_{x=a}^{x=b}  \notag \\
&&+\frac{c_{3}}{t^{2}}\int_{o}^{h_{0}}dh(f((x+h)^{2}))^{\prime \prime
}(x+h)^{-2S+3}dx+O(t^{-3}).
\end{eqnarray}%
Consider the integral
\begin{equation}
I_{2n}(T,T_{1})=\int_{T}^{T_{1}}\frac{(n(x+h))^{-it}dt}{t^{2}\chi
^{2}(1.1+it)}=\int_{T}^{T_{1}}t^{-0.8}e^{2it\log \frac{t}{eM_{n}}%
}dt+O(T^{-0.8}),
\end{equation}%
where the last equation is given by (2.22), $\sigma _{0}=1.1$, and
\begin{equation}
M_{n}=2\pi \sqrt{n(x+h)}.
\end{equation}%
If \ $M_{n}>T-\sqrt{T}$, then

\begin{eqnarray*}
I_{2n}(T,T_{1}) &=&(\int\limits_{|t-M_{n}|\leq \sqrt{M_{n}},T\leq t\leq
T_{1}}+\int\limits_{|t-M_{n}|>\sqrt{M_{n}},T\leq t\leq
T_{1}})t^{-0.8}e^{2it\log \frac{t}{eM_{n}}}dt+O(T^{-0.8}) \\
&\ll &M_{n}^{-0.8+1/2}+\frac{T_{1}^{-0.8}}{|\log \frac{T_{1}}{M_{n}}|}%
+\delta _{T}(M_{n})+\frac{M_{n}^{-0.8}}{|\log \frac{M_{n}+\sqrt{M_{n}}}{M_{n}%
}|}+\frac{M_{n}^{-0.8}}{|\log \frac{M_{n}-\sqrt{M_{n}}}{M_{n}}|} \\
&&+\int\limits_{|t-M_{n}|>M_{n},T\leq t\leq T_{1}}(\frac{t^{-0.8-1}}{|\log
\frac{t}{M_{n}}|}+\frac{t^{-0.8-1}}{|\log \frac{t}{M_{n}}|^{2}})dt.
\end{eqnarray*}%
For any fixed $M_{n}$, as $T_{1}\rightarrow \infty $, the second term is $%
\ll T_{1}^{-0.8}\ll T^{-0.8}$. If $M_{n}\leq \sqrt{T}+T$, then $\delta
_{T}(M_{n})$ is the first term. If $M_{n}>\sqrt{T}+T$, then

\begin{equation*}
\delta _{T}(M_{n})\ll \frac{T^{-0.8}}{|\log \frac{T}{M_{n}}|}\ll
T^{-0.8+1/2}.
\end{equation*}%
Thus,

\begin{eqnarray*}
I_{2n}(T,T_{1}) &\ll &T^{-0.8}+M_{n}^{-0.8+1/2}+\int_{T}^{T_{1}}t^{-0.8-1}dt
\\
&&+\int_{\sqrt{M_{n}}/2\leq |t-M_{n}|\leq 2M_{n}}(\frac{M_{n}}{|t-M_{n}|}+%
\frac{M_{n}^{2}}{|t-M_{n}|^{2}})dt \\
&\ll &M_{n}^{-0.8+1/2}+T^{-0.8}\ll T^{-0.3},
\end{eqnarray*}%
for $M_{n}>T-\sqrt{T}$. If $M_{n}\leq T-\sqrt{T}$, then it is easy to show $%
I_{2n}(T,T_{1})\ll T^{-0.3}$. It follows from (2.27), (2.30) and the
definition of $I_{2n}(T,T_{1})$ that $\delta _{2}(n)\ll n^{-1.1}T^{-0.3}$.
This and (2.29) give
\begin{equation}
\delta _{\lambda }(n)\ll n^{-1.1}T^{-0.2}\ll \varepsilon _{1}n^{-1.1},\
\lambda \geq 2.
\end{equation}%
In the same way,
\begin{equation}
\int_{\sigma _{0}-iT_{1}}^{\sigma _{0}-iT}\frac{j_{\lambda }(s)n^{-S}dS}{%
\chi ^{\lambda }(s)}\ll n^{-1.1}T^{-0.2}\ll \varepsilon _{1}n^{-1.1},\
\lambda \geq 2,\ \sigma _{0}=1.1.
\end{equation}%
Therefore
\begin{equation}
\lim\limits_{T_{1}\rightarrow \infty }\sum\limits_{n\leq N}\frac{d_{\lambda
}(n)}{h_{0}^{\lambda -1}}\int_{0}^{h_{0}}\cdots \int_{0}^{h_{0}}dh_{1}\cdots
dh_{\lambda -1}\int_{(a+h)^{\lambda }}^{(b+h)^{\lambda }}(\int_{\sigma
_{0}+iT}^{\sigma _{0}+iT_{1}}+\int_{\sigma _{0}-iT_{1}}^{\sigma _{0}-iT})%
\frac{(nx)^{-s}ds}{\chi ^{\lambda }(s)}\ll \varepsilon _{1}.
\end{equation}%
By (2.25), (2.26) and (2.35),
\begin{eqnarray}
S_{\lambda } &=&\sum\limits_{n\leq N}\frac{d_{\lambda }(n)}{h_{0}^{\lambda
-1}}\lim\limits_{T\rightarrow \infty }\int_{0}^{h_{0}}\cdots
\int_{0}^{h_{0}}dh_{1}\cdots dh_{\lambda -1}\frac{1}{2\pi i}%
\int_{(a+h)^{\lambda }}^{(b+h)^{\lambda }}\frac{f(x)dx}{2\pi i}\int_{\sigma
_{0}-iT}^{\sigma _{0}+iT}\frac{(nx)^{-s}ds}{\chi ^{\lambda }(s)}  \notag \\
&&+S_{\lambda l}+O(\varepsilon _{1}).
\end{eqnarray}

If $s=-1+it,|t|>1$, then by (2.22),
\begin{equation}
|\frac{1}{\chi (s)}|\ll |t|^{\frac{-3}{2}}.
\end{equation}%
Let $c^{\pm }$ be the broken line from $\sigma _{0}\pm iT$ to $-1\pm iT$ and
again to $-1\pm i\infty $, it is easy to know that
\begin{equation}
\int_{0}^{h_{0}}\cdots \int_{0}^{h_{0}}dh_{1}\cdots dh_{\lambda
-1}\int_{(a+h)^{\lambda }}^{(b+h)^{\lambda
}}f(x)dx(\int_{c^{+}}-\int_{c^{-}})\frac{(nx)^{-s}ds}{\chi ^{\lambda }(s)}%
\rightarrow 0(T\rightarrow \infty )
\end{equation}%
uniformly for $n\leq N$. Therefore, by (2.36), (2.38),
\begin{equation}
S_{\lambda }=\sum\limits_{n\leq N}\frac{d_{\lambda }(n)}{h_{0}^{\lambda -1}}%
\int_{0}^{h_{0}}\cdots \int_{0}^{h_{0}}dh_{1}\cdots dh_{\lambda
-1}\int_{(a+h)^{\lambda }}^{(b+h)^{\lambda }}f(x)J_{\lambda
n}(x)dx+S_{\lambda l}+O(\varepsilon _{1})
\end{equation}%
holds for any $\varepsilon _{1}>0$ as $N\geq N(\varepsilon _{1})$, where $%
S_{\lambda l}$ is (2.8) and
\begin{equation}
J_{\lambda n}(x)=\frac{1}{2\pi i}\int_{c}\frac{(nx)^{-s}ds}{\chi ^{\lambda
}(s)},
\end{equation}%
the path $c$ is an infinite broken line joining the points $-1-i\infty
,-1-iT,\alpha _{\lambda }+5/4-iT,\alpha _{\lambda }+5/4+iT,-1+iT,-1+i\infty $%
.

It is know that (for example, (24), P.360 in \cite{[3]})
\begin{equation}
\Gamma (s)=\sqrt{2\pi }e^{(s-1/2)\log s-s+v(s)},
\end{equation}%
where
\begin{equation}
v(s)=\int_{0}^{\infty }\frac{(\{x\}-1/2)dx}{x+s}=\int_{0}^{\infty }\frac{%
\varphi (x)dx}{(x+s)^{2}}
\end{equation}%
for $s\neq 0,-1,-2\cdots ,$
\begin{equation*}
\varphi (x)=-\int_{0}^{x}(\{y\}-1/2)dy=\sum\limits_{m=1}^{\infty }\frac{%
1-\cos 2m\pi x}{2\pi ^{2}m^{2}}.
\end{equation*}%
So that

\begin{eqnarray}
v(s) &=&\sum\limits_{m=1}^{\infty }\frac{1}{2\pi ^{2}m^{2}}\int_{0}^{\infty }%
\frac{(1-\cos 2m\pi x)dx}{(s+x)^{2}}  \notag \\
&=&\frac{1}{s}\sum\limits_{m=1}^{\infty }\frac{1}{2\pi ^{2}m^{2}}%
-\sum\limits_{m=1}^{\infty }\frac{1}{2\pi ^{2}m^{2}}\int_{0}^{\infty }\frac{%
(\cos 2m\pi x)dx}{(s+x)^{2}}  \notag \\
&=&\frac{c_{1}}{s}-\sum\limits_{m=1}^{\infty }\frac{4}{(2\pi m)^{3}}%
\int_{0}^{\infty }\frac{(\sin 2m\pi x)dx}{(s+x)^{3}}  \notag \\
&=&\frac{c_{1}}{s}+\frac{c_{3}}{s^{3}}+\cdots +\frac{c_{2j_{0}+1}}{%
s^{2j_{0}+1}}+O(\frac{1}{s^{2j_{0}+2}})
\end{eqnarray}%
for large $s\neq -m$, where
\begin{eqnarray}
c_{1} &=&\sum\limits_{m=1}^{\infty }\frac{1}{2\pi ^{2}m^{2}}=\frac{\varsigma
(2)}{2\pi ^{2}}=1/12, \\
c_{3} &=&-\sum\limits_{m=1}^{\infty }\frac{4}{(2\pi m)^{4}}=-\frac{\varsigma
(4)}{4\pi ^{4}}=-1/360, \\
c_{5} &=&1/1260, \\
c_{7} &=&-1/1680,\\
&\cdots&\notag
\end{eqnarray}%

By (2.41) and the first equality of (2.22),
\begin{equation*}
\frac{1}{\chi ^{\lambda }(\frac{s}{\lambda }+1/2-\frac{1}{2\lambda })}%
=2^{\lambda }(2\pi )^{-s-\frac{\lambda -1}{2}}\times \ \ \ \ \ \ \ \ \ \ \ \
\ \ \ \ \ \ \ \ \ \ \ \ \ \ \ \ \ \ \ \ \ \ \ \ \ \ \ \ \ \ \ \ \ \ \
\end{equation*}%
\begin{equation*}
\ \ \ \ \ \times \cos ^{\lambda }(\frac{\pi }{2}(\frac{s}{\lambda }+1/2-%
\frac{1}{2\lambda }))(\sqrt{2\pi })^{\lambda }e^{(s-1/2)\log (\frac{s}{%
\lambda }+1/2-\frac{1}{2\lambda })-\lambda (\frac{s}{\lambda }+1/2-\frac{1}{%
2\lambda })+\lambda v(\frac{s}{\lambda }+1/2-\frac{1}{2\lambda })}
\end{equation*}%
\begin{equation*}
\ \ \ =(2\pi )^{-s}(2\cos \frac{\pi }{2}(\frac{s}{\lambda }+1/2-\frac{1}{%
2\lambda }))^{\lambda }\sqrt{2\pi }e^{(s-1/2)\log s-s+v(s)-(s-1/2)\log
\lambda +\mu _{\lambda }(s)}
\end{equation*}%
\begin{equation}
=\sqrt{\lambda }(2\pi \lambda )^{-s}\Gamma (s)e^{\mu _{\lambda }(s)}(2\cos
\frac{\pi }{2}(\frac{s}{\lambda }+1/2-\frac{1}{2\lambda }))^{\lambda },\ \ \
\ \ \ \ \ \ \ \ \ \ \ \ \ \ \ \ \
\end{equation}%
where
\begin{eqnarray}
\mu _{\lambda }(s) &=&(s-1/2)\log (1+\frac{\lambda -1}{2s})-\frac{\lambda -1%
}{2}  \notag \\
&&+\lambda v(\frac{s}{\lambda }+1/2-\frac{1}{2\lambda })-v(s)  \notag \\
&=&\frac{1-\lambda ^{2}}{24s}+\frac{b_{2}^{\lambda }}{s^{2}}+\cdots
\end{eqnarray}%
and%
\begin{eqnarray}
e^{\mu _{\lambda }(s)} &=&1+\frac{1-\lambda ^{2}}{24s}+\frac{\beta
_{2}^{(\lambda )}}{s^{2}}+\cdots  \notag \\
&=&\sum\limits_{\alpha =0}^{\alpha _{\lambda }+1}\frac{\gamma _{\alpha
}^{(\lambda )}}{(s-1)(s-2)\cdots (s-\alpha )}+R_{\lambda }(s).
\end{eqnarray}%
The first term of the above sum is 1 and
\begin{equation}
R_{\lambda }(s)=O((\frac{1}{s-(\alpha _{\lambda }+1)})^{\alpha _{\lambda
}+2}).
\end{equation}%
Moreover,
\begin{equation}
(2\cos \frac{\pi }{2}(\frac{s}{\lambda }+1/2-\frac{1}{2\lambda }))^{\lambda
}=2\cos (\frac{\pi s}{2}+\frac{(\lambda -1)\pi }{4})+\rho _{\lambda }(s),
\end{equation}%
where
\begin{equation}
\rho _{\lambda }(s)=c_{\lambda -2}\cos \frac{\pi }{2}(\frac{(\lambda -2)s}{%
\lambda }+\theta _{\lambda -2})+\cdots +\rho _{1\lambda }(s),
\end{equation}

\begin{equation*}
\rho _{1\lambda }(s)=\left\{
\begin{array}{ll}
c_{10\lambda }, & \lambda =2\lambda _{1} \\
c_{11\lambda }\cos (\frac{\pi s}{2\lambda }+\theta _{1\lambda }), & \lambda
=2\lambda _{1}+1%
\end{array}%
\right.
\end{equation*}%
Therefore, putting $s=\frac{s^{\prime }}{\lambda }+1/2-\frac{1}{2\lambda }$
in the integral (2.40) and moving the integral line return to the original $%
c $, we have
\begin{eqnarray}
J_{n\lambda }(x) &=&2\lambda ^{-1/2}(nx)^{-\frac{1}{2}+\frac{1}{2\lambda }%
}\sum\limits_{\alpha =0}^{\alpha _{\lambda }+1}\frac{\gamma _{\alpha
}^{(\lambda )}}{2\pi i}\int_{c}\frac{(2\pi \lambda (nx)^{\frac{1}{\lambda }%
})^{-s}\Gamma (s)}{(s-1)(s-2)\cdots (s-\alpha )}\times  \notag \\
&&\times \cos (\frac{\pi s}{2}+\frac{(\lambda -1)}{4})ds+J_{n\lambda \delta
}(x),
\end{eqnarray}%
where
\begin{eqnarray*}
J_{n\lambda \delta }(x) &=&\frac{\lambda ^{-1/2}(nx)^{-\frac{1}{2}+\frac{1}{%
2\lambda }}}{2\pi i}\int_{c}\Gamma (s)(2\pi \lambda (nx)^{\frac{1}{\lambda }%
})^{-s}\\
&&\times(e^{\mu _{\lambda }(s)}\rho _{\lambda }(s)
+2R_{\lambda }(s)\cos (\frac{\pi s}{2}+\frac{(\lambda -1)}{4}))ds \\
&=&\frac{\lambda ^{-1/2}(nx)^{-\frac{1}{2}+\frac{1}{2\lambda }}}{2\pi i}%
\int_{\alpha _{\lambda }+5/4-i\infty }^{\alpha _{\lambda }+5/4+i\infty } \\
&\ll &(nx)^{-\frac{1}{2}+\frac{1}{2\lambda }-\frac{\alpha _{\lambda }+5/4}{%
\lambda }}\int_{-\infty }^{\infty }|\Gamma (\sigma _{1}+it)|\\
&&\times(|\rho _{\lambda }(\sigma _{1}+it)|+|R_{\lambda }(\sigma
_{1}+it)
\cos (\frac{\pi }{2}(\sigma _{1}+it)+\frac{\lambda -1}{4})|)dt, \\
\sigma _{1} &=&\alpha _{\lambda }+5/4.
\end{eqnarray*}%
Since%
\begin{equation*}
|\Gamma (\sigma _{1}+it)\cos (\frac{\pi }{2}(\sigma _{1}+it)(1-\nu
_{0})+\theta )|\ \ \ \ \ \ \ \ \ \ \ \ \ \ \ \ \ \ \ \ \ \ \
\end{equation*}

\begin{eqnarray*}
\ \ \ \ \ \ \ \ \ \ &\ll &\left\{
\begin{array}{ll}
|t|^{\sigma _{1}-1/2}, & \nu _{0}=0 \\
|t|^{\sigma _{1}-1/2}e^{-\frac{\pi |t|\nu _{0}}{2}}, & 0<\nu _{0}\leq 1%
\end{array}%
\right. \\
&\ll &\left\{
\begin{array}{ll}
|t|^{\sigma _{1}-1/2}, & \nu _{0}=0 \\
e^{-|t|\nu _{0}}, & 0<\nu _{0}\leq 1%
\end{array}%
\right.
\end{eqnarray*}%
for $|t|\geq 1$, then by (2.51) and (2.53),
\begin{eqnarray}
J_{n\lambda \delta }(x) &\ll &(nx)^{-1/2+\frac{1}{2\lambda }-\frac{\alpha
_{\lambda }+5/4}{\lambda }}\int_{-\infty }^{\infty }((|t|+1)^{-5/4}+e^{-%
\frac{|t|}{\lambda }})dt  \notag \\
&\ll &(nx)^{-1/2+\frac{1}{2\lambda }-\frac{\alpha _{\lambda }+5/4}{\lambda }%
}.
\end{eqnarray}%
Finally, by (2.54) and (2.55), using $\frac{\Gamma (s)}{(s-1)(s-2)\cdots
(s-\alpha )}=\Gamma (s-\alpha )$, we obtain that
\begin{eqnarray}
J_{n\lambda }(x) &=&2\lambda ^{-1/2}(nx)^{-1/2+\frac{1}{2\lambda }%
}\sum\limits_{\alpha =0}^{\alpha _{\lambda }+1}\frac{\gamma _{\alpha
}^{(\lambda )}(nx)^{\frac{-\alpha }{\lambda }}}{2\pi i(2\pi \lambda
)^{\alpha }}\int_{c}(2\pi \lambda (nx)^{\frac{1}{\lambda }})^{-s}\Gamma
(s)  \notag \\
&&\times \cos (\frac{\pi s}{2}+\frac{(\lambda -1)\pi }{4}+\frac{\pi \alpha }{%
2})ds+O((nx)^{-1/2+\frac{1}{2\lambda }-\frac{\alpha _{\lambda }+5/4}{\lambda
}})
\end{eqnarray}%
It is known that Mellin's formula:
\begin{equation}
\frac{1}{2\pi i}\int_{c}x^{-s}\Gamma (s)\cos (\frac{\pi s}{2}+\theta
)ds=\cos (x+\theta ),
\end{equation}%
hence%
\begin{eqnarray}
J_{n\lambda }(x) &=&2\lambda ^{-1/2}(nx)^{-1/2+\frac{1}{2\lambda }%
}\sum\limits_{\alpha =0}^{\alpha _{\lambda }+1}\frac{\gamma _{\alpha
}^{(\lambda )}(nx)^{\frac{-\alpha }{\lambda }}}{(2\pi \lambda )^{\alpha }}%
  \notag \\
&&\times \cos (2\pi \lambda (nx)^{\frac{1}{\lambda }}+\frac{(\lambda -1)\pi
}{4}+\frac{\pi \alpha }{2})+O((nx)^{-1/2+\frac{1}{2\lambda }-\frac{\alpha
_{\lambda }+5/4}{\lambda }})  \notag \\
&=&2\lambda ^{-1/2}(nx)^{-1/2+\frac{1}{2\lambda }}\sum\limits_{\alpha
=0}^{\alpha _{\lambda }}\frac{\gamma _{\alpha }^{(\lambda )}(nx)^{\frac{%
-\alpha }{\lambda }}}{(2\pi \lambda )^{\alpha }}  \notag \\
&&\times \cos (2\pi \lambda (nx)^{\frac{1}{\lambda }}+\frac{(\lambda -1)\pi
}{4}+\frac{\pi \alpha }{2})+O((nx)^{-1/2+\frac{1}{2\lambda }-\frac{\alpha
_{\lambda }+1}{\lambda }}).
\end{eqnarray}

By (2.39 ) and (2.58) we obtain (2.9). The proof of Lemma 2.2 is complete.

\ \

Lemma 2.3 is the case $\lambda =2$ of Lemma 2.2. To evaluate concretely the
first seventh coefficients $\gamma _{0},\cdots ,\gamma _{7}$ in (2.13), it
is sufficient to take $\lambda =2$ in (2.49). We have
\begin{eqnarray*}
\mu (s) &=&\mu _{2}(s)=(s-1/2)\log (1+\frac{1}{2s})-1/2+2v(\frac{s}{2}%
+1/4)-v(s) \\
&=&-\frac{1}{8s}-\frac{1}{16s^{2}}+\frac{1}{192s^{3}}+\frac{5}{128s^{4}}-%
\frac{1}{640s^{5}}-\frac{61}{768s^{6}}+\cdots
\end{eqnarray*}%
and the calculation gives
\begin{equation*}
e^{\mu _{2}(s)}=1+\sum\limits_{\alpha =1}^{\alpha _{0}}\frac{\gamma _{\alpha
}}{(s-1)(s-2)\cdots (s-\alpha )}+R_{2}(s),
\end{equation*}%
where $\gamma _{1},\cdots ,\gamma _{6}$ are (2.13) , $R_{2}(s)$ satisfies
the corresponding (2.51). Proof of Lemma 2.4 is all the same.

It is known that (for example,(3.17) in \cite{[3]})
\begin{equation}
\Delta (x)=\frac{x^{1/4}}{\pi \sqrt{2}}\sum\limits_{n\leq N}d(n)n^{-3/4}\cos
(4\pi \sqrt{nx}-\pi /4)+O(x^{1/2+\varepsilon }N^{-1/2})+O(x^{\varepsilon })
\end{equation}%
for any $\varepsilon >0$. When $x$ is not an integer, (see (15.24), \cite%
{[3]}),%
\begin{eqnarray}
\Delta (x) &=&\frac{x^{1/4}}{\pi \sqrt{2}}\sum\limits_{n=1}^{\infty
}d(n)n^{-3/4}(\cos (4\pi \sqrt{nx}-\pi /4)  \notag \\
&&+\frac{3}{8(4\pi )}(nx)^{-1/2}\cos (4\pi \sqrt{nx}+\pi /4))+O(x^{-3/4}).
\end{eqnarray}

\begin{lem}
As $x\geq 0.9,\  0<h_{0}\ll x$, $N\geq N(\varepsilon_{1})$,
\begin{equation}
\frac{1}{h_{0}}\int_{0}^{h_{0}}\Delta((x+h)^{2})dh=\frac{1}{h_{0}}
\sum\limits_{n\leq N}d(n)\int_{0}^{h_{0}}\phi(n,(x+h)^{2})dh+O(\varepsilon_{1})
\end{equation}
\noindent holds for any $\varepsilon_{1}>0$, where
$$
\phi(n,x)=\frac{x^{1/4}n^{-3/4}}{\sqrt{2}\pi}\sum\limits_{\alpha=0}^{\alpha_{0}}
\frac{\beta_{\alpha}\cos(4\pi\sqrt{nx}-\pi/4+\alpha\pi/2)}{(4\pi\sqrt{nx})^{\alpha}}
$$
\begin{equation}
+O(x^{1/4}n^{-3/4}(nx)^{-\frac{\alpha_{0}+1}{2}}) \ \ \ \ \ \ \ \ \
\ \
\end{equation}
with  $\beta_{0}=1$, $\beta_{1}=3/8$, $\cdots $,
$\beta_{\alpha_{0}}=\cdots$ for any integer number $\alpha_{0}\geq
0$.
\end{lem}

\begin{proof} In fact , as $\alpha_{0}=0,1,$ (2.60) leads to (2.61); for the general $\alpha_{0}\geq0$, we have
$$
\Delta((x+h)^{2})=\frac{1}{2\pi
i}\int_{C}\frac{(x+h)^{2s}\varsigma^{2}(s)}{s}ds,
$$
\noindent where the path $C$ is from $1.1-i\infty$ to $1.1-i$ to
$1/2-i$ to $1/2+i$ to $1.1+i$ and to $1.1+i\infty$. The proof is the
same as Lemma 2.3.
\end{proof}

\ \

A corollary of Lemma 2.5 is:

\begin{lem}
 $$
 \int_{h_{1}}^{h_{2}}f(h)\Delta((x+h)^{2})dh \ \ \ \ \ \ \ \ \ \ \ \ \ \ \ \ \ \ \ \ \ \ \ \ \ \ \ \ \ \ \ \ \ \ \ \ \ \ \ \ \ \ \ \ \
 $$
\begin{equation}
=\sum\limits_{n\leq N}d(n)\int_{h_{1}}^{h_{2}}f(h)\phi(n,(x+h)^{2})dh+O(\varepsilon_{1})
\end{equation}
holds for any $\varepsilon_{1}>0$ as $N\geq N(\varepsilon_{1})$,
where $ f(x)\in C^{2}[h_{1}$, $h_{2}]$, $x+h_{1}\geq2$,
$h_{1}<h_{2}\ll x$, $\phi(n,x)$ is  (2.62).
\end{lem}

\ \

In the same way, like the proof of Lemma 2.2, for $0.9<a\leq x\leq b<\infty $%
,
\begin{equation*}
\frac{1}{h_{0}^{\lambda -1}}\int_{0}^{h_{0}}\cdots
\int_{0}^{h_{0}}dh_{1}\cdots dh_{\lambda -1}\Delta _{\lambda
}((x+h)^{\lambda })\ \ \ \ \ \ \ \ \ \ \ \ \ \ \ \ \ \ \ \ \ \ \ \ \ \ \ \ \
\ \ \ \ \ \
\end{equation*}%
\begin{equation*}
\ \ \ \ \ \ \ \ \ \ \ \ \ \ \ =\frac{1}{h_{0}^{\lambda -1}}%
\sum\limits_{n\leq N}d_{\lambda }(n)\int_{0}^{h_{0}}\cdots
\int_{0}^{h_{0}}dh_{1}\cdots dh_{\lambda -1}\phi _{\lambda
}(n,(x+h)^{\lambda })
\end{equation*}%
holds for any $\varepsilon _{1}>0$ as $N\geq N(\varepsilon _{1})$, where%
\begin{equation*}
h=h_{1}+\cdots +h_{\lambda -1},\ 0<h_{0}<a/2,
\end{equation*}%
\begin{eqnarray}
\phi _{\lambda }(n,x) &=&\frac{x^{\frac{1}{2}-\frac{1}{2\lambda }}n^{-\frac{1%
}{2}-\frac{1}{2\lambda }}}{\sqrt{\lambda }\pi }\sum\limits_{\alpha
=0}^{\alpha _{\lambda }}\frac{\beta _{\alpha }^{(\lambda )}\cos (2\pi
\lambda (nx)^{\frac{1}{\lambda }}+(\lambda -3)\pi /4+\alpha \pi /2)}{%
(2\lambda \pi (nx)^{1/\lambda })^{\alpha }}  \notag \\
&&+O(x^{\frac{1}{2}-\frac{1}{2\lambda }}n^{-\frac{1}{2}-\frac{1}{2\lambda }%
}(nx)^{-\frac{\alpha _{\lambda }+1}{\lambda }}),
\end{eqnarray}%
\begin{eqnarray*}
\beta _{0}^{(\lambda )} &=&1, \\
\beta _{1}^{(\lambda )} &=&\frac{\lambda -1}{2}+\frac{1-\lambda ^{2}}{24}, \\
&\cdots &
\end{eqnarray*}%
for any positive integer number $\alpha _{\lambda }$.

\ \

\section{SOME LEMMAS (II)}

\begin{lem} {\rm (see(1.17), \cite{[2]})}. For $n\geq3$,
\begin{equation}d(n)\leq {\rm exp}(O(\frac{\log n}{\log\log n}))=n^{\varepsilon(n)},\  \varepsilon(n)=O(\frac{1}{\log\log n})\end{equation}
\end{lem}

The following Lemma 3.2 is a generalization of Theorem B.1 in \cite{[3]},
also the proof belong to it.

\begin{lem}Let
\begin{equation}\psi(V,b,m_{0},x)=\frac{1}{\sqrt{V}}\sum\limits_{j=-\infty}^{\infty}e^{-\frac{\pi(j+b)^{2}}{V}}e(jx)(j+b)^{m_{0}}\end{equation}
\noindent where ${\rm Re}V>0$, $b$ is real, $x=x_{1}+ix_{2}$,
$m_{0}$ is a negative integer, then
\begin{equation}\psi(V,b,m_{0},x)=\sum\limits_{j=-\infty}^{\infty}e^{-\pi V(j+x)^{2}}e(-b(x+j))p_{m_{0}}(x,j)\end{equation}
\noindent where
\begin{equation*}
p_{m_{0}}(x,j)=\int_{-\infty}^{\infty}e^{-\pi\theta^{2}}(Vi(x+j)+\sqrt{V}\theta)^{m_{0}}d\theta
\ \ \ \ \ \ \ \ \ \ \ \ \ \
\end{equation*}
\begin{equation}
\ \ \ \ \ \ \ \ \ =\sum\limits_{\nu=0}^{[\frac{m_{0}}{2}]}\left(
\begin{array}{clr}m_{0} \\2\nu\end{array}\right)\frac{(2\nu-1)!!V^{\nu}}{(2\pi)^{\nu}}(Vi(x+j))^{m_{0}-2\nu},
\end{equation}
$$(-1)!!=1. \ \ \ \ \ \ \ \ \ \ \ \ \ \ \ \ \ \ \ \ \ \ \ \ \ \ \ \ \ \ \ \ \ \ \ \ \ \ \ \ \ \ \ \ \ \ \ \ \ \ $$
\end{lem}

\begin{proof} Let
$$I_{N}=\frac{1}{\sqrt{V}}\int_{C_{N}}\frac{e^{-\frac{\pi(z+b)^{2}}{V}}e(zx)(z+b)^{m_{0}}dz}{e(z)-1}$$
\noindent where $C_{N}$ is the rectangle with vertices at $N+1/2\pm
i,\ -(N+1/2)\pm i$, $N$ is a positive integer, then
$$I_{N}=\frac{1}{\sqrt{V}}\sum\limits_{j=-N}^{N}e^{-\frac{\pi(j+b)^{2}}{V}}e(jx)(j+b)^{m_{0}}\rightarrow\psi(V,b,m_{0},x)\  (N\rightarrow\infty)$$

On the other hand, since
$$\int_{\pm(N+1/2)-i}^{\pm(N+1/2)+i}\frac{e^{-\pi(z+b)^{2}}e(zx)(z+b)^{m_{0}}dz}{e(z)-1}\ll\int_{-1}^{1}\frac{e^{-\varepsilon_{0}N}dy}{1+e^{-2\pi
y}}\rightarrow0\  (N\rightarrow\infty),\ \varepsilon_{0}>0,$$ then
\begin{eqnarray*}
\psi&=&\frac{1}{\sqrt{V}}(\int_{-\infty-i}^{\infty-i}+\int_{-\infty+i}^{\infty+i})\frac{e^{-\frac{\pi(z+b)^{2}}{V}}e(zx)(z+b)^{m_{0}}dz}{e(z)-1}\\
&=&\frac{1}{\sqrt{V}}\sum\limits_{j=-\infty}^{-1}\int_{-\infty-i}^{\infty-i}e^{-\frac{\pi(z+b)^{2}}{V}}e(z(x+j))(z+b)^{m_{0}}dz\\
&&+\frac{1}{\sqrt{V}}\sum\limits_{j=0}^{\infty}\int_{-\infty+i}^{\infty+i}e^{-\frac{\pi(z+b)^{2}}{V}}e(z(x+j))(z+b)^{m_{0}}dz.
\end{eqnarray*}
Putting $z=-b+z'$, and moving the integral line to
$(-\infty,\infty)$, we have
$$
\psi=\sum\limits_{j=-\infty}^{\infty}e(-b(x+j))I_{j},
$$
where
\begin{eqnarray*}
I_{j}&=&\frac{1}{\sqrt{V}}\int_{-\infty}^{\infty}e^{-\frac{\pi
z^{2}}{V}}e(z(x+j))z^{m_{0}}dz\\
&=&\frac{1}{\sqrt{V}}e^{-\pi v
(x+j)^{2}}\int_{-\infty}^{\infty}e^{-\frac{\pi
(z-iV(x+j))^{2}}{V}}z^{m_{0}}dz\\
&=&e^{-\pi V (x+j)^{2}}p_{m_{0}}(x,j),
\end{eqnarray*}
$p_{m_{0}}(x,j)$ is (3.4). The proof is complete.
\end{proof}

\begin{lem} Let $f(x)\in C^{2}(a-\varepsilon,b],\ \varepsilon>0$
\begin{equation}\sigma=\sum\limits_{a\leq n\leq b}f(n),\end{equation}
\begin{equation}M_{\nu}=\max\limits_{a\leq n\leq b}|f^{(\nu)}(x)|,\ \nu=1,2,\end{equation}
then\begin{equation}\sigma=\sum\limits_{-N\leq n\leq
N}\int_{a}^{b}f(x)e(nx)dx+\sigma_{0}(a)+\sigma_{0}(b)+O(\frac{M_{1}+(b-a)M_{2}}{N})\end{equation}
for any natural $N$, where $\sigma_{0}(X)$,  $X=a$  or $b$, such
that the following condition: If $X$ is an integer, then
\begin{equation}
\sigma_{0}(X)= \frac{f(X)}{2};
\end{equation}
if $X$ is not an integer, then
\begin{equation}
\sigma_{0}(X) \ll \frac{1/\{X\}+1/(1-\{X\})}{N}.
\end{equation}
\end{lem}

\begin{proof} We
have\begin{equation}\sigma=\lim_{\varepsilon\rightarrow0_{+}}\sum\limits_{a-\varepsilon\leq
n\leq
b}f(n)=\lim_{\varepsilon\rightarrow0_{+}}\int_{a-\varepsilon}^{b}f(x)d(x+\frac{1}{2}-\{x\})=\int_{a}^{b}f(x)dx+I_{1},\end{equation}
where
\begin{equation}I_{1}=\lim_{\varepsilon\rightarrow0_{+}}\int_{a-\varepsilon}^{b}f(x)d(\frac{1}{2}-\{x\})=
\lim_{\varepsilon\rightarrow0_{+}}f(x)(\frac{1}{2}-\{x\})\mbox
{\LARGE \textbar }_{a-\varepsilon}^{b}+I_{2},\end{equation}
with
\begin{eqnarray*}
I_{2}&=&-\int_{a}^{b}f'(x)(\frac{1}{2}-\{x\})dx\\
&=&-\int_{a}^{b}f'(x)(\int_{0}^{x}(\frac{1}{2}-\{y\})dy)'dx\\
&=&-f'(x)\int_{0}^{x}(\frac{1}{2}-\{y\})dy\mbox {\LARGE \textbar
}_{a}^{b}+\int_{a}^{b}f''(x)(\int_{0}^{x}(\frac{1}{2}-\{y\})dy)dx.
\end{eqnarray*}
Using
\begin{equation}\int_{0}^{x}(\frac{1}{2}-\{y\})dy=\sum\limits_{n=1}^{\infty}\frac{1-\cos2n\pi
x}{2\pi^{2}n^{2}},\end{equation} we have
$$I_{2}=-f'(x)\sum\limits_{n=1}^{N}\frac{1-\cos2n\pi
x}{2\pi^{2}n^{2}}\mbox {\LARGE \textbar
}_{a}^{b}+\int_{a}^{b}f''(x)\sum\limits_{n=1}^{N}\frac{1-\cos2n\pi
x}{2\pi^{2}n^{2}}dx+O(\frac{M_{1}+(b-a)M_{2}}{N})$$
\begin{equation}=2\sum\limits_{n=1}^{N}\int_{a}^{b}f(x)\cos2n\pi xdx-
\sum\limits_{n=1}^{N}\frac{f(x)\sin2n\pi x}{\pi n}\mbox {\LARGE
\textbar }_{a}^{b}+O(\frac{M_{1}+(b-a)M_{2}}{N}).\end{equation}
Since $2\cos 2\pi x =e(x)+e(-x)$, then (3.7) is obtained by (3.10),
( 3.11) and (3.13),
where\begin{equation}\sigma_{0}(a)=-f(a)(\lim_{\varepsilon\rightarrow0_{+}}(\frac{1}{2}-\{a-\varepsilon\})-\sum\limits_{n=1}^{N}\frac{\sin2n\pi
a}{\pi n}),\end{equation}
\begin{equation}\sigma_{0}(b)=f(b)(\frac{1}{2}-\{b\}-\sum\limits_{n=1}^{N}\frac{\sin2n\pi b}{\pi n}).\end{equation}
Let $X=a$ or $b$, clearly, if $X$ is an integer, then
$\sigma_{0}(X)=\frac{1}{2}f(X)$. If $\{X\}=\frac{1}{2}$, then
$\sigma_{0}(X)$ is vanishing. Suppose $0<\{X\}<\frac{1}{2}$. We have
\begin{eqnarray*}
\frac{1}{2}-\{X\}-\sum\limits_{n=1}^{N}\frac{\sin2n\pi X}{\pi
n}&=&\sum\limits_{n=N+1}^{\infty}\frac{\sin2n\pi \{X\}}{\pi n}\\
&=&\lim\limits_{N_{1}\rightarrow\infty}\frac{1}{\pi}\int_{N+1}^{N_{1}}\frac{1}{u}d\sum\limits_{n\leq
u}\sin2n\pi \{X\}\\
&\ll&\frac{1}{N}(\sum\limits_{n\leq N+1}\sin2n\pi
\{X\}+\int_{N+1}^{\infty}\frac{1}{u^{2}}(\sum\limits_{n\leq
u}\sin2n\pi \{X\}))du\\
&\ll&\frac{1}{N}(\frac{1}{\{X\}}+\frac{1}{1-\{X\}}).
\end{eqnarray*}
The case $\frac{1}{2}<\{X\}<1$  is  the same. The proof is complete.
\end{proof}

\begin{lem} Let
\begin{equation}
H=H(X,\alpha,\beta,a,b)=\sum\limits_{a<n\leq
b}d(n)n^{-\alpha}\cos(4\pi\sqrt{nX}+\beta),
\end{equation}
where $0.9<a<b\leq X(\log X)^C$, $C\geq0$; $\alpha, \beta $ are
real, $X$ is large positive number, then
\begin{equation}H\ll(a^{1/4-\alpha}+b^{1/4-\alpha})X^{1/4+\varepsilon},\end{equation}
where
\begin{equation}\varepsilon=\varepsilon(X)=O(\frac{1}{\log\log
X}).\end{equation}
\end{lem}

\begin{proof} Let $h_{0}=\frac{1}{4}$ and
\begin{equation}H'=\frac{1}{h_{0}}\int_{0}^{h_{0}}dh\sum\limits_{\sqrt{a}+h<\sqrt{n}\leq
\sqrt{b}+h}d(n)n^{-\alpha}\cos(4\pi\sqrt{nX}+\beta),\end{equation}

\noindent then from Lemma 3.1
\begin{eqnarray*}
H-H'&=&\frac{1}{h_{0}}\int_{0}^{h_{0}}dh(\sum\limits_{\sqrt{a}<\sqrt{n}\leq
\sqrt{b}}- \sum\limits_{\sqrt{a}+h<\sqrt{n}\leq
\sqrt{b}+h})d(n)n^{-\alpha}\cos(4\pi\sqrt{nX}+\beta)\\
&\ll&\frac{1}{h_{0}}\int_{0}^{h_{0}}(\sum\limits_{|\sqrt{n}-\sqrt{a}|\leq
h}+\sum\limits_{|\sqrt{n}-\sqrt{b}|\leq h})d(n)n^{-\alpha}\\
&\ll&
a^{\frac{1}{2}-\alpha+\varepsilon}+b^{\frac{1}{2}-\alpha+\varepsilon},
\end{eqnarray*}
 \noindent where $\varepsilon=\varepsilon(X)$ is (3.18). So that
\begin{equation}H=\frac{1}{h_{0}}\int_{0}^{h_{0}}dh\sum\limits_{\sqrt{a}+h<\sqrt{n}\leq
\sqrt{b}+h}d(n)n^{-\alpha}\cos(4\pi\sqrt{nX}+\beta)+O(a^{\frac{1}{2}-\alpha+\varepsilon}+b^{\frac{1}{2}-\alpha+\varepsilon})\end{equation}
Taking $f(n)=n^{-\alpha}\cos(4\pi\sqrt{nX}+\beta)$ in Lemma 2.3, we
have
\begin{equation}H=H_{l}+H_{\Delta
N}+O(\varepsilon_{1})+O(a^{\frac{1}{2}-\alpha+\varepsilon}+b^{\frac{1}{2}-\alpha+\varepsilon})\end{equation}
\noindent holds for any $\varepsilon_{1}>0$ as $N$ large, where
\begin{eqnarray*}
 H_{l}&=&\frac{1}{h_{0}}\int_{0}^{h_{0}}dh\int_{(\sqrt{a}+h)^{2}}^{(\sqrt{b}+h)^{2}}u^{-\alpha}(\log
 u+2\gamma)\cos(4\pi\sqrt{uX}+\beta)du\\
 &=&\frac{2}{h_{0}}\int_{0}^{h_{0}}dh\int_{\sqrt{a}+h}^{\sqrt{b}+h}u^{1-2\alpha}(\log
 u^{2}+2\gamma)\cos(4\pi\sqrt{X}u+\beta)du
 \end{eqnarray*}
 \begin{equation}
 \ll(a^{\frac{1}{2}-\alpha+\varepsilon}+b^{\frac{1}{2}-\alpha+\varepsilon})/\sqrt{X},\
 \ \ \ \ \ \ \ \ \ \ \ \ \ \ \ \ \ \ \ \ \ \ \ \ \ \ \ \ \ \ \ \ \ \
 \ \ \
 \end{equation}
$$H_{\Delta N}=\sum\limits_{n\leq N}\frac{d(n)}{h_{0}}\int_{0}^{h_{0}}dh
\int_{(\sqrt{a}+h)^{2}}^{(\sqrt{b}+h)^{2}}u^{-\alpha}\cos(4\pi\sqrt{uX}+\beta)\Theta(nu)du$$
\begin{equation}=2\sqrt{2}\sum\limits_{n\leq N}d(n)n^{-\frac{1}{4}}\sum\limits_{\nu=0}^{1}
\frac{\gamma_{\nu}}{(4\pi\sqrt{n})^{\nu}}j_{n\nu}+\delta_{\Delta
N},\ \ \ \ \ \ \ \ \ \ \ \ \ \end{equation}
\begin{equation}\delta_{\Delta
N}\ll\frac{1}{h_{0}}\sum\limits_{n\leq
N}d(n)n^{-\frac{1}{4}}\int_{0}^{h_{0}}dh\int_{(\sqrt{a}+h)^{2}}^{(\sqrt{b}+h)^{2}}\frac{u^{-\alpha-1}}{n}du\ll
a^{-\alpha}+b^{-\alpha},\end{equation}
$$\ \ \ \ \ \ \ \ \ \ \ \ \ \ \ j_{n\nu}=\frac{1}{h_{0}}\int_{0}^{h_{0}}dh
\int_{\sqrt{a}+h}^{\sqrt{b}+h}u^{\frac{1}{2}-2\alpha-\nu}\cos4\pi(\sqrt{X}u+\beta)
\cos(4\pi\sqrt{n}u+\frac{\pi}{4}+\frac{\pi\nu}{2})du$$
\begin{equation}=j_{n\nu}^{+}+j_{n\nu}^{-}, \ \ \ \ \ \ \ \ \ \ \ \ \ \ \ \ \ \ \ \ \ \ \ \ \ \ \ \ \ \
\ \ \ \ \ \ \ \ \ \ \ \ \ \ \ \ \ \ \ \ \ \ \ \ \ \ \ \ \
\end{equation}
\begin{eqnarray*}
j_{n\nu}^{\pm}&=&\frac{1}{2h_{0}}\int_{0}^{h_{0}}dh\int_{\sqrt{a}+h}^{\sqrt{b}+h}u^{\frac{1}{2}-2\alpha-\nu}
\cos(4\pi(\sqrt{n}\pm\sqrt{X})u+\theta^{\pm})du\\
&=&\frac{1}{2h_{0}}\int_{0}^{h_{0}}\frac{(u+h)^{\frac{1}{2}-2\alpha-\nu}
\sin(4\pi(\sqrt{n}\pm\sqrt{X})(u+h)+\theta^{\pm})}{4\pi(\sqrt{n}\pm\sqrt{X})}\mbox
{\LARGE \textbar }_{\sqrt{a}}^{\sqrt{b}}
\end{eqnarray*}
\begin{equation}-\frac{c}{2h_{0}}\int_{0}^{h_{0}}\frac{dh}{\sqrt{n}\pm\sqrt{X}}
\int_{\sqrt{a}}^{\sqrt{b}}(u+h)^{-\frac{1}{2}-2\alpha-\nu}\sin(4\pi(\sqrt{n}\pm\sqrt{X})(u+h)+\theta^{\pm})du\end{equation}
for $\sqrt{n}\neq\sqrt{X}.$ If $n>2X$, then integrating by parts for
$h$, we have
\begin{equation}j_{n\nu}^{\pm}\ll\frac{a^{\frac{1}{4}-\alpha}+b^{\frac{1}{4}-\alpha}}{(\sqrt{n}\pm\sqrt{X})^{2}}
\ll\frac{a^{\frac{1}{4}-\alpha}+b^{\frac{1}{4}-\alpha}}{n}.\end{equation}
Therefore, it follows from (3.23), (3.24) and (3.27) that
\begin{equation}H_{\Delta N}=2\sqrt{2}\sum\limits_{n\leq
2X}d(n)n^{-\frac{1}{4}}\sum\limits_{\nu=0}^{1}
\frac{\gamma_{\nu}}{(4\pi\sqrt{n})^{\nu}}(j_{n\nu}^{+}+j_{n\nu}^{-})+O(a^{\frac{1}{4}-\alpha}+b^{\frac{1}{4}-\alpha}).\end{equation}
If $|n-X|\leq2$, then by (3.25),
\begin{equation}j_{n\nu}^{\pm}\ll\int_{\sqrt{a}}^{\sqrt{b}+1}u^{\frac{1}{2}-2\alpha}du\ll
a^{\frac{3}{4}-\alpha}+b^{\frac{3}{4}-\alpha}.\end{equation}
If
$|n-X|>2$, then by (3.26),
\begin{equation}j_{n\nu}^{\pm}\ll\frac{1}{|\sqrt{n}-\sqrt{X}|}(a^{\frac{1}{4}-\alpha}+b^{\frac{1}{4}-\alpha}).\end{equation}
Hence, by (3.27), (3.28), (3.29) and Lemma 3.1 we obtain
\begin{equation*}
H_{\Delta N}\ll\sum\limits_{n\leq
2X,|n-X|>2}\frac{d(n)n^{-\frac{1}{4}}}{|\sqrt{n}-\sqrt{X}|}(a^{1/4-\alpha}+b^{1/4-\alpha})+
\sum\limits_{|n-X|\leq2}d(n)n^{-\frac{1}{4}}(a^{3/4-\alpha}+b^{3/4-\alpha})
\end{equation*}
\begin{equation*}
\ll\sum\limits_{n\leq\frac{1}{2}X}d(n)n^{-\frac{1}{4}}\frac{a^{1/4-\alpha}+b^{1/4-\alpha}}{\sqrt{X}}
\ \ \ \ \ \ \ \ \ \ \ \ \ \ \ \ \ \ \ \ \ \ \ \ \ \ \ \ \ \ \ \ \ \
\ \ \ \ \ \ \ \ \ \ \ \ \ \ \ \ \ \ \ \ \ \
\end{equation*}
\begin{equation*}
+(a^{1/4-\alpha}+b^{1/4-\alpha})X^{\frac{1}{4}+\varepsilon(X)}\sum\limits_{X/2<n\leq
2X,|n-X|>2}\frac{1}{|n-X|}+(a^{3/4-\alpha}+b^{3/4-\alpha})X^{-\frac{1}{4}+\varepsilon(X)}
\end{equation*}
\begin{equation}\ll(a^{1/4-\alpha}+b^{1/4-\alpha})X^{\frac{1}{4}+\varepsilon(X)} \ \ \ \ \ \ \ \ \
\ \ \ \ \ \ \ \ \ \ \ \ \ \ \ \ \ \ \ \ \ \ \ \ \ \ \ \ \ \ \ \ \ \ \ \ \ \ \ \ \ \ \ \ \end{equation}
\noindent for $a<b\ll (\log X)^C$. (3.20), (3.21), (3.22) and (3.31)
implicate (3.17). The proof is complete.
\end{proof}

\begin{lem}
 Suppose $a,\ b,\ \alpha,\ \beta,\ X$ be as Lemma
3.4, $$H=\sum\limits_{a<n\leq
b}d(n)n^{-\alpha}f(n)\cos(4\pi\sqrt{nX}+\beta),\ f(u)\ll M,$$
\begin{equation}f'(u)\ll\frac{M}{\sqrt{Xu}},u\in[a,b+1],\end{equation}
\noindent then \begin{equation}H\ll
M(a^{1/4-\alpha}+b^{1/4-\alpha})X^{\frac{1}{4}+\varepsilon(X)}.\end{equation}
\end{lem}
\begin{proof} In fact, by Lemma 3.4,
\begin{eqnarray*}
H&=&\int_{a}^{b}f(u)d\sum\limits_{a<n\leq
u}d(n)n^{-\alpha}f(n)\cos(4\pi\sqrt{nX}+\beta) \\
&\ll&|f(b)||\sum\limits_{a<n\leq
b}d(n)n^{-\alpha}f(n)\cos(4\pi\sqrt{nX}+\beta)|\\
 &&+\int_{a}^{b}|f'(u)||\sum\limits_{a<n\leq
u}d(n)n^{-\alpha}f(n)\cos(4\pi\sqrt{nX}+\beta)|du\\
&\ll&M(a^{1/4-\alpha}+b^{1/4-\alpha})X^{\frac{1}{4}+\varepsilon(X)}(1+\int_{a}^{b}\frac{du}{\sqrt{Xu}})\\
&\ll&M(a^{1/4-\alpha}+b^{1/4-\alpha})X^{\frac{1}{4}+\varepsilon(X)}.
\end{eqnarray*}
The proof is complete.
\end{proof}

\ \

Let
\begin{equation*}
\eta=\eta_{1}+\cdots+\eta_{K},
\end{equation*}
define
\begin{equation}
f(\eta)\mbox {\LARGE \textbar
}_{\eta_{k}}=f(\eta_{1}+\cdots+\eta_{K})\mbox {\LARGE \textbar
}_{\eta_{1}=0}^{\eta_{1}=\nu_{0}}\cdots\mbox {\LARGE \textbar
}_{\eta_{K}=0}^{\eta_{K}=\nu_{0}}
\end{equation}
Also we define
\begin{equation*}
f(\eta_{1})|\mbox {\LARGE \textbar
}_{\eta_{1}=0}^{\eta_{1}=\nu_{0}}|=|f(\nu_{0})|+|f(0)|
\end{equation*}
generally,%
\begin{equation}
f(\eta)|\mbox {\LARGE \textbar
}_{\eta_{K}}|=\sum\limits_{\eta_{1}=0,\nu_{0}}\cdots\sum\limits_{\eta_{K}=0,%
\nu_{0}}|f(\eta)|
\end{equation}
It is obvious that if $f(\eta)\ll M$, then
\begin{equation}
f(\eta)\mbox {\LARGE \textbar }_{\eta_{K}}\ll f(\eta)|%
\mbox {\LARGE
\textbar }_{\eta_{K}}|\ll M2^{K};
\end{equation}
\noindent if $f^{(k)}(\eta)\ll M,\ 0\leq k\leq K$, then

\begin{eqnarray}
f(\eta )\mbox {\LARGE \textbar }_{\eta _{K}} &=&(\int_{0}^{\nu _{0}}\cdots
\int_{0}^{\nu _{0}}f^{(k)}(\eta )d\eta _{1}\cdots d\eta _{K})%
\mbox
{\LARGE \textbar }_{\eta _{k+1}=0}^{\eta _{k+1}=\nu _{0}}\cdots
\mbox
{\LARGE \textbar }_{\eta _{K}=0}^{\eta _{K}=\nu _{0}}  \notag \\
&\ll &M2^{K-k}\int_{0}^{\nu _{0}}\cdots \int_{0}^{\nu _{0}}d\eta _{1}\cdots
d\eta _{K}=M2^{K-k}\nu _{0}^{k}.
\end{eqnarray}

\begin{lem}
 Suppose $f(z)$ is analytic for $|z|\leq\frac{1}{\sqrt{L}}$,
and $|f(z)|\leq M$, then
\begin{equation}f(\eta)\mbox {\LARGE \textbar
}_{\eta_{k}}=f(\eta_{1}+\cdots+\eta_{K})\mbox {\LARGE \textbar
}_{\eta_{1}=0}^{\eta_{1}=\nu_{0}}\cdots\mbox {\LARGE \textbar
}_{\eta_{K}=0}^{\eta_{K}=\nu_{0}}\ll MU^{-99}\end{equation}
\noindent for large U, $L\asymp\log U$, $0<\nu_{0}\leq L^{-2},\
K=[\frac{C_{0}L}{\log L}],\ C_{0}\geq200$.
\end{lem}

\begin{proof}
Clearly,
\begin{equation}0\leq\eta_{1}+\cdots+\eta_{K}\leq\nu_{0}K\ll\frac{1}{L\log
L}.\end{equation}
Moreover,
\begin{eqnarray*}
f(\eta)\mbox {\LARGE \textbar
}_{\eta_{K}}&=&\int_{0}^{\nu_{0}}\cdots\int_{0}^{\nu_{0}}f^{(k)}(\eta)d\eta_{1}\cdots
d\eta_{K}\\
&=&\int_{0}^{\nu_{0}}\cdots\int_{0}^{\nu_{0}}d\eta_{1}\cdots
d\eta_{K}\frac{K!}{2\pi
i}\int_{|z|=\frac{1}{\sqrt{L}}}\frac{f(z)dz}{(z-\eta)^{K+1}}\\
&\ll&
K!\nu_{0}^{K}\int_{|z|=\frac{1}{\sqrt{L}}}\frac{M|dz|}{(\frac{1}{\sqrt{L}}-\frac{1}{L\log
L})^{K+1}}\\
&\ll& K!L^{-2K}M(2\sqrt{L})^{K}\leq M(2KL^{-3/2})^{K}\\
&=&Me^{-K\log\frac{L^{3/2}}{2K}}\\
&\ll& M {\rm exp}(-\frac{200L}{\log L}(\log L^{1/2}+O(\log\log L)))\\
&\ll& MU^{-99},
\end{eqnarray*}
for $L\asymp\log U$.
The proof is complete.
\end{proof}

\ \

\section{SOME LEMMAS \ (III)}

\begin{lem} Let
\begin{equation}
G_{\lambda}(\sqrt{n},\eta)=e^{-\frac{4\pi(\sqrt{n}+\lambda(\sqrt{U}+\eta+\nu))^{2}}{A^{2}}},
\end{equation}
\begin{equation}
W_{\lambda}(\eta)=e(-\lambda(\sqrt{U}+\eta+\nu)^{2}),
\end{equation}
\begin{equation}
 S_{\lambda}(\eta)=W_{\lambda}(\eta)
\sum\limits_{\sqrt{n}\leq\sqrt{U}}d(n)n^{-\frac{1}{4}}G_{\lambda}(\sqrt{n},\eta)
\sum\limits_{\alpha=0}^{2}\frac{\gamma_{\alpha}(-i)^{\alpha}e(-2\sqrt{n}(\sqrt{U}+\eta+\nu)-1/8)}{(4\pi\sqrt{n}(\sqrt{n}+\sqrt{U}+\eta+\nu))^{\alpha}},
\end{equation}
\begin{equation}
S_{\lambda}^{*}(\eta)=W_{\lambda}(\eta)
\sum\limits_{\sqrt{U}<\sqrt{n}
\leq(\lambda_{0}-\lambda)\sqrt{U}}d(n)n^{-\frac{1}{4}}G_{\lambda}(\sqrt{n},\eta)
 e(-2\sqrt{n}(\sqrt{U}+\eta+\nu)-1/8),
\end{equation}
where $\nu=O(L^{-1})$ is real, $A$ is (1.10), $\eta$ is (1.15),
 \begin{equation}1\leq\lambda_{0}\ll L,\  0\leq \lambda\leq\lambda_{0}-1.\end{equation}
Then
\begin{equation}S_{0}^{*}(\eta)=\sum\limits_{\lambda=1}^{\lambda_{0}-1}S_{\lambda}(\eta)+O(U^{-3/4+\varepsilon}),\ \varepsilon=O(\frac{1}{\log
L}).\end{equation} (note that $O_{\eta}$  is (1.28).)
\end{lem}

\begin{proof} Let
\begin{equation}U=[X],\ 0<h\leq
h_{0}=\frac{1}{\sqrt{U}L^{2}},
\end{equation} \noindent
then
\begin{equation}0<(\sqrt{U}+h)^{2}-U<L^{-3/2},
\end{equation}
$$\{\sqrt{U}^{2}\}=\{U\}=0,$$
Hence there doesn't exist any integer in the interval
$(U,(\sqrt{U}+h)^{2}]$, that is
$$\sum\limits_{\sqrt{U}<\sqrt{n}\leq\sqrt{U}+h}=0.$$
In the same way,
$$\sum\limits_{(\lambda_{0}-\lambda)\sqrt{U}<\sqrt{n}\leq(\lambda_{0}-\lambda)\sqrt{U}+h}=0.
$$
Noticing $e(n)=1$, we have
$$
S_{\lambda}^{*}(\eta)=\frac{e(-1/8)W_{\lambda}(\eta)}{h_{0}} \ \ \ \
\ \ \ \ \ \ \ \ \ \ \ \ \ \ \ \ \ \ \ \ \ \ \ \ \ \ \ \ \ \ \ \ \ \
\ \ \ \ \ \ \ \ \ \ \ \ \ \ \ \ \ \ \ \ \ \ \ \ \ \ \ \ \ \ \ \ \ \
\ \ \ \ \ \ \ \
$$
\begin{equation}
\times\int_{0}^{h_{0}}dh
\sum\limits_{\sqrt{U}+h<\sqrt{n}\leq(\lambda_{0}-\lambda)\sqrt{U}+h}d(n)n^{-1/4}G_{\lambda}(\sqrt{n},\eta)e(n-2\sqrt{n}(\sqrt{U}+\eta+\nu)).
\end{equation}
By Lemma 2.3
\begin{eqnarray}
S_{\lambda}^{*}(\eta)&=&\sum\limits_{n\leq
N}\frac{e(-1/8)W_{\lambda}(\eta)d(n)}{h_{0}}\int_{0}^{h_{0}}dh
\int_{(\sqrt{U}+h)^{2}}^{((\lambda_{0}-\lambda)\sqrt{U}+h)^{2}}\notag\\
&&x^{-1/4}G_{\lambda}(\sqrt{n},\eta)\Theta(nx)e(x-2\sqrt{x}(\sqrt{U}+\eta+\nu))dx\notag\\
 &&+\delta_{\lambda}(\eta)+O(\varepsilon_{1}),
\end{eqnarray}
 holds for any $\varepsilon_{1}>0$, as $N\geq
N(\varepsilon_{1})$, where $\Theta(nx)$ is (2.12),

$$
\delta_{\lambda}(\eta)=\frac{e(-1/8)W_{\lambda}(\eta)}{h_{0}}\int_{0}^{h_{0}}dh
\int_{(\sqrt{U}+h)^{2}}^{((\lambda_{0}-\lambda)\sqrt{U}+h)^{2}}x^{-1/4}G_{\lambda}(\sqrt{n},\eta)
\ \ \ \ \ \ \ \ \ \ \ \ \ \ \ \ \ \ \ \ \ \ \ \ \ \ \ \ \ \ \ \ \ \
\ \
$$
$$
\ \ \ \ \ \ \ \ \ \ \ \ \ \ \ \ \ \ \ \ \ \ \ \ \ \ \ \ \ \ \ \ \ \
\ \ \ \ \ \ \ \ \times e(x-2\sqrt{x}(\sqrt{U}+\eta+\nu))(\log
x+2\gamma)dx
$$
$$
=\frac{2e(-1/8)W_{\lambda}(\eta)}{h_{0}}\int_{0}^{h_{0}}dh\int_{\sqrt{U}+h}^{(\lambda_{0}-\lambda)\sqrt{U}+h}x^{1/2}G_{\lambda}(\sqrt{n},\eta)\ \
\ \ \ \ \ \ \ \ \ \ \ \ \ \
\ \ \ \ \ \ \ \ \ \ \ \ \ \ \ \ \ \ \ \ \ \ \ \ \ \ \ \ \ \ \ \ \ \
\ \ \ \ \ \ \ \ \ \ \ \ \ \
$$
\begin{equation}
\ \ \ \ \ \ \ \ \ \ \ \ \ \ \ \ \ \ \ \ \ \ \ \ \ \ \ \ \ \ \ \ \ \
\ \ \ \ \ \ \ \ \ \times e(x^{2}-2\sqrt{x}(\sqrt{U}+\eta+\nu))(\log
x^{2}+2\gamma)dx.
\end{equation}
Moving the integral line of the inner integral we
obtain
\begin{eqnarray*}
\delta_{\lambda}(\eta)&=&\frac{2e(-1/8)W_{\lambda}(\eta)}{h_{0}}\int_{0}^{h_{0}}dh(\int_{\sqrt{U}+h}^{\sqrt{U}+h+z_{0}}
+\int_{\sqrt{U}+h+z_{0}}^{(\lambda_{0}-\lambda)\sqrt{U}+h+z_{0}}+\int_{(\lambda_{0}-\lambda)\sqrt{U}+h+z_{0}}^{(\lambda_{0}-\lambda)\sqrt{U}+h})\\
&&\times z^{\frac{1}{2}}G_{\lambda}(\sqrt{n},\eta)(\log
z^{2}+2\gamma)e(z^{2}-2z(\sqrt{U}+\eta+\nu))dz
\end{eqnarray*}
\begin{equation}
=\frac{2e(-1/8)W_{\lambda}(\eta)}{h_{0}}\int_{0}^{h_{0}}dh(I_{1h}(\eta)+I_{2h}(\eta)
+I_{3h}(\eta)), \ \ \ \ \ \ \ {\rm say} \ \ \ \ \
\end{equation} where
\begin{equation}
z_{0}=Le(\frac{1}{8}).
\end{equation}
Since
\begin{eqnarray*}
G_{\lambda}(z,\eta)&=&e^{-\frac{4\pi(z+\lambda(\sqrt{U}+\eta+\nu))^{2}}{A^{2}}}\\
&=&\int_{-\infty}^{\infty}e^{-\pi\theta^{2}}e(\frac{2\theta(z+\lambda(\sqrt{U}+\eta+\nu))}{A})d\theta
\ \ \ \ \ \ \ \ \ \ \ \ \ \ \ \ \ \ \ \ \ \
\end{eqnarray*}
\begin{equation}\ \ \ \ \ \ \ \ \ \ \ \ \ \ \ \ \ \ \ \ \ \ \ \ \ =\int_{-L}^{L}e^{-\pi\theta^{2}}e(\frac{2\theta\lambda(\sqrt{U}+\nu)}{A})
e(\frac{2(z+\lambda\eta)\theta}{A})d\theta+O(e^{-L^{2}}),\end{equation}
then by (4.12),
\begin{equation}
W_{\lambda}(\eta)I_{1h}(\eta)=\int_{-L}^{L}e^{-\pi\theta^{2}}e(\frac{2\theta\lambda(\sqrt{U}+\nu)}{A})J_{1\theta}(\eta)+O(e^{-L^{2}}),
\end{equation}
where
$$
J_{1\theta}(\eta)=W_{\lambda}(\eta)\int_{\sqrt{U}+h}^{\sqrt{U}+h+z_{0}}f_{\theta}(z)e(z^{2}-2z(\sqrt{U}+\eta+\nu)+\frac{2(z+\lambda\eta)\theta}{A})dz
$$
\begin{equation}
=\int_{0}^{z_{0}}f_{\theta}(\sqrt{U}+h+z)e((\sqrt{U}+h+z)^{2}+2(\sqrt{U}+h+z)(-\sqrt{U}-\nu+\frac{\theta}{A}))\omega(z,\eta)dz,
\end{equation}
\begin{equation}f_{\theta}(z)=z^{1/2}(\log
z^{2}+2\gamma), \ \ \ \ \ \ \ \ \ \ \ \ \ \ \ \ \ \ \ \ \ \ \ \ \ \
\ \ \ \ \ \ \ \ \
\end{equation}
\begin{eqnarray*}
\omega(z,\eta)&=&e(2\eta(-\sqrt{U}-h-z+\frac{\lambda\theta}{A}))W_{\lambda}(\eta)\\
&=&e(2\eta(-\sqrt{U}-h-z+\frac{\lambda\theta}{A})-\lambda(\sqrt{U}+\eta+\nu)^{2})
\end{eqnarray*}
\begin{equation}=e(-\lambda(\sqrt{U}+\nu)^{2})g_{z\theta}(\eta),\ \ \ \ \ \ \ \ \ \ \ \ \ \ \ \ \ \ \ \ \ \  \end{equation}
$$\ \ \ \ \ g_{z\theta}(\eta)=e(-2(\lambda+1)\eta\sqrt{U}+2\eta(-h-z-\nu\lambda+\frac{\lambda\theta}{A})-\lambda\eta^{2})$$
\begin{equation}
=e(2\eta(-h-z+\frac{\lambda\theta}{A})-\lambda\eta^{2}).\ \ \ \ \ \
\ \ \ \ \ \ \ \ \
\end{equation}
The last equality is given by
\begin{equation}
e(-2(\lambda+1)\eta\sqrt{U})=e(\eta\sqrt{U})=1
\end{equation}
where $\eta=\eta_{1}+\cdots+\eta_{K}$, $\eta_{j}=0$  or
$k_{1}/\sqrt{U}$,  $j=1,2\cdots,K$. Thus,
\begin{equation}g_{z\theta}(w)\ll
e^{4\pi|w|(|\frac{\lambda\theta}{A}|+|h|+|z|+\lambda|w|^{2})}\ll
e^{8\pi\sqrt{L}}\end{equation} for $|w|\leq\frac{1}{\sqrt{L}},\
|z|\leq L,\ \lambda\ll L,\ |\theta|\ll L,\ h=(\frac{1}{\sqrt{UL}}),\
A=C_{0}\sqrt{UL}$. By Lemma 3.6, we have
\begin{equation}g_{z\theta}(\eta)|_{\eta_{K}}\ll e^{8\pi\sqrt{L}}U^{-99}\ll U^{-98}\end{equation}
for $\nu_{0}=k_{1}/\sqrt{U}\ll L^{-2}$. It follows from (4.15),
(4.16), (4.17), (4.18), (4.19) and (4.22) that
\begin{equation}W_{\lambda}(\eta)I_{1h}(\eta)|_{\eta}\ll
U^{-90}.\end{equation}
In the same way,
\begin{equation}W_{\lambda}(\eta)I_{3h}(\eta)|_{\eta}\ll
U^{-90}.\end{equation}
By (4.13),
\begin{eqnarray*}
W_{\lambda}(\eta)I_{2h}(\eta)&=&W_{\lambda}(\eta)\int_{\sqrt{U}+h+z_{0}}^{(\lambda_{0}-\lambda)\sqrt{U}+h+z_{0}}
f_{\theta}(z)G_{\lambda}(\sqrt{n},\eta)e(z^{2}-2z(\sqrt{U}+\eta+\nu))dz\\
&\ll&\int_{\sqrt{U}+h}^{(\lambda_{0}-\lambda)\sqrt{U}+h}\sqrt{UL}|e((x+z_{0})^{2}-2(z_{0}+x)(\sqrt{U}+\eta+\nu)))|dx\\
&=&\sqrt{UL}\int_{\sqrt{U}+h}^{(\lambda_{0}-\lambda)\sqrt{U}+h}|e(z_{0}^{2}+2z_{0}(x-(\sqrt{U}+\eta+\nu)))|dx\\
&=&\sqrt{UL}\int_{\sqrt{U}+h}^{(\lambda_{0}-\lambda)\sqrt{U}+h}e^{-2\pi
L^{2}-2\sqrt{2}\pi(x-(\sqrt{U}+\eta+\nu))}
\end{eqnarray*}
\begin{equation}\ll\sqrt{UL}e^{-L^{2}}=O_{\eta}(U^{-90})\ \ \ \ \ \ \ \ \ \ \ \ \ \ \ \ \ \ \ \ \ \ \ \ \ \ \ \ \ \ \ \ \ \ \ \  \end{equation}
for $z_{0}=Le(\frac{1}{z})$. It follows from (4.12), (4.23), (4.24)
and (4.25) that
\begin{equation}\delta_{\lambda}(\eta)=O_{\eta}(U^{-90}).\end{equation}
Taking $\alpha_{0}=2$ in (2.12), we have
\begin{equation}\Theta(nx)=\sqrt{2}(\sqrt{nx})^{-1/2}\sum\limits_{\alpha=0}^{2}
\frac{\gamma_{\alpha}}{(4\pi)^{\alpha}}(\sqrt{nx})^{-\alpha}\cos(4\pi\sqrt{nx}+\frac{\pi}{4}+\frac{\alpha\pi}{2})+O((\sqrt{nx})^{-7/2}).\end{equation}
By (4.26) and (4.27) and putting $x=x_{1}^{2}$ in the integral
(4.10), we have
\begin{equation}S_{\lambda}^{*}(\eta)=\&_{\lambda}^{+}(\eta)+\&_{\lambda}^{-}(\eta)+O(\varepsilon_{1})+O_{\eta}(U^{-\frac{3}{2}}),\end{equation}
where
\begin{equation}\&_{\lambda}^{\pm}(\eta)=\sqrt{2}e(-\frac{1}{8}\pm\frac{1}{8})\sum\limits_{n\leq
N}d(n)n^{-\frac{1}{4}}\sum\limits_{\alpha=0}^{2}\gamma_{\alpha}(\frac{\pm
i}{4\pi\sqrt{n}})^{\alpha}j_{n\alpha}^{\pm}(\eta),\end{equation}
\begin{equation}j_{n\alpha}^{\pm}(\eta)=\frac{W_{\lambda}(\eta)}{h_{0}}
\int_{0}^{h_{0}}dh\int_{\sqrt{U}+h}^{(\lambda_{0}-\lambda)\sqrt{U}+h}
\frac{G_{\lambda}(\sqrt{n},\eta)e(x^{2}+2x(\pm\sqrt{n}-(\sqrt{U}+\eta+\nu)))dx}{x^{\alpha}}.\end{equation}
As $n>n_{0}>2U$, the integration by parts gives
$$
j_{n\alpha}^{\pm}(\eta)=\ \ \ \ \ \ \ \ \ \ \ \ \ \ \ \ \ \ \ \ \ \
\ \ \ \ \ \ \ \ \ \ \ \ \ \ \ \ \ \ \ \ \ \ \ \ \ \ \ \ \ \ \ \ \ \
\ \ \ \ \ \ \ \ \ \ \ \ \ \ \ \
$$
\begin{eqnarray*}
&&=\frac{W_{\lambda}(\eta)}{h_{0}}\int_{0}^{h_{0}}
\frac{G_{\lambda}(x+h,\eta)e((x+h)^{2}+2(x+h)(\pm\sqrt{n}-\sqrt{U}-\eta-\nu))}
{(x+h)^{\alpha}4\pi
i(x+h\pm\sqrt{n}-\sqrt{U}-\eta-\nu)}|_{\sqrt{U}}^{(\lambda_{0}-\lambda)\sqrt{U}}\\
&&-\frac{W_{\lambda}(\eta)}{4\pi
ih_{0}}\int_{0}^{h_{0}}(\frac{G_{\lambda}(x+h,\eta)}
{(x+h)^{\alpha}})_{x}^{\prime}\frac{e((x+h)^{2}+2(x+h)(\pm\sqrt{n}-\sqrt{U}-\eta-\nu))dx}{(x+h\pm\sqrt{n}-\sqrt{U}-\eta-\nu)}\\
&&\ll\frac{U}{n}, \end{eqnarray*} for
$h_{0}=\frac{1}{\sqrt{U}L^{2}}$, where the last inequality is given
by integrating by parts for $h$. Since
$$\sum\limits_{n>n_{0}}d(n)n^{-5/4}U\ll\frac{UL}{n_{0}^{1/4}},$$
then taking $n_{0}=U^{20}$,  by (4.29) we obtain
\begin{equation}
\&_{\lambda}^{\pm}(\eta)=\sqrt{2}e(-\frac{1}{8}\pm\frac{1}{8})\sum\limits_{n\leq
U^{20}}d(n)n^{-\frac{1}{4}}\sum\limits_{\alpha=0}^{2}\gamma_{\alpha}(\frac{\pm
i}{4\pi\sqrt{n}})^{\alpha}j_{n\alpha}^{\pm}(\eta)+O_{\eta}(U^{-\frac{3}{2}}).
\end{equation}
As $\sqrt{n}>(\lambda_{0}-\lambda-1)\sqrt{U}$, moving the integral
of $j_{n\alpha}^{-}(\eta)$, we obtain
$$
j_{n\alpha}^{-}(\eta)=\frac{1}{h_{0}}\int_{0}^{h_{0}}dh(\int_{\sqrt{U}+h}^{\sqrt{U}+h-z_{0}}
+\int_{\sqrt{U}+h-z_{0}}^{(\lambda_{0}-\lambda)\sqrt{U}+h-z_{0}}+\int_{(\lambda_{0}-\lambda)\sqrt{U}+h-z_{0}}^{(\lambda_{0}-\lambda)\sqrt{U}+h}
$$
$$
\times \frac{e(z^{2}-2z(\sqrt{U}+\sqrt{n}+\nu))f(z,\eta)dz}{z^{\alpha}}
$$
\begin{equation}
=\frac{1}{h_{0}}\int_{0}^{h_{0}}dh(I_{1nh}^{-}(\eta)+I_{2nh}^{-}(\eta)+I_{3nh}^{-}(\eta)), \ \ \ \ {\rm say}
\end{equation}
where $z_{0}=Le(\frac{1}{8})$ and
\begin{equation}f(z,\eta)=W_{\lambda}(\eta)G_{\lambda}(z,\eta)e(-2z\eta).\end{equation}
For $z=x-\rho e(\frac{1}{8})$, we have
\begin{equation}|e(z^{2}-2\sqrt{n}(\sqrt{U}+\sqrt{n}+\nu)z)|=e^{-2\pi\rho^{2}-2\sqrt{2}\pi\rho(\sqrt{U}+\sqrt{n}+\nu-x)}\ll e^{-2\pi\rho^{2}}\end{equation}
holds for $x\leq(\lambda_{0}-\lambda)\sqrt{U}+h,\
\sqrt{n}>(\lambda_{0}-\lambda-1)\sqrt{U},\ h,\nu=O(L^{-1}),\
0\leq\rho\leq L$. As $z\in I_{1nh}^{-}(\eta),\  z=\sqrt{U}+h-\rho
e(\frac{1}{8})$; as $z\in I_{3nh}^{-}(\eta),\
z=(\lambda_{0}-\lambda)\sqrt{U}+h-\rho e(\frac{1}{8})$. Repeating
the proof of (4.26), we have
\begin{equation}f(z,\eta)=O_{\eta}(U^{-90}),\  z\in I_{1nh}^{-},\ I_{3nh}^{-}.\end{equation}
It follows from (4.34), (4.35 ) and $z^{-\alpha}\leq1$ that
\begin{equation}I_{1nh}^{-}(\eta)+I_{3nh}^{-}(\eta)=O_{\eta}(U^{-90}).\end{equation}
As $z\in I_{2nh}^{-}(\eta), z=x-Le(\frac{1}{8})$. Clearly,
$f(z,\eta)\ll 1,z=O(L)$. Hence, by (4.33) and the definition (4.32)
of $I_{2nh}^{-}(\eta)$ we have
\begin{equation}I_{2nh}^{-}(\eta)\ll\int_{\sqrt{U}+h}^{(\lambda_{0}-\lambda)\sqrt{U}+h}e^{-2\pi L^{2}}dx\ll e^{- L^{2}}\ll_{\eta}U^{-90}.\end{equation}
Moreover, (4.32), (4.36) and (4.37) deduce that
\begin{equation}j_{n\alpha}^{-}(\eta)=O_{\eta}(U^{-90})\end{equation}
for $\sqrt{n}>(\lambda_{0}-\lambda-1)\sqrt{U}$. Moving the integral
line of $j_{n\alpha}^{+}(\eta)$, we obtain
\begin{eqnarray*}
j_{n\alpha}^{+}(\eta)&=&\frac{1}{h_{0}}\int_{0}^{h_{0}}dh(\int_{\sqrt{U}+h}^{\sqrt{U}+h+z_{0}}
+\int_{\sqrt{U}+h+z_{0}}^{(\lambda_{0}-\lambda)\sqrt{U}+h+z_{0}}+\int_{(\lambda_{0}-\lambda)\sqrt{U}+h+z_{0}}^{(\lambda_{0}-\lambda)\sqrt{U}+h})\\
&&\times
\frac{e(z^{2}+2(-\sqrt{U}+\sqrt{n}-\nu)z)f(z,\eta)dz}{z^{\alpha}}
\end{eqnarray*}
\begin{equation}=\frac{1}{h_{0}}\int_{0}^{h_{0}}dh(I_{1nh}^{+}(\eta)+I_{2nh}^{+}(\eta)+I_{3nh}^{+}(\eta)), \ \ \ \ {\rm say}\end{equation}
If $z=x+\rho e(\frac{1}{8})$, then
\begin{eqnarray*}
|e(z^{2}+2(\sqrt{U}+\sqrt{n}-\nu)z)|&=&|e(i\rho^{2}+2\rho
e(\frac{1}{8})(-\sqrt{U}+\sqrt{n}-\nu+x))|\\
&=&e^{-2\pi\rho^{2}-2\sqrt{2}\rho(-\sqrt{U}+\sqrt{n}-\nu+x)}
\end{eqnarray*}
\begin{equation}
\ll e^{-2\pi\rho^{2}} \ \ \ \ \ \ \end{equation} for
$x\geq\sqrt{U}+h,n\geq1, and \ \nu,h=O(L^{-1}),0\leq\rho\leq L$.

Similar to the proof of (4.38), we have
\begin{equation}j_{n\alpha}^{+}(\eta)=O_{\eta}(U^{-90})\end{equation}
for $n\geq1$.  It follows from (4.28), (4.31), (4.38) and (4.41)
that
\begin{equation}S_{\lambda}^{*}(\eta)=\sqrt{2}e(-\frac{1}{4})
\sum\limits_{\sqrt{n}\leq
(\lambda_{0}-\lambda-1)\sqrt{U}}d(n)n^{-\frac{1}{4}}
\sum\limits_{\alpha=0}^{2}\gamma_{\alpha}(\frac{-
i}{4\pi\sqrt{n}})^{\alpha}j_{n\alpha}^{-}(\eta)+\delta(\eta)+O_{\eta}(U^{-\frac{3}{2}}),\end{equation}
where
\begin{equation}\delta(\eta)\ll_{\eta}\sum\limits_{n\leq
U^{20}}d(n)n^{-\frac{1}{4}}U^{-90}\ll_{\eta}U^{-\frac{3}{2}}.\end{equation}
Moving the inner integral line of $j_{n\alpha}^{-}(\eta)$, we have
$$
j_{n\alpha}^{-}(\eta)=\frac{1}{h_{0}}\int_{0}^{h_{0}}dh(\int_{\sqrt{U}+h}^{\sqrt{U}+h-z_{0}}
+\int_{\sqrt{U}+h-z_{0}}^{x_{n}-z_{0}}+\int_{x_{n}-z_{0}}^{x_{n}+z_{0}}
+\int_{x_{n}+z_{0}}^{(\lambda_{0}-\lambda)\sqrt{U}+h+z_{0}}+\int_{(\lambda_{0}-\lambda)\sqrt{U}+h+z_{0}}^{(\lambda_{0}-\lambda)\sqrt{U}+h})
$$
$$
\times
\frac{W_{\lambda}(\eta)e((z-x_{n})^{2}-x_{n}^{2})dz}{z^{\alpha}} \ \
\ \ \ \ \ \ \ \ \ \ \ \ \ \ \ \ \ \ \ \ \ \ \ \ \ \ \ \ \ \ \ \ \ \
\ \ \ \ \ \ \ \ \ \ \ \ \ \ \ \ \ \
$$
\begin{equation}
=\frac{1}{h_{0}}\int_{0}^{h_{0}}dh(I_{-1nh}^{-}(\eta)
+I_{-2nh}^{-}(\eta)+I_{0}(\eta)+I_{1nh}^{+}(\eta)+I_{2nh}^{+}(\eta)),\
\ \ \ {\rm say},
\end{equation} \noindent for $\sqrt{n}\leq
(\lambda_{0}-\lambda-1)\sqrt{U}$, where $z_{0}=Le(\frac{1}{8}),$
\begin{equation}x_{n}=\sqrt{U}+\sqrt{n}+\nu+\eta.
\end{equation}

\begin{picture}(400,200)
\put(100,100){\line(1,0){10}}\put(120,100){\line(1,0){10}}\put(140,100){\line(1,0){10}}\put(160,100){\line(1,0){10}}\put(180,100){\line(1,0){10}}
\put(200,100){\line(1,0){10}}\put(120,100){\line(1,0){10}}\put(140,100){\line(1,0){10}}\put(160,100){\line(1,0){10}}\put(180,100){\line(1,0){10}}
\put(200,100){\line(1,0){10}}\put(220,100){\line(1,0){10}}\put(240,100){\line(1,0){10}}\put(260,100){\line(1,0){10}}\put(280,100){\line(1,0){10}}
\put(300,100){\line(1,1){50}}\put(200,100){\line(1,1){50}}
\put(200,100){\line(-1,-1){50}} \put(100,100){\line(-1,-1){50}}
\put(350,150){\line(-1,0){100}}\put(250,150){\vector(1,0){50}}
\put(50,50){\line(1,0){100}}\put(50,50){\vector(1,0){50}}
\put(360,160){\makebox(10,1)[c]{($\lambda_0-\lambda)\sqrt{U}+h+z_0$}}
\put(250,160){\makebox(5,1)[c]{$x_n+z_0$}}
\put(300,90){\makebox(10,1)[c]{$(\lambda_0-\lambda)\sqrt{U}+h$}}
\put(210,90){\makebox(2,1)[c]{$x_n$}}
\put(150,40){\makebox(5,1)[c]{$x_n-z_0$}}
\put(50,40){\makebox(7,1)[c]{$\sqrt{U}+h-z_0$}}
\put(100,110){\makebox(4,1)[c]{$\sqrt{U}+h$}}
\put(220,110){\makebox(1,2)[c]{$\frac{\pi}{4}$}}\put(200,20){\makebox(5,1)[c]{Fig.4.1}}
\end{picture}

\noindent Repeating the proof of (4.38) and (4.41), we have
\begin{equation}I_{-2nh}^{-}(\eta)+I_{-1nh}^{-}(\eta)+I_{1nh}^{+}(\eta)+I_{2nh}^{+}(\eta)\ll_{\eta}U^{-90}.\end{equation}
Hence
\begin{eqnarray*}
j_{n\alpha}^{-}(\eta)&=&\frac{1}{h_{0}}\int_{0}^{h_{0}}I_{0}(\eta)dh+O_{\eta}(U^{-90})=I_{0}(\eta)+O_{\eta}(U^{-90})\\
&=&W_{\lambda}(\eta)e(-x_{n}^{2})\int_{x_{n}-z_{0}}^{x_{n}+z_{0}}\frac{G_{\lambda}(z,\eta)e((z-x_{n})^{2})dz}{z^{\alpha}}+O_{\eta}(U^{-90})
\ \ \ \ \ \ \ \ \ \ \ \ \ \ \ \
\end{eqnarray*}
\begin{equation}=W_{\lambda}(\eta)e(-x_{n}^{2}+\frac{1}{8})
\int_{-L}^{L}\frac{G_{\lambda}(x_{n}+\rho
e(\frac{1}{8}),\eta)e^{-2\pi\rho^{2}}d\rho}{(x_{n}+\rho
e(\frac{1}{8}))^{\alpha}}+O_{\eta}(U^{-90}).\end{equation}
By (4.1)
and (4.2),
$$G_{\lambda}(x_{n}+\rho e(\frac{1}{8}),\eta)=e^{-\frac{4\pi(x_{n}+\rho e(\frac{1}{8})+\lambda(\sqrt{U}+\eta+\nu))^{2}}{A^{2}}} \ \ \ \ \ \ \
\ \ \ \ \ \ \ \ \ \ \ \ \ \ \ \ $$
\begin{equation}=G_{\lambda+1}(\sqrt{n},\eta)(1+C_{1}(n)\rho+C_{2}(n)\rho^{2}+C_{3}(n)\rho^{3}+O(\frac{\rho^{4}}{U^{2}})),\end{equation}
$$C_{1}(n)\leq\frac{1}{\sqrt{U}},\ C_{3}(n)\ll\frac{1}{\sqrt{U}^{3}},$$
$$C_{2}(n)\ll\frac{1}{U},\ C_{2}'(x)\ll\frac{1}{U\sqrt{xU}}$$
and
\begin{eqnarray*}
W_{\lambda}(\eta)e(-x_{n}^{2})&=&e(-(\lambda+1)(\sqrt{U}+\eta+\nu)^{2}-2\sqrt{n}(\sqrt{U}+\eta+\nu)-n)\\
&=&W_{\lambda+1}(\eta)e(-2\sqrt{n}(\sqrt{U}+\eta+\nu))
\end{eqnarray*}
 for
$e(-n)=1$. Therefore,
\begin{eqnarray*}
j_{n\alpha}^{-}(\eta)&=&\frac{e(\frac{1}{8})W_{\lambda+1}(\eta)G_{\lambda+1}(\sqrt{n},\eta)e(-2\sqrt{n}(\sqrt{U}+\eta+\nu))}
{(\sqrt{U}+\sqrt{n}+\nu+\eta)^{\alpha}}\\
&&\times\int_{-L}^{L}e(-2\pi\rho^{2})(1+C_{1}(n)\rho+C_{2}(n)\rho^{2}+C_{3}(n)\rho^{3}+O(\frac{\rho^{4}}{U^{2}}))
(1+O(\frac{\alpha|\rho|}{(\sqrt{U})^{\alpha}}))
\end{eqnarray*}
\begin{equation}=\frac{e(\frac{1}{8})W_{\lambda+1}(\eta)G_{\lambda+1}(\sqrt{n},\eta)e(-2\sqrt{n}(\sqrt{U}+\eta+\nu))}
{\sqrt{2}(\sqrt{U}+\sqrt{n}+\nu+\eta)^{\alpha}}(1+C_{20}(n)+O(\frac{1}{U^{2}})+\frac{\alpha}{(\sqrt{U})^{\alpha}})\end{equation}
where \begin{equation}C_{20}(n)\ll \frac{1}{U},\
C_{20}'(x)\ll\frac{1}{U\sqrt{xU}}.\end{equation}

It follows from (4.42) and (4.49) that
$$S_{\lambda}^{*}(\eta)=e(-\frac{1}{8})W_{\lambda+1}(\eta)\sum\limits_{\sqrt{n}
\leq
(\lambda_{0}-\lambda-1)\sqrt{U}}d(n)n^{-\frac{1}{4}}G_{\lambda+1}(\sqrt{n},\eta)
\ \ \ \ \ \ \ \ \ \ \ \ \ \ \ \ \ \ \ \ \ \ \ \ \ \ \ \ \ $$
\begin{equation}\times\sum\limits_{\alpha=0}^{2}\gamma_{\alpha}\frac{(- i)^{\alpha}}{(4\pi\sqrt{n}(\sqrt{U}
+\eta+\nu))^{\alpha}}e(-2\sqrt{n}(\sqrt{U}+\eta+\nu))+\delta_{1}+\delta_{2}+O_{\eta}(U^{-\frac{3}{2}}),\end{equation}
where
$$\delta_{1}\ll\sum\limits_{\sqrt{n}\leq
(\lambda_{0}-\lambda-1)\sqrt{U}}d(n)n^{-\frac{1}{4}}
\sum\limits_{\alpha=0}^{2}\frac{1}{(\sqrt{n}+\sqrt{U})^{\alpha}}(\frac{1}{U^{2}}+\frac{\alpha}{(\sqrt{U})^{\alpha}}))\
\ \ \ \ \ \ \ \ \ \ \ \ \ \ \ \ \ \ \ \ $$
\begin{equation}\ll U^{-3/4}\ll_{\eta}U^{-3/4+\varepsilon}, \ 0<\varepsilon=O(\frac{1}{\log L}), \ \ \ \ \ \ \ \
\ \ \ \ \ \ \ \ \ \ \ \ \ \ \ \ \ \ \ \ \ \ \ \ \ \ \ \
\end{equation}
\begin{eqnarray*}
\delta_{2}&=&e(-\frac{1}{8})W_{\lambda+1}(\eta)\sum\limits_{\sqrt{n}
\leq
(\lambda_{0}-\lambda-1)\sqrt{U}}C_{20}(\sqrt{n})d(n)n^{-\frac{1}{4}}G_{\lambda+1}(\sqrt{n},\eta)\\
&&\times\sum\limits_{\alpha=0}^{2}\gamma_{\alpha}\frac{(-
i)^{\alpha}}
{(4\pi\sqrt{n}(\sqrt{U}+\sqrt{n}+\eta+\nu))^{\alpha}}e(-2\sqrt{n}(\sqrt{U}+\eta+\nu))\ \ \ \ \ \ \ \ \ \ \ \ \  \\
&=&e(-\frac{1}{8})W_{\lambda+1}(\eta) \end{eqnarray*}
\begin{equation}
\ \ \ \ \ \ \ \ \ \ \ \ \ \ \  \times\sum\limits_{\sqrt{n} \leq
(\lambda_{0}-\lambda-1)\sqrt{U}}C_{20}(\sqrt{n})d(n)n^{-\frac{1}{4}}e(-2\sqrt{n}(\sqrt{U}+\eta+\nu))+O(U^{-3/4}).
\end{equation}
By (4.51) and Lemma 3.4,
\begin{equation}\delta_{2}\ll U^{-3/4+\varepsilon}\ll_{\eta}U^{-3/4+\varepsilon},\ 0<\varepsilon=O(\frac{1}{\log L}).\end{equation}
The last inequality has been used the explain (1.32). Thus, by
(4.51), if $\lambda=\lambda_{0}-2$, then
\begin{equation}S_{\lambda-2}^{*}(\eta)=S_{\lambda_{0}-1}(\eta)+\delta_{\lambda_{0}-1}(\eta)+O_{\eta}(U^{-3/4+\varepsilon});\end{equation}
if $0\leq\lambda\leq\lambda_{0}-1$, then
\begin{equation}S_{\lambda}^{*}(\eta)=S_{\lambda+1}(\eta)+S_{\lambda+1}^{*}(\eta)+\delta_{\lambda+1}(\eta)+O_{\eta}(U^{-3/4+\varepsilon}),\end{equation}
where $0<\varepsilon=O(\frac{1}{\log L})$ and
$$\delta_{\lambda+1}(\eta)=\ \ \ \ \ \ \ \ \ \ \ \ \ \ \ \ \ \ \ \ \ \ \ \ \ \ \ \ \ \ \ \ \ \ \ \ \ \ \ \ \ \ \ \ \ \ \ \ \ \ \ \ \ \ \ \
\ \ \ \ \ \ \ \ \ \ \ \ \ \ \ \ \ \ \ \ \ \ \ \ \ \ \ \ \ \ \ \ \ \
\ \ \ \ \ \ \  $$
\begin{equation}=e(-\frac{1}{8})W_{\lambda+1}(\eta)
\sum\limits_{\alpha=1}^{2}\gamma_{\alpha}(\frac{- i}{4\pi})^{\alpha}
\sum\limits_{\sqrt{U}<\sqrt{n}\leq
(\lambda_{0}-\lambda-1)\sqrt{U}}d(n)n^{-\frac{1}{4}-\frac{\alpha}{2}}
\frac{e(-2\sqrt{n}(\sqrt{U}+\eta+\nu))}{(\sqrt{U}+\sqrt{n}+\nu+\eta)^{\alpha}}.\end{equation}

Taking $f(n)=\frac{1}{(\sqrt{U}+\sqrt{n}+\nu+\eta)^{\alpha}}$ in
Lemma 3.5, we know that the inner sum is $\ll U^{-3/4+\varepsilon},\
\alpha\geq 1$, that is
\begin{equation}\delta_{\lambda+1}(\eta)\ll
U^{-3/4+\varepsilon},\ 0<\varepsilon=O(\frac{1}{\log
L}).\end{equation}

Finally, (4.6) follows from (4.55), (4.56) and (4.58) and the proof
of Lemma 4.1 is complete.
\end{proof}

\begin{lem} Let
$$S_{\frac{1}{2}}(\eta)=\sum\limits_{\frac{\sqrt{U}}{2}<\sqrt{n}
\leq\lambda_{0}\sqrt{U}}d(n)n^{-\frac{1}{4}}e^{-\frac{4\pi
n}{A^{2}}}e(-2\sqrt{n}(\sqrt{U}+\eta+\nu)-\frac{1}{8}) \ \ \ \ \ \ \
\ \ \ \ \ \ \ \ \ \ \ $$
\begin{equation}=S_{0}^{*}(\eta)+\sum\limits_{\frac{\sqrt{U}}{2}<\sqrt{n}
\leq\sqrt{U}}d(n)n^{-\frac{1}{4}}e^{-\frac{4\pi
n}{A^{2}}}e(-2\sqrt{n}(\sqrt{U}+\eta+\nu)-\frac{1}{8}),\end{equation}
then
\begin{equation}S_{\frac{1}{2}}(\eta)=\sum\limits_{\lambda=1}^{\lambda_{0}}S_{\lambda}^{+}
(\eta)+\sum\limits_{\lambda=1}^{\lambda_{0}}S_{\lambda}^{-}(\eta)+O_{\eta}(U^{-3/4+\varepsilon}),\
0<\varepsilon=O(\frac{1}{\log L}),\end{equation} where
$S_{0}^{*}(\eta),\ \eta,\ \nu$ are as Lemma 4.1, and
$$S_{\lambda}^{+}(\eta)=W_{\lambda}(\eta)\sum\limits_{\sqrt{n}
\leq \frac{\sqrt{U}}{2}}d(n)n^{-\frac{1}{4}}G_{\lambda}(\sqrt{n},\eta)\ \ \ \ \ \ \ \ \ \ \ $$
\begin{equation}\ \ \ \ \ \ \ \ \ \ \ \times\sum\limits_{\alpha=0}^{2}\gamma_{\alpha}
\frac{(-
i)^{\alpha}e(-2\sqrt{n}(\sqrt{U}+\eta+\nu)-\frac{1}{8})}{(4\pi\sqrt{n}(\sqrt{U}+\sqrt{n}+\eta+\nu))^{\alpha}},
\end{equation}
$$S_{\lambda}^{-}(\eta)=W_{\lambda}(\eta)\sum\limits_{\sqrt{n}
\leq \frac{\sqrt{U}}{2}}d(n)n^{-\frac{1}{4}}G_{\lambda}(-\sqrt{n},\eta)\ \ \ \ \ \ \ \ \ \ \ $$
\begin{equation}\ \ \ \ \ \ \ \ \ \ \ \times\sum\limits_{\alpha=0}^{2}\gamma_{\alpha}
\frac{(
i)^{\alpha}e(2\sqrt{n}(\sqrt{U}+\eta+\nu)+\frac{1}{8})}{(4\pi\sqrt{n}(\sqrt{U}-\sqrt{n}+\eta+\nu))^{\alpha}}.\end{equation}
\end{lem}

\begin{proof} Sinc\begin{eqnarray*}
S_{\lambda}(\eta)&=&W_{\lambda}(\eta)\sum\limits_{\sqrt{n}\leq\sqrt{U}}
d(n)n^{-\frac{1}{4}}G_{\lambda}(\sqrt{n},\eta) \\
&& \times\sum\limits_{\alpha=0}^{2}\gamma_{\alpha}\frac{(-
i)^{\alpha}e(-2\sqrt{n}
(\sqrt{U}+\eta+\nu)-\frac{1}{8})}{(4\pi\sqrt{n}(\sqrt{U}+\sqrt{n}+\eta+\nu))^{\alpha}}\\
&=&W_{\lambda}(\eta)(\sum\limits_{\sqrt{n}\leq\frac{\sqrt{U}}{2}}+\sum\limits_{\frac{\sqrt{U}}{2}<\sqrt{n}\leq\sqrt{U}}),
\end{eqnarray*}
then by Lemma 4.1,
\begin{equation}S_{\frac{1}{2}}(\eta)=\sum\limits_{\lambda=1}^{\lambda_{0}}S_{\lambda}^{+}(\eta)
+\sum\limits_{\lambda=0}^{\lambda_{0}-1}H_{\lambda}(\eta)+O_{\eta}(U^{-3/4+\varepsilon}),\end{equation}
where $S_{\lambda_{0}}^{+}(\eta)\ll e^{-L^{2}}$ is used, and
\begin{eqnarray*}
H_{\lambda}(\eta)&=&W_{\lambda}(\eta)\sum\limits_{\frac{\sqrt{U}}{2}<\sqrt{n}\leq\sqrt{U}}
d(n)n^{-\frac{1}{4}}G_{\lambda}(\sqrt{n},\eta)\\
&& \times\sum\limits_{\alpha=0}^{2}\gamma_{\alpha}\frac{(-
i)^{\alpha}e(-2\sqrt{n}(\sqrt{U}+\eta+\nu)
-\frac{1}{8})}{(4\pi\sqrt{n}(\sqrt{U}+\sqrt{n}+\eta+\nu))^{\alpha}}\\
&=&W_{\lambda}(\eta)
\sum\limits_{\frac{\sqrt{U}}{2}<\sqrt{n}\leq\sqrt{U}}
d(n)n^{-\frac{1}{4}}G_{\lambda}(\sqrt{n},\eta)e(-2\sqrt{n}
(\sqrt{U}+\eta+\nu)-\frac{1}{8})
\end{eqnarray*}
\begin{equation}
+O(U^{-3/4+\varepsilon}),\ \ \ \ \ \ \ \ \ \ \ \ \ \ \ \ \ \ \ \ \ \
\ \ \ \ \ \ \ \ \ \ \ \ \ \ \ \ \ \ \ \ \ \ \ \ \ \ \ \
\end{equation}
 the $O$-term is given by the
proof of (4.58).

The following proof is similar to Lemma 4.1. Since
$$0\leq\{(\frac{\sqrt{U}}{2})^{2}\}\leq\frac{3}{4},(\frac{\sqrt{U}}{2}+h)^{2}-\frac{\sqrt{U}}{2})^{2}\leq\frac{1}{L^{3/2}}$$
for $0<h\leq\frac{1}{\sqrt{U}L^{2}}$, then we have the corresponding
(4.9). If replace $\sqrt{U}$ by $\sqrt{U}/2$ in (4.18), we have also
\begin{equation}e(\pm2\eta(\frac{\sqrt{U}}{2}))=1\end{equation}
for $\eta=\eta_{1}+\cdots+\eta_{K},\eta_{j}=0$  or $k_{1}/\sqrt{U}$.
Similarly, we obtain that

\begin{equation}
H_{\lambda}(\eta)=\sqrt{2}e(\frac{1}{8}) \sum\limits_{\sqrt{n}
\leq\frac{\sqrt{U}}{2}}
d(n)n^{-\frac{1}{4}}\sum\limits_{\alpha=0}^{2}\gamma_{\alpha}
(\frac{i}{4\pi})^{\alpha}j_{n\alpha}^{+}(\eta) +O_{\eta}(U^{-3/2}),
\end{equation}

(different from (4.42), present
main term is $j_{n\alpha}^{+}$, because
$x\in(\frac{\sqrt{U}}{2},\sqrt{U}+h)$) where
$$j_{n\alpha}^{+}(\eta)=\frac{W_{\lambda}(\eta)}{h_{0}}
\int_{0}^{h_{0}}dh\int_{\frac{\sqrt{U}}{2}+h}^{\sqrt{U}+h}\frac{G_{\lambda}(x,\eta)e(x^{2}+2x(\sqrt{n}-\sqrt{U}-\eta-\nu))dx}{x^{\alpha}}$$
\begin{equation}=\frac{W_{\lambda}(\eta)e(-x_{n}^{2})}{h_{0}}
\int_{0}^{h_{0}}dh\int_{\frac{\sqrt{U}}{2}+h}^{\sqrt{U}+h}\frac{G_{\lambda}(x,\eta)e((x-x_{n})^{2})dx}{x^{\alpha}},
\ \ \ \ \end{equation}
\begin{equation}x_{n}=\sqrt{U}+\eta+\nu-\sqrt{n}. \ \ \ \ \ \ \ \ \ \ \ \ \ \ \ \ \ \ \ \ \ \ \ \ \ \
\ \ \ \ \ \ \ \ \ \ \ \ \ \ \ \ \ \ \ \ \ \ \ \  \ \ \
\end{equation}
Repeating the proof of (4.51) we see that
\begin{eqnarray}
H_{\lambda}(\eta)&=&e(\frac{1}{8})W_{\lambda}(\eta)
\sum\limits_{\sqrt{n}\leq\frac{\sqrt{U}}{2}}
d(n)n^{-\frac{1}{4}}e(-x_{n}^{2})G_{\lambda}(x_{n},\eta)
\sum\limits_{\alpha=0}^{2}\gamma_{\alpha}(\frac{i}{4\pi})^{\alpha}\frac{1}{(4\pi\sqrt{n}x_{n})^{\alpha}}\notag\\
&&+O_{\eta}(U^{-3/4+\varepsilon}).
\end{eqnarray} Since
\begin{equation}G_{\lambda}(x_{n},\eta)=e^{-\frac{4\pi(\sqrt{U}+\eta+\nu-\sqrt{n}+\lambda(\sqrt{U}+\eta+\nu))^{2}}{A^{2}}}
=G_{\lambda+1}(-\sqrt{n},\eta) \ \ \ \ \ \ \ \end{equation}
$$W_{\lambda}(\eta)e(-x_{n}^{2})=e(-(\lambda+1)(\sqrt{U}+\eta+\nu)^{2}+2\sqrt{n}(\sqrt{U}+\eta+\nu)+n)$$
\begin{equation}=W_{\lambda+1}(\eta)e(2\sqrt{n}(\sqrt{U}+\eta+\nu)), \ \ \ \ \ \ \ \ \ \ \ \end{equation}
then by (4.69),
\begin{equation}H_{\lambda}(\eta)=S_{\lambda+1}^{-}(\eta)+O(U^{-3/4+\varepsilon}), \end{equation}
where $\varepsilon=O(\frac{1}{\log L})$ and $S_{\lambda}^{-}(\eta)$
is (4.62). Thus, we obtain (4.60) immediately by (4.63) and (4.72).
The proof of Lemma 4.2 is complete.
\end{proof}

\begin{lem} Let
\begin{equation}\lambda_{0}=b_{0}=[C_{0}^{2}L],\ C_{0}\geq200,\end{equation}
where $\eta$ is (1.15). If
$\frac{1}{8}\leq\{2\sqrt{U}\nu\}\leq\frac{3}{8}$, then
$${\rm Re}S_{\frac{1}{2}}(\eta)
=-{\rm Re}\sum\limits_{\sqrt{n}\leq \frac{\sqrt{U}}{2}}
d(n)n^{-\frac{1}{4}}e^{-\frac{4\pi
n}{A^{2}}}e(-2\sqrt{n}(\sqrt{U}+\eta+\nu)-\frac{1}{8})$$
\begin{equation}\times\sum\limits_{\alpha=0}^{2}\gamma_{\alpha}(\frac{-i}{4\pi\sqrt{n}(\sqrt{n}+\sqrt{U}
+\eta+\nu)})^{\alpha}+O_{\eta}(U^{-3/4+\varepsilon}),\end{equation}
where $S_{\frac{1}{2}}(\eta)$ is (4.59),
$\varepsilon=O(\frac{1}{\log L})$.
\end{lem}

\begin{proof} By the definitions (4.1) and (4.2) of
$G_{\lambda}(\sqrt{n},\eta)$ and $W_{\lambda}(\eta)$,
\begin{equation}G_{0}(\sqrt{n},\eta)=e^{-\frac{4\pi n}{A^{2}}},\end{equation}
\begin{equation}W_{0}(\eta)=1.\end{equation}
By Lemma 4.2,
\begin{equation}S_{\frac{1}{2}}(\eta)=\sum\limits_{\lambda=1}^{b_{0}}(S_{\lambda}^{+}(\eta)+S_{\lambda}^{-}(\eta))
+O_{\eta}(U^{-3/4+\varepsilon}),\ \varepsilon=O(\frac{1}{\log
L}).\end{equation} Since
\begin{equation}\frac{1}{(\sqrt{n}(\sqrt{U}+\eta+\nu-\sqrt{n}))^{\alpha}}
=\frac{1}{(\sqrt{n}(\sqrt{U}+\eta+\nu+\sqrt{n}))^{\alpha}}+(\sqrt{n})^{1-\alpha}\delta_{\alpha}(n),\end{equation}
where
\begin{eqnarray}
\delta_{\alpha}(n)&=&\frac{1}{\sqrt{n}}(\frac{1}{(\sqrt{U}+\eta+\nu-\sqrt{n})^{\alpha}}-\frac{1}{(\sqrt{U}+\eta+\nu+\sqrt{n})^{\alpha}})
\ll\frac{\alpha}{(\sqrt{U})^{\alpha+1}},\\
\delta_{\alpha}'(x)&\ll&\frac{\alpha}{\sqrt{xU}(\sqrt{U})^{\alpha+1}},
\end{eqnarray}
moreover, $\delta_{\alpha}(n)=0$ for $\alpha=0$. By Lemma 3.5 ,
\begin{equation}\sum\limits_{\alpha=1}^{2}\gamma_{\alpha}(\frac{i}{4\pi})^{\alpha}
\sum\limits_{\sqrt{n}\leq \frac{\sqrt{U}}{2}}
d(n)n^{-\frac{1}{4}}(\sqrt{n})^{1-\alpha}\delta_{\alpha}(n)G_{\lambda}
(-\sqrt{n},\eta)e(2\sqrt{n}(\sqrt{U}+\eta+\nu))\ll
U^{-3/4+\varepsilon}.\end{equation} This and (4.62) give
\begin{eqnarray}
S_{\lambda}^{-}(\eta)&=&W_{\lambda}(\eta)\sum\limits_{\sqrt{n}\leq
\frac{\sqrt{U}}{2}} d(n)n^{-\frac{1}{4}}G_{\lambda}(-\sqrt{n},\eta)
\sum\limits_{\alpha=0}^{2}\gamma_{\alpha}e(2\sqrt{n}(\sqrt{U}+\eta+\nu)+\frac{1}{8})\notag\\
&&\times(\frac{i}{4\pi\sqrt{n}(\sqrt{n}+\sqrt{U}+\eta+\nu)})^{\alpha}
+O(U^{-3/4+\varepsilon}),\ \varepsilon=O(\frac{1}{\log L}).
\end{eqnarray} It follows from (4.77), (4.82) and (4.61) that
\begin{eqnarray}
S_{\frac{1}{2}}(\eta)+\overline{S_{\frac{1}{2}}(\eta)}&=&
\sum\limits_{\sqrt{n}\leq \sqrt{U}}\sum\limits_{\alpha=0}^{2}
d(n)n^{-\frac{1}{4}}\gamma_{\alpha}
\frac{1}{(4\pi\sqrt{n}(\sqrt{n}+\sqrt{U}+\eta+\nu))^{\alpha}}\notag\\
&&\times((-i)^{\alpha}e(-2\sqrt{n}(\sqrt{U}+\eta+\nu)-\frac{1}{8})N(n)\notag\\
&&+\overline{(-i)^{\alpha}e(-2\sqrt{n}(\sqrt{U}+\eta+\nu)-\frac{1}{8})N(n)})
+O_{\eta}(U^{-3/4+\varepsilon}),
\end{eqnarray} where
\begin{equation}N(n)=\sum\limits_{\lambda=1}^{b_{0}}W_{\lambda}(\eta)G_{\lambda}(\sqrt{n},\eta)
+\sum\limits_{\lambda=1}^{b_{0}}\overline{W_{\lambda}(\eta)}G_{\lambda}(-\sqrt{n},\eta).\end{equation}
Putting $\lambda=-\lambda'$ in the second sum, note
\begin{equation}\overline{W_{-\lambda}(\eta)}=e(\overline{\lambda(\sqrt{U}+\eta+\nu)})=W_{\lambda}(\eta),\end{equation}
\begin{equation}G_{-\lambda}(-\sqrt{n},\eta)=e^{-\frac{4\pi(-\sqrt{n}-\lambda(\sqrt{U}+\eta+\nu))^{2}}{A^{2}}}=G_{\lambda}(\sqrt{n},\eta),\end{equation}
we have
\begin{eqnarray*}
N(n)&=&\sum\limits_{\lambda=-b_{0},\lambda\neq0}^{b_{0}}W_{\lambda}(\eta)G_{\lambda}(\sqrt{n},\eta)\\
&=&-W_{0}(\eta)G_{0}(\sqrt{n},\eta)+\sum\limits_{\lambda=-b_{0}}^{b_{0}}e^{-\frac{4\pi(\sqrt{n}
+\lambda(\sqrt{U}+\eta+\nu))^{2}}{A^{2}}}e(-\lambda(\sqrt{U}+\eta+\nu)^{2})\\
&=&-e^{-\frac{4\pi n}{A^{2}}}+\sum\limits_{\lambda
=-\infty}^{\infty}e^{-\frac{\pi(\lambda+b_{n})^{2}}{V_{1}}}e(-\lambda(\sqrt{U}+\eta+\nu)^{2})+O(U^{-10}),
\end{eqnarray*}
where
\begin{equation}V_{1}=\frac{A^{2}}{(\sqrt{U}+\eta+\nu)^{2}}, \ \ \ b_{n}=\frac{\sqrt{n}}{\sqrt{U}+\eta+\nu}.\end{equation}
Using
\begin{eqnarray*}
e(-\lambda(\sqrt{U}+\eta+\nu)^{2})&=&e(-\lambda(U+2\eta\sqrt{U}+2\nu\sqrt{U}+(\eta+\nu)^{2}))\\
&=&e(-\lambda(\{2\nu\sqrt{U}\}+(\eta+\nu)^{2})),
\end{eqnarray*}
for $U=[U], \ \ \eta=k_{1}/\sqrt{U}$, where $k_{1}$ is an integer,
we have
$$N(n)=-e^{-\frac{4\pi
n}{A^{2}}}+\sum\limits_{\lambda=-\infty}^{\infty}e^{-\frac{\pi}{V_{1}}(\lambda+b_{n})^{2}}e(-\lambda(\{2\nu\sqrt{U}\}+(\eta+\nu)^{2}))+O(U^{-10}).$$
By Lemma 3.2,
$$N(n)=\sqrt{V_{1}}\sum\limits_{\lambda=-\infty}^{\infty}e^{-\pi V_{1}(\lambda-\{2\nu\sqrt{U}\}-
(\eta+\nu)^{2})}e(b_{n}(\{2\nu\sqrt{U}\})+(\eta+\nu)^{2}-\lambda)-e^{-\frac{4\pi
n}{A^{2}}}+O(U^{-10}).$$ Since
$\frac{1}{8}\leq\{2\nu\sqrt{U}\}\leq\frac{3}{8},\
(\eta+\nu)^{2}=O(L^{-1})$ and
$$V_{1}=\frac{A^{2}}{(\sqrt{U}+\eta+\nu)^{2}}=C_{0}^{2}L(1+o(1))\geq\frac{200^{2}L}{2},$$
\noindent then $$N(n)=-e^{-\frac{4\pi n}{A^{2}}}+O(\sqrt{L}e^{-\frac{\pi C_{0}^{2}L}{200}})$$
\begin{equation}=-e^{-\frac{4\pi n}{A^{2}}}+O(U^{-10}).\end{equation}
 By (4.83) and (4.88) we obtain (4.74). The proof is complete.
\end{proof}

\ \

\section{EVALUATION of $S(B)$}

Now we evaluate \ $S(B)$ \ in (1.19), i.e.,
\begin{equation*}
S(B)=A\int_{-\frac{5\sqrt{L}}{A}}^{\frac{5\sqrt{L}}{A}}e^{-\pi
A^{2}\theta^{2}}d\theta \sum\limits_{\xi_{1}<\sqrt{n}-B-\theta\leq%
\xi_{2}}d(n)n^{\frac{-1}{4}},
\end{equation*}
where $A,\ B,\ \xi_{1},\ \xi_{2}, \cdots $ are (1.10)-(1.18).

By Lemma 2.4 and taking $\alpha _{0}=2$ in (1.12),
\begin{eqnarray}
S(B) &=&\sum\limits_{n\leq N}d(n)A\int_{-\frac{5\sqrt{L}}{A}}^{\frac{5\sqrt{L%
}}{A}}e^{-\pi A^{2}\theta ^{2}}d\theta \int_{(B+\xi _{1}+\theta
)^{2}}^{(B+\xi _{2}+\theta )^{2}}x^{-\frac{1}{4}}\Theta
(nx)dx+S_{l}(B)+O(\varepsilon _{1})  \notag \\
&=&2\sqrt{2}\sum\limits_{n\leq N}d(n)n^{-\frac{1}{4}}\sum\limits_{\alpha
=0}^{2}\frac{\gamma _{\alpha }A}{(4\pi \sqrt{n})^{\alpha }}\int_{-\frac{5%
\sqrt{L}}{A}}^{\frac{5\sqrt{L}}{A}}e^{-\pi A^{2}\theta ^{2}}j_{n\alpha
}(\theta )d\theta  \notag \\
&&+S_{l}(B)+O(U^{-\frac{3}{2}})+O(\varepsilon _{1})
\end{eqnarray}%
for any $\varepsilon _{1}>0$ as $N\geq N(\varepsilon _{1})$, where
\begin{eqnarray}
j_{n\alpha }(\theta ) &=&\int_{B+\xi _{1}+\theta }^{B+\xi _{2}+\theta }\frac{%
\cos (4\pi \sqrt{n}x+\frac{\pi }{4}+\frac{\alpha \pi }{2})dx}{x^{\alpha }}
\notag \\
&=&\mathrm{Re}\ e(\alpha /4+1/8)\int_{\xi _{1}}^{\xi _{2}}\frac{e(2\sqrt{n}%
(B+x+\theta ))dx}{(B+x+\theta )^{\alpha }},
\end{eqnarray}%
\begin{equation}
S_{l}(B)=A\int_{-\frac{5\sqrt{L}}{A}}^{\frac{5\sqrt{L}}{A}}e^{-\pi
A^{2}\theta ^{2}}d\theta \int_{(B+\xi _{1}+\theta )^{2}}^{(B+\xi _{2}+\theta
)^{2}}x^{\frac{-1}{4}}(\log x+2\gamma )dx.
\end{equation}%
The inner integral of $S_{l}(B)$ is
\begin{equation*}
2\int_{\xi _{1}+\theta }^{\xi _{2}+\theta }(B+x)^{\frac{1}{2}}(\log
(B+x)^{2}+2\gamma )dx=\sum\limits_{k=0}^{10}C_{k}\eta
^{k}+O(U^{-9}).
\end{equation*}%
Let $B^{\prime}=B-\eta=\sqrt{U}+\frac{k+j}{2\sqrt{U}}$ (see 1.11).
Then by (3.37),
\begin{equation}
S_{l}(B)=S_{l}(B^{\prime }+\eta )=O_{\eta }(U^{-9}).
\end{equation}%
By (5.2), we have%
\begin{eqnarray*}
j_{n\alpha }(\theta ) &=&\mathrm{Re}\ e(\alpha /4+1/8)(\frac{e(2\sqrt{n}%
(B+x+\theta ))}{4\pi i\sqrt{n}(B+x+\theta )^{\alpha }}%
\mbox {\LARGE \textbar
}_{x=\xi _{1}}^{x=\xi _{2}} \\
&&+O(\frac{1}{n}(1+\int_{\xi _{1}}^{\xi _{2}}\frac{\alpha dx}{(B+x+\theta
)^{\alpha +1}}))).
\end{eqnarray*}%
The integration by parts for $\theta $ gives
\begin{equation*}
A\int_{-\frac{5\sqrt{L}}{A}}^{\frac{5\sqrt{L}}{A}}e^{-\pi A^{2}\theta
^{2}}j_{n\alpha }(\theta )d\theta \ll \frac{A}{n}.
\end{equation*}%
Therefore,
\begin{equation}
\sum\limits_{U^{20}<n\leq N}d(n)n^{-\frac{1}{4}}\sum\limits_{\alpha =0}^{2}%
\frac{\gamma _{\alpha }A}{(4\pi \sqrt{n})^{\alpha }}\int_{-\frac{5\sqrt{L}}{A%
}}^{\frac{5\sqrt{L}}{A}}e^{-\pi A^{2}\theta ^{2}}j_{n\alpha }(\theta
)d\theta \ll \sum\limits_{U^{20}<n}d(n)n^{-\frac{5}{4}}A\ll U^{-4}.
\end{equation}%
By (5.1), (5.4) and (5.5) we obtain
\begin{equation}
S(B)=2\sqrt{2}\sum\limits_{n\leq U^{20}}d(n)n^{\frac{-1}{4}%
}\sum\limits_{\alpha =0}^{2}\frac{\gamma _{\alpha }A}{(4\pi \sqrt{n}%
)^{\alpha }}\int_{-\frac{5\sqrt{L}}{A}}^{\frac{5\sqrt{L}}{A}}e^{-\pi
A^{2}\theta ^{2}}j_{n\alpha }(\theta )d\theta +O_{\eta }(U^{-\frac{3}{2}}),\
(\varepsilon _{1}\rightarrow 0_{+}).
\end{equation}%
If $\sqrt{n}>C_{0}^{2}L\sqrt{U}$, then
\begin{equation*}
\int_{-\frac{5\sqrt{L}}{A}}^{\frac{5\sqrt{L}}{A}}e^{-\pi A^{2}\theta
^{2}}j_{n\alpha }(\theta )d\theta \ \ \ \ \ \ \ \ \ \ \ \ \ \ \ \ \ \ \ \ \
\ \ \ \ \ \ \ \ \ \ \ \ \ \ \ \ \ \ \ \ \ \ \ \ \ \ \ \ \ \ \ \ \ \ \ \ \ \
\ \ \ \ \ \ \ \ \ \ \ \ \ \ \ \ \ \ \ \ \ \ \ \ \ \
\end{equation*}
\begin{equation}
=\mathrm{Re}\ e(\alpha /4+1/8)\int_{\xi _{1}}^{\xi _{2}}e(2\sqrt{n}%
(B+x))dxA\int_{-\frac{5\sqrt{L}}{A}}^{\frac{5\sqrt{L}}{A}}e^{-\pi \theta
^{2}}\frac{e(2\sqrt{n}\theta )d\theta }{(B+x+\theta )^{\alpha }}.
\end{equation}%
Moving the inner integral line of (5.7), we have
\begin{equation}
A(\int_{-\frac{5\sqrt{L}}{A}}^{-\frac{5\sqrt{L}}{A}+\frac{i}{C_{0}\sqrt{U}}%
}+\int_{-\frac{5\sqrt{L}}{A}+\frac{i}{C_{0}\sqrt{U}}}^{\frac{5\sqrt{L}}{A}+%
\frac{i}{C_{0}\sqrt{U}}}+\int_{\frac{5\sqrt{L}}{A}+\frac{i}{C_{0}\sqrt{U}}}^{%
\frac{5\sqrt{L}}{A}})e^{-\pi A^{2}w^{2}}\frac{e(2\sqrt{n}w)dw}{%
(B+x+w)^{\alpha }}=I_{1}+I_{2}+I_{3},\ \ \ \ \mathrm{say}.
\end{equation}%
As $w\in I_{1}$, $w=-\frac{5\sqrt{L}}{A}+iv$, $0\leq v\leq \frac{1}{C_{0}%
\sqrt{U}}$, and
\begin{equation}
|e(2\sqrt{n}w)|=|e(2i\sqrt{n}v)|=e^{-4\pi \sqrt{n}v}.
\end{equation}%
Hence
\begin{eqnarray*}
I_{1} &\ll &A\int_{0}^{\frac{1}{C_{0}\sqrt{U}}}|e^{-\pi A^{2}(-\frac{5\sqrt{L%
}}{A}+iv)^{2}}|e^{-4\pi \sqrt{n}v}dv \\
&\ll &A\int_{0}^{\frac{1}{C_{0}\sqrt{U}}}e^{-\pi C_{0}^{2}LU(\frac{25L}{%
C_{0}^{2}U^{L}}-\frac{1}{C_{0}^{4}\sqrt{U}})}dv \\
&\ll &\frac{A}{\sqrt{U}}e^{-20\pi L}\ll U^{-20}.
\end{eqnarray*}%
In the same way, $I_{3}\ll U^{-20}$. When $w\in I_{2}$, $w=u+\frac{i}{C_{0}%
\sqrt{U}}$. By (5.9),
\begin{equation*}
|e(2\sqrt{n}w)|=e^{\frac{-4\pi \sqrt{n}}{C_{0}\sqrt{U}}}\ll e^{-4\pi
C_{0}L}\ll U^{-40},\ \mathrm{for}\ \sqrt{n}>C_{0}^{2}L\sqrt{U}.
\end{equation*}%
Hence
\begin{equation*}
I_{2}\ll U^{-20}.
\end{equation*}%
Therefore,%
\begin{equation}
\int_{-\frac{5\sqrt{L}}{A}}^{\frac{5\sqrt{L}}{A}}e^{-\pi A^{2}\theta
^{2}}j_{n\alpha }(\theta )d\theta \ll U^{-20}
\end{equation}%
for $\sqrt{n}>C_{0}^{2}L\sqrt{U}.$ If $b_{0}$ is (4.74), i.e.%
\begin{equation}
b_{0}=[C_{0}^{2}L],
\end{equation}%
then%
\begin{equation*}
\sum\limits_{(b_{0}+1/2)\sqrt{U}<\sqrt{n}\leq U^{10}}d(n)n^{\frac{-1}{4}%
}\sum\limits_{\alpha =0}^{2}\frac{\gamma _{\alpha }A}{(4\pi \sqrt{n}%
)^{\alpha }}\int_{-\frac{5\sqrt{L}}{A}}^{\frac{5\sqrt{L}}{A}}e^{-\pi
A^{2}\theta ^{2}}j_{n\alpha }(\theta )d\theta \ \ \ \ \ \ \ \ \ \ \ \ \ \ \
\
\end{equation*}%
\begin{equation*}
\ \ \ \ \ \ \ \ \ \ \ll \sum\limits_{(b_{0}+1/2)\sqrt{U}<\sqrt{n}\leq
U^{10}}d(n)n^{\frac{-1}{4}}U^{-20}\ll U^{-3}.
\end{equation*}%
This and (5.6) give
\begin{equation}
S(B)=2\sqrt{2}\sum\limits_{\sqrt{n}\leq (b_{0}+1/2)\sqrt{U}}d(n)n^{\frac{-1}{%
4}}\sum\limits_{\alpha =0}^{2}\frac{\gamma _{\alpha }A}{(4\pi \sqrt{n}%
)^{\alpha }}\int_{-\frac{5\sqrt{L}}{A}}^{\frac{5\sqrt{L}}{A}}e^{-\pi
A^{2}\theta ^{2}}j_{n\alpha }(\theta )d\theta +O_{\eta }(U^{-1}),
\end{equation}%
where $j_{n\alpha }(\theta )$ is (5.2). By (5.12) and (5.2),
\begin{equation}
S(B)=2\sqrt{2}\mathrm{Re}\int_{\xi _{1}}^{\xi _{2}}(T_{0}(B+\xi
)+T_{1}(B+\xi ))d\xi +O_{\eta }(U^{-1})
\end{equation}%
where $\xi _{1},\ \xi _{2}$ are (1.18),
\begin{equation*}
T_{0}(B+\xi )=\ \ \ \ \ \ \ \ \ \ \ \ \ \ \ \ \ \ \ \ \ \ \ \ \ \ \ \ \ \ \
\ \ \ \ \ \ \ \ \ \ \ \ \ \ \ \ \ \ \ \ \ \ \ \ \ \ \ \ \ \ \ \ \ \ \ \ \ \
\ \ \ \ \ \ \ \ \ \ \ \ \ \ \ \ \ \ \ \ \ \ \ \
\end{equation*}%
\begin{equation}
=e(1/8)\sum\limits_{\sqrt{n}\leq \sqrt{U}/2}d(n)n^{\frac{-1}{4}%
}\sum\limits_{\alpha =0}^{2}\frac{\gamma _{\alpha }Ai^{\alpha }}{(4\pi \sqrt{%
n})^{\alpha }}\int_{-\frac{5\sqrt{L}}{A}}^{\frac{5\sqrt{L}}{A}}e^{-\pi
A^{2}\theta ^{2}}\frac{e(2\sqrt{n}(B+\xi +\theta ))d\theta }{(B+\xi +\theta
)^{\alpha }},
\end{equation}%
\begin{equation*}
T_{1}(B+\xi )=\ \ \ \ \ \ \ \ \ \ \ \ \ \ \ \ \ \ \ \ \ \ \ \ \ \ \ \ \ \ \
\ \ \ \ \ \ \ \ \ \ \ \ \ \ \ \ \ \ \ \ \ \ \ \ \ \ \ \ \ \ \ \ \ \ \ \ \ \
\ \ \ \ \ \ \ \ \ \ \ \ \ \ \ \ \ \ \ \ \ \ \ \ \
\end{equation*}%
\begin{eqnarray}
&=&e(1/8)\sum\limits_{\sqrt{U}/2<\sqrt{n}\leq (b_{0}+1/2)\sqrt{U}}d(n)n^{%
\frac{-1}{4}}A\int_{-\frac{5\sqrt{L}}{A}}^{\frac{5\sqrt{L}}{A}}e^{-\pi
A^{2}\theta ^{2}}e(2\sqrt{n}(B+\xi +\theta ))d\theta \ \ \ \ \ \ \ \ \ \ \ \
\notag \\
&&+\delta _{1}(B+\xi )
\end{eqnarray}%
for $e(\frac{\alpha }{4})=i^{\alpha }$ with
\begin{equation*}
\delta _{1}(B+\xi )=e(1/8)\sum\limits_{\alpha =1}^{2}\frac{\gamma _{\alpha
}Ai^{\alpha }}{(4\pi \sqrt{n})^{\alpha }}\ \ \ \ \ \ \ \ \ \ \ \ \ \ \ \ \ \
\ \ \ \ \ \ \ \ \ \ \ \ \ \ \ \ \ \ \ \ \ \ \ \ \ \ \ \ \ \ \ \ \ \ \ \ \ \
\ \ \ \ \ \ \ \ \ \ \ \ \ \ \ \ \ \ \ \ \ \ \ \
\end{equation*}%
\begin{equation}
\ \ \ \ \ \ \ \times \int_{-\frac{5\sqrt{L}}{A}}^{\frac{5\sqrt{L}}{A}%
}e^{-\pi A^{2}\theta ^{2}}\cdot \sum\limits_{\sqrt{U}/2<\sqrt{n}\leq
(b_{0}+1/2)\sqrt{U}}\frac{d(n)n^{\frac{-1}{4}}e(2\sqrt{n}(B+\xi +\theta
))d\theta }{(\sqrt{n}(B+x+\theta ))^{\alpha }}.
\end{equation}%
Using Lemma 3.4 for the inner sum, we have
\begin{equation}
\delta _{1}(B+\xi )\ll U^{-1+\varepsilon }.
\end{equation}%
The inner integral of (5.15) is
\begin{equation*}
A\int_{-\infty }^{\infty }e^{-\pi A^{2}\theta ^{2}}e(2\sqrt{n}(\sqrt{U}+\xi
+\theta ))d\theta +O(e^{-L^{2}})=e^{\frac{-4\pi n}{A^{2}}}+O(e^{-L^{2}}).
\end{equation*}%
Hence
\begin{equation}
T_{1}(B+\xi )=e(1/8)\sum\limits_{\sqrt{U}/2<\sqrt{n}\leq (b_{0}+1/2)\sqrt{U}%
}d(n)n^{\frac{-1}{4}}e^{\frac{-4\pi n}{A^{2}}}e(2\sqrt{n}(B+\xi ))+O_{\eta
}(U^{-1+\varepsilon }).
\end{equation}%
Therefore,
\begin{equation}
S(B)=2\sqrt{2}\mathrm{Re}\int_{\xi _{1}}^{\xi _{2}}(T_{0}(B+\xi
)+T_{1}(B+\xi ))d\xi +O_{\eta }(U^{-1+\varepsilon }),
\end{equation}%
where $T_{0}(B+\xi )$ is (5.14), $T_{1}(B+\xi )$ is (5.18), $\xi _{1},\ \xi
_{2}$ are (1.18).

\ \

\section{EVALUATION of $T_{1}(B+\protect\xi)$}

We are going to evaluate (5.18). For $0<h\leq h_{0}=1/U$, $U=[X]$
\begin{equation*}
0<(\frac{\sqrt{U}}{2}+h)^{2}-\frac{U}{4}=\sqrt{U}h+h^{2}\leq \frac{1}{L},
\end{equation*}%
\begin{equation*}
0\leq \{\frac{U}{4}\}\leq \frac{3}{4}.
\end{equation*}%
Hence there doesn't exist any integer in the interval $((\frac{\sqrt{U}}{2}%
)^{2},(\frac{\sqrt{U}}{2}+h)^{2}]$. So that
\begin{equation}
\sum\limits_{\sqrt{U}/2<\sqrt{n}\leq \sqrt{U}/2+h}=0,\text{ for}\ 0<h\leq
h_{0}=1/U.
\end{equation}%
In the same way,
\begin{equation}
\sum\limits_{(b_{0}+1/2)\sqrt{U}<\sqrt{n}\leq (b_{0}+1/2)\sqrt{U}+h}=0.
\end{equation}%
It follows from (5.18) that
\begin{equation}
T_{1}(B+\xi )=\frac{e(1/8)}{h_{0}}\int_{0}^{h_{0}}dh\sum\limits_{\sqrt{U}/2<%
\sqrt{n}\leq (b_{0}+1/2)\sqrt{U}+h}d(n)n^{\frac{-1}{4}}e^{\frac{-4\pi n}{%
A^{2}}}e(2\sqrt{n}(B+\xi ))+O_{\eta }(U^{-1+\varepsilon }),
\end{equation}%
where $h_{0}=1/U$. By Lemma 2.3
\begin{eqnarray}
T_{1}(B+\xi ) &=&\sum\limits_{n\leq N}d(n)\frac{e(1/8)}{h_{0}}%
\int_{0}^{h_{0}}dh\int_{(\sqrt{U}/2+h)^{2}}^{((b_{0}+1/2)\sqrt{U}+h)^{2}}x^{%
\frac{-1}{4}}e^{\frac{-4\pi x}{A^{2}}}e(2\sqrt{x}(B+\xi ))\Theta (nx)dx
\notag \\
&&+T_{1l}(B+\xi )+O(\varepsilon _{1})
\end{eqnarray}%
for any $\varepsilon _{1}>0$, as $N\geq N(\varepsilon _{1})$, where $\Theta
(nx)$ is (2.12) and
\begin{equation*}
T_{1l}(B+\xi )=\frac{e(1/8)}{h_{0}}\int_{0}^{h_{0}}dh\int_{(\sqrt{U}%
/2+h)^{2}}^{((b_{0}+1/2)\sqrt{U}+h)^{2}}x^{\frac{-1}{4}}e^{\frac{-4\pi x}{%
A^{2}}}e(2\sqrt{x}(B+\xi ))(\log x+2\gamma )dx
\end{equation*}%
\begin{equation}
=\frac{2e(1/8)}{h_{0}}\int_{0}^{h_{0}}dh\int_{\sqrt{U}/2+h}^{(b_{0}+1/2)%
\sqrt{U}+h}x^{\frac{1}{2}}e^{\frac{-4\pi x^{2}}{A^{2}}}e(2x(B+\xi ))(\log
x^{2}+2\gamma )dx.
\end{equation}%
Changing the inner integral line of (6.5), we have%
\begin{equation*}
T_{1l}(B+\xi )=\frac{2e(1/8)}{h_{0}}\ \ \ \ \ \ \ \ \ \ \ \ \ \ \ \ \ \ \ \
\ \ \ \ \ \ \ \ \ \ \ \ \ \ \ \ \ \ \ \ \ \ \ \ \ \ \ \ \ \ \ \ \ \ \ \ \ \
\ \ \ \ \ \ \ \ \
\end{equation*}%
\begin{equation*}
\times \int_{0}^{h_{0}}dh(\int_{\sqrt{U}/2+h}^{\sqrt{U}/2+h+i}+\int_{\sqrt{U}%
/2+h+i}^{(b_{0}+1/2)\sqrt{U}+h+i}+\int_{(b_{0}+1/2)\sqrt{U}+h+i}^{(b_{0}+1/2)%
\sqrt{U}+h})e(2z(B+\xi ))f_{l}(z)dz
\end{equation*}%
\begin{equation}
=\frac{2e(1/8)}{h_{0}}\int_{0}^{h_{0}}dh(I_{1h}(B+\xi )+I_{2h}(B+\xi
)+I_{3h}(B+\xi )),\ \ \ \ \mathrm{say} \ \ \ \ \ \ \ \ \ \
\end{equation}%
where
\begin{equation}
f_{l}(z)=z^{1/2}e^{\frac{-4\pi z^{2}}{A}}(\log z^{2}+2\gamma ).
\end{equation}%
Furthermore,
\begin{equation*}
I_{1h}(B+\xi ) =\int_{0}^{i}e(2(B+\xi )(\sqrt{U}/2+h+z))f_{l}(\sqrt{U}%
/2+h+z)dz \ \ \ \ \ \ \ \ \ \ \ \ \ \ \ \ \
\end{equation*}
\begin{equation*}
=\int_{0}^{i}e(2(B^{\prime }+\xi )(\sqrt{U}/2+h+z))f_{l}(\sqrt{U}%
/2+h+z)e(\eta (\sqrt{U}+2h+2z))dz,
\end{equation*}
where $B=B^{\prime }+\eta $. By (1.14) and (1.15), $e(\eta \sqrt{U})=1$,
hence, by Lemma 3.6
\begin{equation*}
e(\eta (\sqrt{U}+2h+2z))=e(\eta (2h+2z))=O_{\eta }(U^{-99}).
\end{equation*}%
Therefore,%
\begin{equation*}
I_{1h}(B+\xi )\ll _{\eta }U^{-99}\int_{0}^{1}|f(\sqrt{U}/2+h+i\rho
)e(2(B^{\prime }+\xi )i\rho )|d\rho \ll _{\eta }U^{-20}.
\end{equation*}%
In the same way,
\begin{equation*}
I_{3h}(B+\xi )\ll _{\eta }U^{-20}.
\end{equation*}%
When $z\in I_{2h}(B+\xi )$, $z=x+i$. Thus,
\begin{equation*}
|e(2z(B+\xi ))f_{l}(z)|\ll \sqrt{U}L|e(2(B+\xi )i)|=\sqrt{U}Le^{-4\pi (B+\xi
)}.
\end{equation*}%
Hence
\begin{equation*}
I_{2h}(B+\xi )\ll \sqrt{U}Le^{-4\pi (B+\xi )}\int_{\sqrt{U}/2+h}^{(b_{0}+1/2)%
\sqrt{U}+h}dx\ll U^{-21}.
\end{equation*}%
Therefore, by (6.6),
\begin{equation}
T_{1l}(B+\xi )=O_{\eta }(U^{-19}).
\end{equation}%
Putting $x=x_{1}^{2}$ in the integral (6.4), taking $\alpha _{0}=2$ in
(2.12), noticing $2\cos 2\pi x=e(x)+e(-x),e(\frac{\alpha }{4})=i^{\alpha }$,
by (6.4) and (6.8) we obtain
\begin{equation}
T_{1}(B+\xi )=T_{1}^{+}(B+\xi )+T_{1}^{-}(B+\xi )+O_{\eta
}(U^{-19})+O(\varepsilon _{1}),
\end{equation}%
where
\begin{equation}
T_{1}^{\pm }(B+\xi )=\sqrt{2}\sum\limits_{n\leq
N}d(n)n^{-1/4}\sum\limits_{\alpha =0}^{2}\frac{\gamma _{\alpha }(\pm
i)^{\alpha }}{(4\pi \sqrt{n})^{\alpha }}j_{n\alpha }^{\pm }(B+\xi ),
\end{equation}%
\begin{equation}
j_{n\alpha }^{\pm }(B+\xi )=\frac{e(1/8\pm 1/8)}{h_{0}}\int_{0}^{h_{0}}dh%
\int_{\sqrt{U}/2+h}^{(b_{0}+1/2)\sqrt{U}+h}\frac{e^{\frac{-4\pi x^{2}}{A^{2}}%
}e(2x(B+\xi \pm \sqrt{n}))dx}{x^{\alpha }}.
\end{equation}%
If $\sqrt{n}>2\sqrt{U}$, then
\begin{eqnarray}
j_{n\alpha }^{\pm }(B+\xi ) &=&\frac{e(1/8\pm 1/8)}{4\pi ih_{0}}%
\int_{0}^{h_{0}}dh\frac{e^{\frac{-4\pi (x+h)^{2}}{A^{2}}}e(2(x+h)(B+\xi \pm
\sqrt{n}))}{(B+\xi \pm \sqrt{n})(x+h)^{\alpha }}\mbox {\LARGE \textbar
}_{x=\sqrt{U}/2}^{x=(b_{0}+1/2)\sqrt{U}}  \notag \\
&&+O(\frac{1}{(B+\xi \pm \sqrt{n})^{2}})  \notag \\
&\ll &\frac{1}{h_{0}(B+\xi \pm \sqrt{n})^{2}}\ll \frac{U}{n}
\end{eqnarray}%
Since
\begin{equation*}
\sum\limits_{U^{20}<n\leq N}d(n)n^{-1/4-\alpha /2}\frac{U}{n}\ll U^{-3}
\end{equation*}%
for $\alpha \geq 0$, then by (6.10),
\begin{equation}
T_{1}^{\pm }(B+\xi )=\sqrt{2}\sum\limits_{\sqrt{n}\leq
U^{10}}d(n)n^{-1/4}\sum\limits_{\alpha =0}^{2}\frac{\gamma _{\alpha }(\pm
i)^{\alpha }}{(4\pi \sqrt{n})^{\alpha }}j_{n\alpha }^{\pm }(B+\xi
)+O(U^{-3}).
\end{equation}%
Let $z=x+iy$, we see that
\begin{equation}
|e(2(B+\xi \pm \sqrt{n}))|=e(2iy(B+\xi \pm \sqrt{n}))=e^{-4\pi y(B+\xi \pm
\sqrt{n})}.
\end{equation}%
Using the path (6.6), repeating the proof of (6.8), we obtain
\begin{equation}
j_{n\alpha }^{+}(B+\xi )=O_{\eta }(U^{-19}),\ \mathrm{for}\ n\geq 1.
\end{equation}%
In the same way, if $\sqrt{n}\leq \sqrt{U}-L^{2}$, then the right side of
(6.14) is
\begin{equation*}
e^{-4\pi y(B+\xi -\sqrt{n})}\ll e^{-4\pi yL^{2}},\ \ y\geq 0.
\end{equation*}%
The path (6.6) gives also
\begin{equation}
j_{n\alpha }^{-}(B+\xi )=O_{\eta }(U^{-19}),\ \mathrm{for}\ \sqrt{n}\leq
\sqrt{U}-L^{2}.
\end{equation}%
If $\sqrt{n}>\sqrt{U}+L^{2}$, then moving the inner integral line of $%
j_{n\alpha }^{-}(B+\xi )$ to
\begin{equation*}
\int_{\sqrt{U}/2+h}^{\sqrt{U}/2+h-i}+\int_{\sqrt{U}/2+h-i}^{(b_{0}+1/2)\sqrt{%
U}+h-i}+\int_{(b_{0}+1/2)\sqrt{U}+h-i}^{(b_{0}+1/2)\sqrt{U}+h},
\end{equation*}%
noticing
\begin{equation*}
|e(2(B+\xi -\sqrt{n})z)|=e^{-4\pi |y(B+\xi -\sqrt{n})|}
\end{equation*}%
for $z=x+iy$, $y\leq 0$, we have also
\begin{equation}
j_{n\alpha }^{-}(B+\xi )\ll _{\eta }U^{-19},\ \mathrm{for}\ \sqrt{n}\leq
\sqrt{U}-L^{2}.
\end{equation}%
It follows from (6.9), (6.13), (6.15), (6.16) and (6.17) that
\begin{equation*}
T_{1}(B+\xi )=\sqrt{2}\sum\limits_{-L^{2}<\sqrt{n}-\sqrt{U}\leq
L^{2}}d(n)n^{-1/4}\sum\limits_{\alpha =0}^{2}\frac{\gamma _{\alpha
}(-i)^{\alpha }}{(4\pi \sqrt{n})^{\alpha }}j_{n\alpha }^{-}(B+\xi )
\end{equation*}%
\begin{equation}
+\delta +O_{\eta }(U^{-2})\ \ \ (\varepsilon _{1}\rightarrow 0_{+}),\ \ \ \
\ \ \ \ \ \ \ \ \
\end{equation}%
where $j_{n\alpha }^{-}(B+\xi )$ is (6.11) and
\begin{equation}
\delta \ll _{\eta }\sum\limits_{\sqrt{n}\leq
U^{10}}d(n)n^{-1/4}\sum\limits_{\alpha =0}^{2}\frac{U^{-19}}{(\sqrt{n}%
)^{\alpha }}\ll _{\eta }U^{-2}.
\end{equation}

Next, we evaluate $j_{n\alpha }^{-}(B+\xi )$, $\alpha =1,2.$ Using the
integration by parts, It is easy to know that
\begin{equation*}
j_{n\alpha }^{-}(B+\xi )\ll \frac{1}{(\sqrt{U})^{\alpha }|B+\xi -\sqrt{n}|},
\end{equation*}%
or,
\begin{eqnarray*}
j_{n\alpha }^{-}(B+\xi ) &\ll &\frac{1}{h_{0}}\int_{0}^{h_{0}}dh\int_{\sqrt{U%
}/2+h}^{(b_{0}+1/2)\sqrt{U}+h}\frac{dx}{(\sqrt{U})^{\alpha }} \\
&\ll &\frac{b_{0}\sqrt{U}}{(\sqrt{U})^{\alpha }}\ll L(\sqrt{U})^{1-\alpha }.
\end{eqnarray*}%
Hence
\begin{equation*}
j_{n\alpha }^{-}(B+\xi )\ll \min (\frac{L}{(\sqrt{U})^{1-\alpha }},\frac{1}{(%
\sqrt{U})^{\alpha }|B+\xi -\sqrt{n}|}).
\end{equation*}%
Moreover,
\begin{equation*}
\sum\limits_{\alpha =1}^{2}\frac{\gamma _{\alpha }(-i)^{\alpha }}{(4\pi
)^{\alpha }}\sum\limits_{-L^{2}<\sqrt{n}-\sqrt{U}\leq
L^{2}}d(n)n^{-1/4-\alpha /2}j_{n\alpha }^{-}(B+\xi )\ \ \ \ \ \ \ \ \ \ \ \
\ \ \ \ \ \ \ \ \ \ \ \ \
\end{equation*}%
\begin{eqnarray}
&\ll &L\sum\limits_{|B+\xi -\sqrt{n}|\leq \frac{1}{\sqrt{U}}%
}d(n)n^{-3/4}+\sum\limits_{|B+\xi -\sqrt{n}|\geq \frac{1}{\sqrt{U}},|\sqrt{n}%
-\sqrt{U}|\leq L^{2}}\frac{d(n)n^{-3/4}}{|(B+\xi )^{2}-n|}  \notag \\
&\ll &U^{-\frac{3}{4}+\varepsilon },\ 0<\varepsilon =O(\frac{1}{\log
L})
\end{eqnarray}%
This and (6.18), (6.19), (6.11) give
\begin{equation*}
T_{1}(B+\xi )=\sqrt{2}\sum\limits_{-L^{2}<\sqrt{n}-\sqrt{U}\leq
L^{2}}d(n)n^{-1/4}j_{n_{0}}^{-}(B+\xi )+\delta +O_{\eta }(U^{-\frac{3}{4}%
+\varepsilon })\ \ \ \ \ \ \ \ \ \ \ \ \ \ \ \ \ \ \ \ \ \ \ \ \ \
\end{equation*}%
\begin{equation*}
=\frac{\sqrt{2}}{h_{0}}\int_{0}^{h_{0}}dh\sum\limits_{-L^{2}<\sqrt{n}-\sqrt{U%
}\leq L^{2}}d(n)n^{-1/4}\int_{\sqrt{U}/2+h}^{(b_{0}+1/2)\sqrt{U}+h}e^{\frac{%
-4\pi x^{2}}{A^{2}}}e(2x(B+\xi -\sqrt{n}))dx
\end{equation*}%
\begin{equation}
+O_{\eta }(U^{-\frac{3}{4}+\varepsilon }),\ h_{0}=\frac{1}{U},\
0<\varepsilon =O(\frac{1}{\log L})\ \ \ \ \ \ \ \ \ \ \ \ \ \ \ \ \
\ \ \ \ \ \ \ \ \ \ \ \ \ \ \ \ \ \ \ \ \ \ \ \ .
\end{equation}%
By Lemma 3.4,
\begin{equation*}
\sum\limits_{-L^{2}<\sqrt{n}-\sqrt{U}\leq L^{2}}d(n)n^{-1/4}e(-2\sqrt{n}%
x)\ll x^{1/4+\varepsilon }\ll U^{1/4+\varepsilon }.
\end{equation*}%
Hence
\begin{equation*}
\frac{1}{h_{0}}\int_{0}^{h_{0}}dh(\int_{\sqrt{U}/2+h}^{\sqrt{U}%
/2}+\int_{(b_{0}+1/2)\sqrt{U}}^{(b_{0}+1/2)\sqrt{U}+h})e^{\frac{-4\pi x^{2}}{%
A^{2}}}\sum\limits_{-L^{2}<\sqrt{n}-\sqrt{U}\leq
L^{2}}d(n)n^{-1/4}e(2x(B+\xi -\sqrt{n}))
\end{equation*}%
\begin{equation}
\ll \frac{U^{1/4+\varepsilon }}{h_{0}}\int_{0}^{h_{0}}hdh\ll U^{-\frac{3}{4}%
+\varepsilon },\ 0<\varepsilon =O(\frac{1}{\log L}). \ \ \ \ \ \ \ \
\ \ \ \ \ \ \ \ \ \ \ \ \ \ \ \ \ \ \ \ \ \ \ \ \ \ \ \ \ \ \ \ \ \
\ \ \ \ \
\end{equation}%
It follows from (6.21) and (6.22) that
\begin{equation*}
T_{1}(B+\xi )\ \ \ \ \ \ \ \ \ \ \ \ \ \ \ \ \ \ \ \ \ \ \ \ \ \ \ \ \ \ \ \
\ \ \ \ \ \ \ \ \ \ \ \ \ \ \ \ \ \ \ \ \ \ \ \ \ \ \ \ \ \ \ \ \ \ \ \ \ \
\ \ \ \ \ \ \ \ \ \ \ \ \ \ \ \ \ \ \ \ \ \ \ \ \ \ \ \ \ \ \ \ \ \ \ \ \ \
\ \ \ \ \ \ \ \ \ \
\end{equation*}%
\begin{equation}
=\sqrt{2}\sum\limits_{-L^{2}<\sqrt{n}-\sqrt{U}\leq L^{2}}d(n)n^{-1/4}\int_{%
\sqrt{U}/2}^{(b_{0}+1/2)\sqrt{U}}e^{\frac{-4\pi (x)^{2}}{A^{2}}}e(2x(B+\xi -%
\sqrt{n}))dx+O_{\eta }(U^{-\frac{3}{4}+\varepsilon }),
\end{equation}%
where $b_{0}=[C_{0}^{2}L]$, $0<\varepsilon =O(\frac{1}{\log L})$.

\ \

\section{EVALUATION of SUM $\sum\limits_{k=0}^{m}T_{1}(B+\protect\xi)$}

Let
\begin{equation}
B_{1}=\sqrt{U}+\frac{j}{2\sqrt{U}}+\eta +\xi ,
\end{equation}%
then by (1.10),
\begin{equation}
B+\xi =B_{1}+\frac{k}{2\sqrt{U}}.
\end{equation}%
Moreover, denote
\begin{equation}
Q_{1}(B_{1})=\sum\limits_{k=0}^{m}T_{1}(B+\xi
)=\sum\limits_{k=0}^{m}T_{1}(B_{1}+\frac{k}{2\sqrt{U}}),
\end{equation}%
where $T_{1}(B)$ is (6.27). In this section we are going to prove that
\begin{eqnarray}
\mathrm{Re}(Q_{1}(B_{1})) &=&\mathrm{Re}\sum\limits_{k=0}^{m}T_{1}(B_{1}+%
\frac{k}{2\sqrt{U}})  \notag \\
&=&-\frac{\sqrt{U}}{\sqrt{2}}\sum\limits_{0<\sqrt{n}-B_{1}\leq \frac{m}{2%
\sqrt{U}}}d(n)n^{-1/4}+\mathrm{Re}\ \delta _{Q_{1}}(B_{1}),
\end{eqnarray}%
where $\delta _{Q_{1}}(B_{1})\ll U^{\frac{1}{4}+\varepsilon }$ , $%
0<\varepsilon =O(\frac{1}{\log L})=O(\frac{1}{\log \log U})$.

\ \

\textit{Proof} of (7.4): By (6.23), (7.1), (7.2 ) and (7.3),
\begin{eqnarray}
Q_{1}(B_{1}) &=&\sqrt{2}\sum\limits_{-L^{2}<\sqrt{n}-\sqrt{U}\leq
L^{2}}d(n)n^{-1/4}\int_{\frac{\sqrt{U}}{2}}^{(b_{0}+1/2)\sqrt{U}}e^{-\frac{%
4\pi x^{2}}{A^{2}}}e(2(B_{1}-\sqrt{n})x)\rho (x)dx  \notag \\
&&+O_{\eta }(U^{-\frac{1}{4}+\varepsilon })
\end{eqnarray}%
for $m\ll \sqrt{U}L^{-2}$, where $0<\varepsilon =O(\frac{1}{\log L})$ and%
\begin{equation}
\rho (x)=\sum\limits_{k=0}^{m}e(\frac{xk}{\sqrt{U}}).
\end{equation}%
By Lemma 3.3,
\begin{eqnarray}
\rho (x) &=&\frac{1}{2}+\frac{1}{2}e(\frac{xm}{\sqrt{U}})+O(U^{-9})+\sum%
\limits_{b=-U^{10}}^{U^{10}}\int_{0}^{m}e((\frac{x}{\sqrt{U}}-b)u)du  \notag
\\
&=&\rho _{0}(x)+\rho _{1}(x)+\rho _{2}(x)+O(U^{-9}),
\end{eqnarray}%
where
\begin{eqnarray}
\rho _{0}(x) &=&\sum\limits_{b=1}^{b_{0}}\int_{0}^{m}e((\frac{x}{\sqrt{U}}%
-b)u)du, \\
\rho _{1}(x) &=&\frac{1}{2}+\frac{1}{2}e(\frac{xm}{\sqrt{U}})=\frac{1}{2}%
\sum\limits_{\nu =o,m}e(\frac{x\nu }{\sqrt{U}}), \\
\rho _{2}(x) &=&\sum_{b}\ ^{\prime }\int_{0}^{m}e((\frac{x}{\sqrt{U}}%
-b)u)du=\sum_{b}\ ^{\prime }\frac{e(\frac{xu}{\sqrt{U}})}{2\pi i(\frac{x}{%
\sqrt{U}}-b)}\mbox {\LARGE \textbar }_{u=0}^{u=m}
\end{eqnarray}%
for $e(-b\nu )=1,\ \nu =0,m$, and $\sum^{\prime }$ is for $U^{-10}\leq b\leq
0$, or $b_{0}+1\leq b\leq U^{10}$. It follows from (7.5) and (7.7) that
\begin{equation}
Q_{1}(B_{1})=Q_{10}(B_{1})+Q_{11}(B_{1})+Q_{12}(B_{1})+O_{\eta }(U^{-1}),
\end{equation}%
where
\begin{equation}
Q_{1j}(B_{1})=\sqrt{2}\sum\limits_{-L^{2}<\sqrt{n}-\sqrt{U}\leq
L^{2}}d(n)n^{-1/4}J_{j}(n),\ j=0,1,2,
\end{equation}%
\begin{equation}
J_{j}(n)=\int_{C}e^{-\frac{4\pi z^{2}}{A^{2}}}e(2(B_{1}-\sqrt{n})z)\rho
_{j}(z)dz=\int_{C_{A}^{+}}=\int_{C_{A}^{-}},
\end{equation}%
$C$ is from $\frac{\sqrt{U}}{2}$ to $(b_{0}+1/2)\sqrt{U}$, $C_{A}^{\pm }$ is
the broken line joining the points $\frac{\sqrt{U}}{2},\frac{\sqrt{U}}{2}\pm
iA/2$, $(b_{0}+1/2)\sqrt{U}\pm iA/2\ and\ (b_{0}+1/2)\sqrt{U}.$

\begin{picture}(400,200)
\put(100,60){\line(1,0){200}} \put(100,150){\line(1,0){200}}
\put(100,60){\line(0,1){90}} \put(300,60){\line(0,1){90}}
\put(100,105){\circle*{4}} \put(300,105){\circle*{4}}
\put(100,105){\vector(0,1){30}} \put(100,105){\vector(0,-1){30}}
\put(300,150){\vector(0,-1){40}}
\put(300,60){\vector(0,1){40}}\put(85,105){\makebox(2,2)[c]{$\frac{\sqrt{U}}{2}$}}\put(330,105){\makebox(7,2)[c]{$(b_0+\frac{1}{2})\sqrt{U}$}}
\put(200,50){\makebox(2,2)[c]{$C^-_A$}}\put(200,160){\makebox(2,2)[c]{$C^+_A$}}\put(200,20){\makebox(5,1)[c]{Fig.
7.1}}
\end{picture}

By (7.10) and (7.13),
\begin{equation}
J_{2}(n)=\sum_{b}\ ^{\prime }\int_{C}\frac{e^{-\frac{4\pi z^{2}}{A^{2}}%
}e(2(B_{1}+\frac{u}{2\sqrt{U}}-\sqrt{n})z)dz}{2\pi i(\frac{z}{\sqrt{U}}-b)}%
\mbox {\LARGE \textbar
}_{u=0}^{u=m}.
\end{equation}%
It is obvious that the integrand is analytic for $\frac{\sqrt{U}}{2}\leq
\mathrm{Re}\ z\leq (b_{0}+1/2)\sqrt{U}$, as $U^{-10}\leq b\leq 0$, or $%
b_{0}+1\leq b\leq U^{10}$. For $z=x+iy$, $\frac{\sqrt{U}}{2}\leq x\leq
(b_{0}+1/2)\sqrt{U}$, $-A/2\leq y\leq A/2$, we have
\begin{equation}
|e^{-\frac{4\pi z^{2}}{A^{2}}}|=e^{-\frac{4\pi (x^{2}-y^{2})}{A^{2}}}\ll e^{-%
\frac{4\pi x^{2}}{A^{2}}},
\end{equation}%
\begin{equation}
|e(2(B_{1}+\frac{u}{2\sqrt{U}}-\sqrt{n})z)|=e^{-4\pi y(B_{1}+\frac{u}{2\sqrt{%
U}}-\sqrt{n})}.
\end{equation}%
If $\sqrt{n}\leq B_{1}+\frac{u}{2\sqrt{U}}$, then moving the integral line $%
C $ to $C_{A}^{+}$, we see that
\begin{eqnarray}
\int_{C} =\int_{C_{A}^{+}}&=&(\int_{\frac{\sqrt{U}}{2}}^{\frac{\sqrt{U}}{2}+%
\frac{iA}{2}}+\int_{\frac{\sqrt{U}}{2}+\frac{iA}{2}}^{(b_{0}+1/2)\sqrt{U}+%
\frac{iA}{2}}+\int_{(b_{0}+1/2)\sqrt{U}+\frac{iA}{2}}^{(b_{0}+1/2)\sqrt{U}})
\notag \\
&&\times \frac{e^{-\frac{4\pi z^{2}}{A^{2}}}e(2(B_{1}+\frac{u}{2\sqrt{U}}-%
\sqrt{n})z)dz}{2\pi i(\frac{z}{\sqrt{U}}-b)}  \notag \\
&=&I_{1}+I_{2}+I_{3},\ \ \ \mathrm{say}
\end{eqnarray}%
Clearly, for $b\leq 0$ or $b\geq b_{0}+1$ and
\begin{equation}
\frac{1}{\frac{z}{\sqrt{U}}-b}\ll \frac{1}{|b|+1},\ z\in C_{A}^{\pm },
\end{equation}%
we have
\begin{eqnarray*}
I_{1} &\ll &\int_{0}^{\frac{A}{2}}\frac{e^{-4\pi y(B_{1}+\frac{u}{2\sqrt{U}}-%
\sqrt{n})}dy}{|b|+1} \\
&\ll &\frac{1}{(|b|+1)(B_{1}+\frac{u}{2\sqrt{U}}-\sqrt{n})}=\frac{1}{%
(|b|+1)|B_{1}+\frac{u}{2\sqrt{U}}-\sqrt{n}|}.
\end{eqnarray*}%
In the same way
\begin{equation*}
I_{3}\ll \frac{1}{(|b|+1)|B_{1}+\frac{u}{2\sqrt{U}}-\sqrt{n}|}.
\end{equation*}%
By (7.15)-(7.18),
\begin{eqnarray*}
I_{2} &\ll &\frac{1}{|b|+1}e^{-2\pi A(B_{1}+\frac{u}{2\sqrt{U}}-\sqrt{n}%
)}\int_{\frac{\sqrt{U}}{2}}^{(b_{0}+1/2)\sqrt{U}}e^{-\frac{4\pi x^{2}}{A^{2}}%
}dx \\
&\ll &\frac{A}{|b|+1}e^{-2\pi A(B_{1}+\frac{u}{2\sqrt{U}}-\sqrt{n})} \\
&=&\frac{1}{(|b|+1)|B_{1}+\frac{u}{2\sqrt{U}}-\sqrt{n}|}(A|B_{1}+\frac{u}{2%
\sqrt{U}}-\sqrt{n}|e^{-2\pi A|B_{1}+\frac{u}{2\sqrt{U}}-\sqrt{n}|}) \\
&\ll &\frac{1}{(|b|+1)|B_{1}+\frac{u}{2\sqrt{U}}-\sqrt{n}|}.
\end{eqnarray*}%
Hence the integral (7.17) is
\begin{equation}
\int_{C}\ll \frac{1}{(|b|+1)|B_{1}+\frac{u}{2\sqrt{U}}-\sqrt{n}|}
\end{equation}%
for $\sqrt{n}\leq B_{1}+\frac{u}{2\sqrt{U}}$. If $\sqrt{n}>B_{1}+\frac{u}{2%
\sqrt{U}}$, then moving the path $C$ to $C_{A}^{-}$, we have also (7.19).
Hence, by (7.14) we obtain
\begin{eqnarray}
J_{2}(n) &\ll &\sum\limits_{u=0,m}\frac{1}{|B_{1}+\frac{u}{2\sqrt{U}}-\sqrt{n%
}|}\sum_{b}\ ^{\prime }\frac{1}{(|b|+1)}  \notag \\
&\ll &L\sum\limits_{u=0,m}\frac{1}{|B_{1}+\frac{u}{2\sqrt{U}}-\sqrt{n}|}.
\end{eqnarray}%
In the same way,
\begin{equation}
J_{1}\ll \sum\limits_{\nu =0,1}\frac{1}{|B_{1}+\frac{\nu }{2\sqrt{U}}-\sqrt{n%
}|}=\sum\limits_{u=0,m}\frac{1}{|B_{1}+\frac{u}{2\sqrt{U}}-\sqrt{n}|}.
\end{equation}%
Therefore, it follows from (7.12),(7.20), (7.21) and Lemma 3.1 that
\begin{eqnarray}
Q_{11}(B_{1})+Q_{12}(B_{1}) \ll \sum\limits_{u=0,m}\sum\limits_{-L^{2}<\sqrt{%
n}-\sqrt{U}\leq L^{2}}\frac{Ld(n)n^{-1/4}}{|B_{1}+\frac{u}{2\sqrt{U}}-\sqrt{n%
}|} &  \notag \\
\ll U^{\frac{1}{4}+\varepsilon }\sum\limits_{u=0,m}\sum\limits_{-L^{2}\sqrt{U%
}<n-U\leq L^{2}\sqrt{U}}\frac{1}{|n-(B_{1}+\frac{u}{2\sqrt{U}})^{2}|} &
\notag \\
\ll U^{\frac{1}{4}+\varepsilon }(1+\sum\limits_{u=0,m}(\frac{1}{\{(B_{1}+%
\frac{u}{2\sqrt{U}})^{2}\}}+\frac{1}{1-\{(B_{1}+\frac{u}{2\sqrt{U}})^{2}\}}%
,)),&
\end{eqnarray}%
where $\varepsilon =O((\log L)^{-1})$. Moreover,
\begin{eqnarray*}
\{(B_{1}+\frac{u}{2\sqrt{U}})^{2}\} &=&\{(\sqrt{U}+\frac{j+u}{2\sqrt{U}}%
+\eta +\xi )^{2}\} \\
&=&\{U+j+u+2\sqrt{U}(\eta +\xi )+(\frac{j+u}{2\sqrt{U}}+\eta +\xi )^{2}\} \\
&=&\{2\sqrt{U}\xi +O(L^{-1})\}
\end{eqnarray*}%
for $u=0,m,\ 2\sqrt{U}\eta =2k_{1}$, $j,u\ll \sqrt{U}L^{-2},\ \eta \ll
L^{-1} $. Hence by (1.18): $\frac{1}{8}\leq 2\sqrt{U}\xi \leq \frac{3}{8},$
we have
\begin{equation}
\frac{1}{9}\leq \{(B_{1}+\frac{u}{2\sqrt{U}})^{2}\}\leq \frac{4}{9}.
\end{equation}%
Furthermore, by (7.22) and (7.23), we have
\begin{equation}
Q_{11}(B_{1})+Q_{12}(B_{1})=\delta _{Q_{11}}\ll U^{\frac{1}{4}+\varepsilon }.
\end{equation}%
This and (7.11) give
\begin{eqnarray}
Q_{1}(B_{1}) &=&Q_{10}(B_{1})+\delta _{Q_{11}}  \notag \\
&=&\sqrt{2}\sum\limits_{-L^{2}<\sqrt{n}-\sqrt{U}\leq
L^{2}}d(n)n^{-1/4}J_{0}(n)+\delta _{Q_{11}}  \notag \\
&=&\sqrt{2}\sum\limits_{b=1}^{b_{0}}\int_{0}^{m}du\sum\limits_{-L^{2}<\sqrt{n%
}-\sqrt{U}\leq L^{2}}d(n)n^{-1/4}\int_{C}f_{nb}(z,u)dz+\delta _{Q_{11}},
\end{eqnarray}%
where $C$ is from $\frac{\sqrt{U}}{2}$ to $(b_{0}+\frac{1}{2})\sqrt{U}$, $%
\delta _{Q_{11}}\ll U^{\frac{1}{4}+\varepsilon }$ and
\begin{equation}
f_{nb}(z,u)=e^{-\frac{4\pi z^{2}}{A^{2}}}e(2(B_{1}-\sqrt{n})z+(\frac{z}{%
\sqrt{U}}-b)u).
\end{equation}%
Clearly,
\begin{eqnarray}
\int_{0}^{m}du\sum\limits_{-L^{2}<\sqrt{n}-\sqrt{U}\leq L^{2}}
&=&\int_{0}^{m}du(\sum\limits_{\sqrt{U}-L^{2}<\sqrt{n}\leq B_{1}+\frac{u}{2%
\sqrt{U}}}+\sum\limits_{B_{1}+\frac{u}{2\sqrt{U}}<\sqrt{n}\leq \sqrt{U}%
+L^{2}})  \notag \\
&=&\sum\limits_{\sqrt{U}-L^{2}<\sqrt{n}\leq B_{1}+\frac{m}{2\sqrt{U}}%
}\int_{M_{1}(n)}^{m}+\sum\limits_{B_{1}<\sqrt{n}\leq \sqrt{U}%
+L^{2}}\int_{0}^{M_{2}(n)},
\end{eqnarray}%
where
\begin{equation}
M_{1}(n)=\max (0,2\sqrt{U}(\sqrt{n}-B_{1})),
\end{equation}%
\begin{equation}
M_{2}(n)=\min (m,2\sqrt{U}(\sqrt{n}-B_{1})).
\end{equation}%
Using
\begin{equation*}
\int_{C}f_{nb}(z,u)dz=\int_{C_{A}^{+}}=\int_{C_{A}^{-}},
\end{equation*}%
where $C_{A}^{\pm }$ be as Fig 7.1, by (7.25) and (7.27) we have%
\begin{equation}
Q_{1}(B_{1})=\sqrt{2}\sum\limits_{b=1}^{b_{0}}(w^{+}(b)+w^{-}(b))+\delta
_{Q_{11}},
\end{equation}%
where
\begin{equation*}
\delta _{Q_{11}}=O(U^{\frac{1}{4}+\varepsilon }),\ \varepsilon =O((\log
L)^{-1}),
\end{equation*}%
\begin{eqnarray}
w^{+}(b) =\sum\limits_{\sqrt{U}-L^{2}<\sqrt{n}\leq B_{1}+\frac{m}{2\sqrt{U}}%
}d(n)n^{-1/4}\int_{C_{A}^{+}}dz\int_{M_{1}(n)}^{m}f_{nb}(z,u)du \ \ \ \ \ \
\ \ \ \ \ \ \ \ \ \ \ \ \ \ \ \ \ \   \notag \\
=\sum\limits_{\sqrt{U}-L^{2}<\sqrt{n}\leq B_{1}+\frac{m}{2\sqrt{U}}%
}d(n)n^{-1/4}\int_{C_{A}^{+}}\frac{e^{-\frac{4\pi z^{2}}{A^{2}}}e(2(B_{1}+%
\frac{u}{2\sqrt{U}}-\sqrt{n})z-bu)}{2\pi i(\frac{z}{\sqrt{U}}-b)}%
\mbox
{\LARGE \textbar }_{u=M_{1}(n)}^{u=m}.
\end{eqnarray}%
In the same way,
\begin{equation}
w^{-}(b)=\sum\limits_{B_{1}<\sqrt{n}\leq \sqrt{U}+L^{2}}d(n)n^{-1/4}%
\int_{C_{A}^{-}}\frac{e^{-\frac{4\pi z^{2}}{A^{2}}}e(2(B_{1}+\frac{u}{2\sqrt{%
U}}-\sqrt{n})z-bu)}{2\pi i(\frac{z}{\sqrt{U}}-b)}\mbox {\LARGE \textbar
}_{u=0}^{u=M_{2}(n)}
\end{equation}%
For a function $g_{n}(u)$, by (7.28) we have%
\begin{eqnarray}
\sum\limits_{\sqrt{U}-L^{2}<\sqrt{n}\leq B_{1}+\frac{m}{2\sqrt{U}}}g_{n}(u)%
\mbox {\LARGE \textbar }_{u=M_{1}(n)}^{u=m} \ \ \ \ \ \ \ \ \ \ \ \ \ \ \ \
\ \ \ \ \ \ \ \ \ \ \ \ \ \ \ \ \ \ \ \ \ \ \ \ \ \ \ \ \ \ \ \ \ \ \ \ \ \
\ \ \   \notag \\
=\sum\limits_{\sqrt{U}-L^{2}<\sqrt{n}\leq B_{1}}g_{n}(u)%
\mbox {\LARGE \textbar
}_{u=0}^{u=m}+\sum\limits_{B_{1}<\sqrt{n}\leq B_{1}+\frac{m}{2\sqrt{U}}%
}g_{n}(u)\mbox {\LARGE \textbar
}_{u=2\sqrt{U}(\sqrt{n}-\sqrt{U})}^{u=m} \ \ \ \ \   \notag \\
=\sum\limits_{\sqrt{U}-L^{2}<\sqrt{n}\leq B_{1}+\frac{u}{2\sqrt{U}}}g_{n}(u)%
\mbox {\LARGE \textbar
}_{u=0}^{u=m}-\sum\limits_{0<\sqrt{n}-B_{1}\leq \frac{m}{2\sqrt{U}}}g_{n}(2%
\sqrt{U}(\sqrt{n}-B_{1}))
\end{eqnarray}%
Likewise ,%
\begin{equation*}
\sum\limits_{B_{1}<\sqrt{n}\leq \sqrt{U}+L^{2}}g_{n}(u)%
\mbox {\LARGE
\textbar }_{u=0}^{u=M_{2}(n)}=\ \ \ \ \ \ \ \ \ \ \ \ \ \ \ \ \ \ \ \ \ \ \
\ \ \ \ \ \ \ \ \ \ \ \ \ \ \ \ \ \ \ \ \ \ \ \ \ \ \ \ \ \ \ \ \ \ \ \ \ \
\ \ \ \ \ \ \ \ \
\end{equation*}

\begin{equation}
=\sum\limits_{B_{1}+\frac{u}{2\sqrt{U}}<\sqrt{n}\leq \sqrt{U}+L^{2}}g_{n}(u)%
\mbox {\LARGE \textbar
}_{u=0}^{u=m}+\sum\limits_{0<\sqrt{n}-B_{1}\leq \frac{m}{2\sqrt{U}}}g_{n}(2%
\sqrt{U}(\sqrt{n}-B_{1}))
\end{equation}%
Noticing $e(-ub)=1$ for $u=0,m$, and
\begin{equation*}
e(2(B_{1}+\frac{u}{2\sqrt{U}}-\sqrt{n})z-bu)\mbox {\LARGE \textbar
}_{u=2\sqrt{U}(\sqrt{n}-B_{1})}=\ \ \ \ \ \ \ \ \ \ \ \ \ \ \ \ \ \ \ \ \ \
\ \ \ \ \ \ \ \ \ \ \ \ \ \ \ \ \ \
\end{equation*}%
\begin{eqnarray*}
&=&e(-2b\sqrt{U}(\sqrt{n}-B_{1}))=e(-2b\sqrt{nU}+2b\sqrt{U}(\sqrt{U}+\frac{j%
}{2\sqrt{U}}+\eta +\xi )) \\
&=&e(-2b\sqrt{nU}+2b\sqrt{U}\xi )
\end{eqnarray*}%
for $\eta =k_{1}/\sqrt{U}$, we know that (7.30)-(7.34) deduce
\begin{equation}
Q_{1}(B_{1})=\sqrt{2}\sum%
\limits_{b=1}^{b_{0}}(w_{1}^{+}(b)+w_{1}^{-}(b)+w_{0}(b))+\delta _{Q_{11}},
\end{equation}%
where
\begin{equation}
w_{1}^{+}(b) =\sum\limits_{\sqrt{U}-L^{2}<\sqrt{n}\leq B_{1}+\frac{u}{2\sqrt{%
U}}}d(n)n^{-1/4}\int_{C_{A}^{+}}\frac{e^{-\frac{4\pi z^{2}}{A^{2}}}e(2(B_{1}+%
\frac{u}{2\sqrt{U}}-\sqrt{n})z)}{2\pi i(\frac{z}{\sqrt{U}}-b)}dz%
\mbox
{\LARGE \textbar }_{u=0}^{u=m}
\end{equation}
\begin{equation}
w_{1}^{-}(b) =\sum\limits_{B_{1}+\frac{u}{2\sqrt{U}}<\sqrt{n}\leq \sqrt{U}%
+L^{2}}d(n)n^{-1/4}\int_{C_{A}^{-}}\frac{e^{-\frac{4\pi z^{2}}{A^{2}}%
}e(2(B_{1}+\frac{u}{2\sqrt{U}}-\sqrt{n})z)}{2\pi i(\frac{z}{\sqrt{U}}-b)}dz%
\mbox {\LARGE \textbar }_{u=0}^{u=m}
\end{equation}
\begin{eqnarray}
w_{0}(b) &=&\sum\limits_{0<\sqrt{n}-B_{1}\leq \frac{m}{2\sqrt{U}}%
}d(n)n^{-1/4}\oint_{C_{A}^{-}-C_{A}^{+}}\frac{e^{-\frac{4\pi z^{2}}{A^{2}}%
}e(-2b\sqrt{nU}+2b\sqrt{U}\xi )dz}{2\pi i(\frac{z}{\sqrt{U}}-b)} \ \ \ \ \ \
\ \ \ \ \ \ \ \ \ \   \notag \\
&=&\sqrt{U}e^{-\frac{4\pi Ub^{2}}{A^{2}}}e(2b\sqrt{U}\xi )\sum\limits_{0<%
\sqrt{n}-B_{1}\leq \frac{m}{2\sqrt{U}}}d(n)n^{-1/4}e(-2b\sqrt{nU})
\end{eqnarray}%
Repeating the proof of (7.24),
\begin{equation}
\delta _{Q_{10}}=\sum\limits_{b=1}^{b_{0}}(w_{1}^{+}(b)+w_{1}^{-}(b))\ll U^{%
\frac{1}{4}+\varepsilon },\ \ \ \varepsilon =O(\frac{1}{\log L}),
\end{equation}%
we obtain from (7.35)-(7.39) that
\begin{equation}
Q_{1}(B_{1})=\sqrt{2U}\sum\limits_{b=1}^{b_{0}}e^{-\frac{4\pi Ub^{2}}{A^{2}}%
}e(2b\sqrt{U}\xi )\sum\limits_{0<\sqrt{n}-B_{1}\leq \frac{m}{2\sqrt{U}}%
}d(n)n^{-1/4}e(-2b\sqrt{nU})+\delta _{Q_{1}},
\end{equation}%
where $\delta _{Q_{1}}=\delta _{Q_{10}}+\delta _{Q_{11}}.\ $Therefore, using
$e((U+n)b)=1$ and
\begin{equation*}
e(-2b\sqrt{nU})=e((U+n)b-2b\sqrt{nU})=e(b(\sqrt{n}-\sqrt{U})^{2}),
\end{equation*}%
we have
\begin{equation}
Q_{1}(B_{1})=\frac{\sqrt{U}}{\sqrt{2}}\sum\limits_{0<\sqrt{n}-B_{1}\leq
\frac{m}{2\sqrt{U}}}d(n)n^{-1/4}N(n)+\delta _{Q_{11}},
\end{equation}%
where
\begin{eqnarray*}
N(n) &=&2\sum\limits_{b=1}^{b_{0}}e^{-\frac{4\pi Ub^{2}}{A^{2}}}e(b(2\sqrt{U}%
\xi +(\sqrt{n}-\sqrt{U})^{2})) \\
&=&2\sum\limits_{b=1}^{\infty }e^{-\frac{4\pi Ub^{2}}{A^{2}}}e(b(2\sqrt{U}%
\xi +(\sqrt{n}-\sqrt{U})^{2}))+O(U^{-10}),
\end{eqnarray*}%
the last equality is given by
\begin{equation*}
e^{-\frac{4\pi Ub^{2}}{A^{2}}}=e^{-\frac{4\pi b^{2}}{C_{0}^{2}}}\ll
U^{-10}e^{-\frac{b^{2}}{C_{0}^{2}}}
\end{equation*}%
for $b>b_{0}=[C_{0}^{2}L],\ C_{0}\geq 200.$ By Lemma 3.2,
\begin{eqnarray*}
\mathrm{Re}N(n) &=&-1+\sum\limits_{b=-\infty }^{\infty }e^{-\frac{4\pi Ub^{2}%
}{A^{2}}}e(b(2\sqrt{U}\xi +(\sqrt{n}-\sqrt{U})^{2}))+O(U^{-10}) \\
&=&-1+\frac{A}{2\sqrt{U}}\sum\limits_{b=-\infty }^{\infty }e^{-\frac{\pi
A^{2}}{4U}(2\sqrt{U}\xi +(\sqrt{n}-\sqrt{U})^{2}-b)^{2}}+O(U^{-10}).
\end{eqnarray*}%
By (7.41),%
\begin{equation*}
(\sqrt{n}-\sqrt{U})^{2}=O(L^{-1}).
\end{equation*}%
Thus,
\begin{equation*}
\frac{1}{9}\leq 2\sqrt{U}\xi +(\sqrt{n}-\sqrt{U})^{2}\leq \frac{3}{7}
\end{equation*}%
for
\begin{equation*}
\frac{1}{8}\leq 2\sqrt{U}\xi \leq \frac{3}{8}.
\end{equation*}%
Hence
\begin{eqnarray}
\mathrm{Re}N(n) &=&-1+O(\sqrt{L}e^{-\frac{\pi C_{0}^{2}L}{4\cdot 11^{2}}})+O(%
\sqrt{L}\sum\limits_{b=1}^{\infty }e^{-\frac{C_{0}^{2}Lb^{2}}{4}})+O(U^{-10})
\notag \\
&=&-1+O(U^{-10}).
\end{eqnarray}%
It follows from (7.41) and (7.42) that
\begin{equation}
Q_{1}(B_{1})=-\frac{\sqrt{U}}{\sqrt{2}}\sum\limits_{0<\sqrt{n}-B_{1}\leq
\frac{m}{2\sqrt{U}}}d(n)n^{-1/4}+\delta _{Q_{1}},
\end{equation}%
where $\delta _{Q_{1}}=\delta _{Q_{10}}+\delta _{Q_{11}}\ll U^{\frac{1}{4}%
+\varepsilon },\ 0<\varepsilon \ll \frac{1}{\log L}$. The proof of (7.4) is
complete.\ \ \ \ \ \ \ $\boxempty$

\ \

\textbf{Note}

For other application to be able, let us write down the accurate $\delta
_{Q_{1}}(B_{1})$, the $O$-term of (7.43):
\begin{equation}
\delta _{Q_{1}}(B_{1})=\delta _{Q_{10}}+\delta _{Q_{11}}=Q_{11}+Q_{12}+\sqrt{%
2}\sum\limits_{b=1}^{b_{0}}(w_{1}^{+}(b)+w_{1}^{-}(b))+O(U^{-\frac{1}{4}%
+\varepsilon (U)}),
\end{equation}%
where $\varepsilon (U)\ll 1/\log \log U$,
\begin{eqnarray}
Q_{1j}(B_{1}) &  \notag \\
=&\sqrt{2}\sum\limits_{-L^{2}<\sqrt{n}-\sqrt{U}\leq
L^{2}}d(n)n^{-1/4}\int_{C}e^{-\frac{4\pi z^{2}}{A^{2}}}e(2(B_{1}-\sqrt{n}%
)z)\rho _{j}(z)dz,  \notag \\
&j=1,2, \ \ \ \ \ \ \ \ \ \ \ \ \ \ \ \ \ \ \ \ \ \ \ \ \ \ \ \ \ \ \ \ \ \
\ \ \ \ \ \ \ \ \ \ \ \ \ \ \ \ \ \ \ \ \ \ \ \ \ \ \ \ \ \ \ \ \ \ \ \
\end{eqnarray}
\begin{eqnarray}
\rho _{1}(z) &=&\frac{1}{2}+\frac{1}{2}e(\frac{zm}{\sqrt{U}}),  \notag \\
\rho _{2}(z) &=&\sum_{b}\ ^{\prime }\frac{e(\frac{zu}{\sqrt{U}})}{2\pi i(%
\frac{z}{\sqrt{U}}-b)}\mbox {\LARGE \textbar }_{u=0}^{u=m},  \notag \\
w_{1}^{+}(b) &=&\sum\limits_{\sqrt{U}-L^{2}<\sqrt{n}\leq B_{1}+\frac{u}{2%
\sqrt{U}}}d(n)n^{-1/4}\int_{C_{A}^{+}}f_{nb}(z,u)\mbox {\LARGE \textbar
}_{u=0}^{u=m}dz, \\
w_{1}^{-}(b) &=&\sum\limits_{B_{1}+\frac{u}{2\sqrt{U}}<\sqrt{n}\leq \sqrt{U}%
+L^{2}}d(n)n^{-1/4}\int_{C_{A}^{-}}f_{nb}(z,u)\mbox
{\LARGE \textbar }_{u=0}^{u=m}dz,
\end{eqnarray}%
where $C$ is from $\sqrt{U}/2$ to $(b_{0}+1/2)\sqrt{U}$, $\sum^{\prime }$ is
for $-U^{10}\leq b\leq 0$ or $b_{0}+1\leq b\leq U^{10}$, $C_{A}^{\pm }$ is
Fig 7.1,
\begin{equation}
f_{nb}(z,u)=\frac{e^{-\frac{4\pi z^{2}}{A^{2}}}e(2(B_{1}+\frac{u}{2\sqrt{U}}-%
\sqrt{n})z)}{2\pi i(\frac{z}{\sqrt{U}}-b)}.
\end{equation}

\ \

\section{EVALUATION OF SUM $\sum\limits_{k=0}^{m}T(B+\protect\xi)$}

Let
\begin{eqnarray}
Q(B_{1}) &=&\sum\limits_{k=0}^{m}T(B+\xi )=\sum\limits_{k=0}^{m}(T_{0}(B_{1}+%
\frac{k}{2\sqrt{U}})+T_{1}(B_{1}+\frac{k}{2\sqrt{U}}))  \notag \\
&=&Q_{0}(B_{1})+Q_{1}(B_{1}),\ \ \ \ \ \ \ \ \mathrm{say}
\end{eqnarray}%
where $T_{0}(B_{1})$ is (5.14), $T_{1}(B_{1})$ is (5.18) and
\begin{equation}
B_{1}=\sqrt{U}+\frac{j}{2\sqrt{U}}+\eta +\xi ,\ m\ll \sqrt{U}L^{-2}.
\end{equation}%
In this section , we are going to prove that
\begin{eqnarray}
\mathrm{Re}\ Q(B_{1})&=&2\sqrt{U}\sum\limits_{\sqrt{n}\leq \frac{\sqrt{U}}{2}%
}d(n)n^{-\frac{1}{4}}e^{-\frac{4\pi n}{A^{2}}}\int_{0}^{\frac{m}{2\sqrt{U}}%
}\cos (4\pi \sqrt{n}(B_{1}+x)+\frac{\pi }{4})dx  \notag \\
&&+\mathrm{Re}\ \delta _{Q}(B_{1}),
\end{eqnarray}%
where $\delta _{Q}\ll U^{\frac{1}{4}+\varepsilon }$.

\ \

(This equality is a key of success of our paper. But if there was no sum of
the right side, we could replace $k$ and $\sum\limits_{k=0}^{m}S(B)$ by $%
k+k^{\prime }$ and $\sum\limits_{k,k^{\prime }=0}^{m}S(B)$, respectively,
and we should be succeed, although it would be more numerous.)

\ \

\textit{Proof} of (8.3):

First, we evaluate further $\mathrm{Re}\ Q_{1}(B_{1})$. By (7.23):
\begin{equation*}
\frac{1}{9}\leq \{(B_{1}+\frac{u}{2\sqrt{U}})^{2}\}\leq \frac{4}{9}
\end{equation*}%
for $u=0,m$. For $A=C_{0}\sqrt{UL}$, $C_{0}\geq 200$, $-\frac{5\sqrt{L}}{A}%
\leq \theta \leq \frac{5\sqrt{L}}{A}$ and $u=0,m$, we have
\begin{equation*}
|(B_{1}+\frac{u}{2\sqrt{U}}+\theta )^{2}-(B_{1}+\frac{u}{2\sqrt{U}})^{2}|\ \
\ \ \ \ \ \ \ \ \ \ \ \ \ \ \ \ \ \ \ \ \ \ \ \ \ \ \
\end{equation*}%
\begin{equation}
=|\theta ||2B_{1}+\frac{u}{\sqrt{U}}+\theta |\leq \frac{5}{200\sqrt{U}}(2%
\sqrt{U}+O(1))<\frac{1}{19}
\end{equation}%
Hence it does not run over any integer number from $(B_{1}+\frac{u}{2\sqrt{U}%
})^{2}$ to $(B_{1}+\frac{u}{2\sqrt{U}}+\theta )^{2}$ when $u=0,m$. So that
\begin{eqnarray*}
\int_{0}^{\frac{5\sqrt{L}}{A}}e^{-\pi A^{2}\theta ^{2}}\sum\limits_{0<\sqrt{n%
}-B_{1}-\frac{u}{2\sqrt{U}}\leq \theta } &=&0, \\
\int_{0}^{\frac{5\sqrt{L}}{A}}e^{-\pi A^{2}\theta ^{2}}\sum\limits_{-\theta
\leq \sqrt{n}-B_{1}-\frac{u}{2\sqrt{U}}\leq 0} &=&0.
\end{eqnarray*}%
\newline
Hence
\begin{equation*}
A\int_{-\frac{5\sqrt{L}}{A}}^{\frac{5\sqrt{L}}{A}}e^{-\pi A^{2}\theta
^{2}}d\theta \sum\limits_{0<\sqrt{n}-B_{1}-\theta \leq \frac{m}{2\sqrt{U}}%
}d(n)n^{-1/4}\ \ \ \ \ \ \ \ \ \ \ \ \ \ \ \ \ \ \ \ \ \ \ \ \ \ \ \
\end{equation*}%
\begin{eqnarray}
\ \ \ \ \ \ \ \ \ \ \ \ \ &=&(A\int_{-\frac{5\sqrt{L}}{A}}^{\frac{5\sqrt{L}}{%
A}}e^{-\pi A^{2}\theta ^{2}}d\theta )\sum\limits_{0<\sqrt{n}-B_{1}\leq \frac{%
m}{2\sqrt{U}}}d(n)n^{-1/4}  \notag \\
&=&(1+O(U^{-10}))\sum\limits_{0<\sqrt{n}-B_{1}\leq \frac{m}{2\sqrt{U}}%
}d(n)n^{-1/4},
\end{eqnarray}%
where the last equality is given by
\begin{equation*}
A\int_{-\frac{5\sqrt{L}}{A}}^{\frac{5\sqrt{L}}{A}}e^{-\pi A^{2}\theta
^{2}}d\theta =A\int_{-\infty }^{\infty }+O(U^{-10})=1+O(U^{-10}).
\end{equation*}%
It follows from (7.4) and (8.5) that
\begin{eqnarray}
\mathrm{Re}\ Q_{1}(B_{1}) &=&-\frac{\sqrt{U}A}{\sqrt{2}}\int_{-\frac{5\sqrt{L%
}}{A}}^{\frac{5\sqrt{L}}{A}}e^{-\pi A^{2}\theta ^{2}}\sum\limits_{0<\sqrt{n}%
-B_{1}-\theta \leq \frac{m}{2\sqrt{U}}}d(n)n^{-1/4}  \notag \\
&&+\delta _{Q_{1}}(B_{1})+O(U^{-1/4+\varepsilon (U)}),
\end{eqnarray}%
where $\delta _{Q_{1}}(B_{1})\ll U^{1/4+\varepsilon (U)}$ is (7.43). By
Lemma 2.4,
\begin{eqnarray}
\mathrm{Re}\ Q_{1}(B_{1}) &=&-\frac{\sqrt{U}A}{\sqrt{2}}\sum\limits_{n\leq
N}d(n)\int_{-\frac{5\sqrt{L}}{A}}^{\frac{5\sqrt{L}}{A}}e^{-\pi A^{2}\theta
^{2}}d\theta \int_{(B_{1}+\theta )^{2}}^{(B_{1}+\frac{m}{2\sqrt{U}}+\theta
)^{2}}x^{-1/4}\Theta (nx)dx  \notag \\
&&+\&_{l}+\delta _{Q_{1}}(B_{1})+O(U^{-1/4+\varepsilon (U)})+O(\varepsilon
_{1})  \notag \\
&=&-2\sqrt{U}\sum\limits_{n\leq N}d(n)n^{-1/4}\sum\limits_{\alpha =0}^{2}%
\frac{\gamma _{\alpha }A}{(4\pi \sqrt{n})^{\alpha }}\int_{-\frac{5\sqrt{L}}{A%
}}^{\frac{5\sqrt{L}}{A}}e^{-\pi A^{2}\theta ^{2}}j_{n\alpha }(\theta )d\theta
\notag \\
&&+\&_{l}+\delta _{Q_{1}}(B_{1})+O(U^{-1/4+\varepsilon (U)})+O(\varepsilon
_{1})
\end{eqnarray}%
holds for any $\varepsilon _{1}>0$ as $N\geq N(\varepsilon _{1})$, where
\begin{equation}
j_{n\alpha }(\theta )=\int_{B_{1}+\theta }^{B_{1}+\frac{m}{2\sqrt{U}}+\theta
}\frac{\cos (4\pi \sqrt{n}x+\frac{\pi }{4}+\frac{\alpha \pi }{2})}{x^{\alpha
}}dx,
\end{equation}%
\begin{equation}
\&_{l}=\&_{l}(\eta )=-\sqrt{2U}\int_{-\frac{5\sqrt{L}}{A}}^{\frac{5\sqrt{L}}{%
A}}e^{-\pi A^{2}\theta ^{2}}\int_{B_{0}+\theta }^{B_{0}+\frac{m}{2\sqrt{U}}%
+\theta }\delta _{x}(\eta )dx,
\end{equation}%
where
\begin{equation*}
B_{0}=B_{1}-\eta =\sqrt{U}+\frac{j}{2\sqrt{U}}+\xi ,\ \ \ \ \ \
\end{equation*}%
\begin{equation*}
\delta _{x}(\eta )=(x+\eta )^{1/2}(\log (x+\eta )^{2}+2\gamma ).
\end{equation*}%
By Lemma 3.6,
\begin{equation*}
\delta _{x}(\eta )\ll _{\eta }U^{-99}
\end{equation*}%
(the definition of  "$\ll _{\eta }$ "  see (1.28)), we know that
\begin{equation}
\&_{l}=\&_{l}(\eta )\ll _{\eta }U^{-90}.
\end{equation}%
Moreover,
\begin{equation*}
A\int_{-\frac{5\sqrt{L}}{A}}^{\frac{5\sqrt{L}}{A}}e^{-\pi A^{2}\theta
^{2}}j_{n\alpha }(\theta )d\theta \ \ \ \ \ \ \ \ \ \ \ \ \ \ \ \ \ \ \ \ \
\ \ \ \ \ \ \ \ \ \ \ \ \ \ \ \ \ \ \ \ \ \ \ \ \ \ \ \ \ \ \ \ \ \ \ \ \ \
\
\end{equation*}%
\begin{eqnarray}
&=&A\int_{-\frac{5\sqrt{L}}{A}}^{\frac{5\sqrt{L}}{A}}e^{-\pi A^{2}\theta
^{2}}d\theta \frac{\sin (4\pi \sqrt{n}(x+\theta )+\frac{\pi }{4}+\frac{%
\alpha \pi }{2})}{4\pi \sqrt{n}(x+\theta )^{\alpha }}%
\mbox {\LARGE \textbar
}_{x=B_{1}}^{x=B_{1}+\frac{m}{2\sqrt{U}}}  \notag \\
&&+A\int_{-\frac{5\sqrt{L}}{A}}^{\frac{5\sqrt{L}}{A}}e^{-\pi A^{2}\theta
^{2}}d\theta \int_{B_{1}}^{B_{1}+\frac{m}{2\sqrt{U}}}\frac{\alpha \sin (4\pi
\sqrt{n}(x+\theta )+\frac{\pi }{4}+\frac{\alpha \pi }{2})}{4\pi \sqrt{n}%
(x+\theta )^{\alpha +1}}dx  \notag \\
&\ll &\frac{A}{n(\sqrt{U})^{\alpha }},
\end{eqnarray}%
where the last inequality is given by the integration by parts for $\theta $%
. Hence, using
\begin{eqnarray}
\sum\limits_{U^{3}<n\leq N}d(n)n^{-1/4}\frac{A}{n} &\ll &U^{-1/4}L^{2}, \\
\sum\limits_{U/4<n\leq N}d(n)n^{-1/4-\alpha /2}\frac{A}{n(\sqrt{U})^{\alpha }%
} &\ll &U^{-1/4}
\end{eqnarray}%
for $\alpha \geq 1$, by (8.7), (8.10), (8.12), (8.13) and putting $%
\varepsilon _{1}\rightarrow 0_{+}$, we have
\begin{eqnarray}
\mathrm{Re}\ Q_{1}(B_{1}) &=&-2\sqrt{U}\sum\limits_{\frac{U}{4}<n\leq
U^{3}}d(n)n^{-1/4}A\int_{-\frac{5\sqrt{L}}{A}}^{\frac{5\sqrt{L}}{A}}e^{-\pi
A^{2}\theta ^{2}}j_{n0}(\theta )d\theta  \notag \\
&&+Q_{10}(B_{1})+\mathrm{Re}\ \delta _{Q_{1}}(B_{1})+O_{\eta }(U^{-1/4}),
\end{eqnarray}%
where
\begin{eqnarray}
Q_{10}(B_{1}) &=&-2\sqrt{U}\sum\limits_{\sqrt{n}\leq \frac{\sqrt{U}}{2}%
}d(n)n^{-1/4}\sum\limits_{\alpha =0}^{2}\frac{\gamma _{\alpha }A}{(4\pi
\sqrt{n})^{\alpha }}\int_{-\frac{5\sqrt{L}}{A}}^{\frac{5\sqrt{L}}{A}}e^{-\pi
A^{2}\theta ^{2}}j_{n\alpha }(\theta )d\theta  \notag \\
&=&-2\sqrt{U}\sum\limits_{\sqrt{n}\leq \frac{\sqrt{U}}{2}}d(n)n^{-1/4}\sum%
\limits_{\alpha =0}^{2}\frac{\gamma _{\alpha }A}{(4\pi \sqrt{n})^{\alpha }}%
\int_{-\frac{5\sqrt{L}}{A}}^{\frac{5\sqrt{L}}{A}}e^{-\pi A^{2}\theta
^{2}}d\theta \times  \notag \\
&&\times \int_{B_{1}+\theta }^{B_{1}+\frac{m}{2\sqrt{U}}+\theta }\frac{\cos
(4\pi \sqrt{n}x+\frac{\pi }{4}+\frac{\alpha \pi }{2})}{x^{\alpha }}dx.
\end{eqnarray}%
Furthermore,
\begin{equation*}
A\int_{-\frac{5\sqrt{L}}{A}}^{\frac{5\sqrt{L}}{A}}e^{-\pi A^{2}\theta
^{2}}j_{n0}(\theta )d\theta =\ \ \ \ \ \ \ \ \ \ \ \ \ \ \ \ \ \ \ \ \ \ \ \
\ \ \ \ \ \ \ \ \ \ \ \ \ \ \ \ \ \ \ \ \ \ \ \ \ \ \ \ \ \ \ \ \ \ \ \ \ \
\ \ \ \ \ \
\end{equation*}%
\begin{eqnarray}
&=&\int_{B_{1}}^{B_{1}+\frac{m}{2\sqrt{U}}}dx\mathrm{Re}\ A\int_{-\frac{5%
\sqrt{L}}{A}}^{\frac{5\sqrt{L}}{A}}e^{-\pi A^{2}\theta ^{2}}e(2\sqrt{n}%
(x+\theta )+1/8)d\theta  \notag \\
&=&\int_{B_{1}}^{B_{1}+\frac{m}{2\sqrt{U}}}dx\mathrm{Re}\ e(2\sqrt{n}%
x+1/8)A\int_{-\infty }^{\infty }e^{-\pi A^{2}\theta ^{2}}e(2\sqrt{n}\theta
)d\theta +O(U^{-10})  \notag \\
&=&e^{-\frac{4\pi n}{A}}\int_{B_{1}}^{B_{1}+\frac{m}{2\sqrt{U}}}\cos (4\pi
\sqrt{n}x+\pi /4)dx+O(U^{-10})  \notag \\
&=&e^{-\frac{4\pi n}{A}}\int_{0}^{\frac{m}{2\sqrt{U}}}\cos (4\pi \sqrt{n}%
(B_{1}+x)+\pi /4)dx+O(U^{-10}).
\end{eqnarray}%
Noticing $e^{-\frac{4\pi n}{A}}\ll U^{-10}$ for $\sqrt{n}>b_{0}\sqrt{U}$, it
follows from (8.14) and (8.16) that
\begin{eqnarray}
\mathrm{Re}\ Q_{1}(B_{1}) &=&-2\sqrt{U}\int_{0}^{\frac{m}{2\sqrt{U}}%
}(\sum\limits_{\frac{\sqrt{U}}{2}<\sqrt{n}\leq b_{0}\sqrt{U}}d(n)n^{-\frac{1%
}{4}}e^{-\frac{4\pi n}{A^{2}}}\cos (4\pi \sqrt{n}(B_{1}+x)+\pi /4))dx  \notag
\\
&&+Q_{10}(B_{1})+\delta _{Q_{1}}(B_{1})+O_{\eta }(U^{-1/4+\varepsilon (U)}),
\end{eqnarray}%
where $Q_{10}(B_{1})$ is (8.15), $\delta _{Q_{1}}(B_{1})\ll U^{\frac{1}{4}%
+\varepsilon }$ is (7.43).

Next we evaluate $\mathrm{Re}\ Q_{0}(B_{1})$. By (5.14) and (7.2),
\begin{eqnarray}
Q_{0}(B_{1}) &=&\sum\limits_{k=0}^{m}T_{0}(B+\xi )  \notag \\
&=&e(\frac{1}{8})\sum\limits_{\sqrt{n}\leq \frac{\sqrt{U}}{2}%
}d(n)n^{-1/4}\sum\limits_{\alpha =0}^{2}\gamma _{\alpha }(\frac{i}{4\pi
\sqrt{n}})^{\alpha }A\int_{-\frac{5\sqrt{L}}{A}}^{\frac{5\sqrt{L}}{A}%
}e^{-\pi A^{2}\theta ^{2}}  \notag \\
&&\times e(2\sqrt{n}(B_{1}+\theta ))\rho (n,\alpha ,\theta )d\theta ,
\end{eqnarray}%
where
\begin{equation}
\rho (n,\alpha ,\theta )=\sum\limits_{k=0}^{m}\frac{e(\frac{\sqrt{n}k}{\sqrt{%
U}})}{(B_{1}+\theta +\frac{k}{2\sqrt{U}})^{\alpha }}.\ \ \
\end{equation}%
By Lemma 3.3,
\begin{eqnarray}
\rho (n,\alpha ,\theta ) &=&\int_{0}^{m}\frac{e(\frac{\sqrt{n}u}{\sqrt{U}})du%
}{(B_{1}+\theta +\frac{u}{2\sqrt{U}})^{\alpha }}+\rho _{\delta }(n,\alpha
,\theta )+O(U^{-9})  \notag \\
&=&2\sqrt{U}\int_{B_{1}+\theta }^{B_{1}+\frac{m}{2\sqrt{U}}+\theta }\frac{e(2%
\sqrt{n}(x-B_{1}-\theta ))dx}{x^{\alpha }}  \notag \\
&&+\rho _{\delta }(n,\alpha ,\theta )+O(U^{-9}),
\end{eqnarray}%
where
\begin{eqnarray}
\rho _{\delta }(n,\alpha ,\theta ) &=&\sum\limits_{-U^{10}\leq b\leq
U^{10},b\neq 0}\int_{0}^{m}\frac{e((\frac{\sqrt{n}}{\sqrt{U}}-b)u)du}{%
(B_{1}+\theta +\frac{u}{2\sqrt{U}})^{\alpha }}+\frac{1}{2(B_{1}+\theta
)^{\alpha }}  \notag \\
&&+\frac{e(\frac{\sqrt{n}m}{\sqrt{U}})}{2(B_{1}+\theta +\frac{m}{2\sqrt{U}}%
)^{\alpha }}  \notag \\
&=&\sum\limits_{-U^{10}\leq b\leq U^{10},b\neq 0}\frac{e(\frac{\sqrt{n}u}{%
\sqrt{U}})}{2\pi i(\frac{\sqrt{n}}{\sqrt{U}}-b)(B_{1}+\theta +\frac{u}{2%
\sqrt{U}})^{\alpha }}\mbox
{\LARGE \textbar }_{u=0}^{u=m}  \notag \\
&&+\sum\limits_{-U^{10}\leq b\leq U^{10},b\neq 0}\int_{0}^{m}\frac{\frac{%
\alpha }{2\sqrt{U}}e((\frac{\sqrt{n}}{\sqrt{U}}-b)u)du}{2\pi i(\frac{\sqrt{n}%
}{\sqrt{U}}-b)(B_{1}+\theta +\frac{u}{2\sqrt{U}})^{\alpha +1}}  \notag \\
&&+\frac{1}{2(B_{1}+\theta )^{\alpha }}+\frac{e(\frac{\sqrt{n}m}{\sqrt{U}})}{%
2(B_{1}+\theta +\frac{m}{2\sqrt{U}})^{\alpha }}.
\end{eqnarray}%
It is easy to prove that in Lemma 3.5 ,
\begin{eqnarray}
\delta _{Q_{0}}(B_{1}) &=&\int_{-\frac{5L}{A}}^{\frac{5L}{A}}e^{-\pi \theta
^{2}}d\theta \sum\limits_{\sqrt{n}\leq \frac{\sqrt{U}}{2}}d(n)n^{-1/4}\sum%
\limits_{\alpha =0}^{2}\gamma _{\alpha }(\frac{i}{4\pi \sqrt{n}})^{\alpha
}e(2\sqrt{n}(B_{1}+\theta ))\rho _{\delta }(n,\alpha )  \notag \\
&\ll &U^{1/4+\varepsilon }.
\end{eqnarray}%
Then by (8.18), (8.20) and (8.22), we obtain that
\begin{eqnarray*}
Q_{0}(B_{1}) &=&2\sqrt{U}e(\frac{1}{8})\sum\limits_{\sqrt{n}\leq \frac{\sqrt{%
U}}{2}}d(n)n^{-1/4}\sum\limits_{\alpha =0}^{2}\gamma _{\alpha }(\frac{i}{%
4\pi \sqrt{n}})^{\alpha }A\int_{-\frac{5\sqrt{L}}{A}}^{\frac{5\sqrt{L}}{A}}
\\
&&\times e^{-\pi A^{2}\theta ^{2}}\int_{B_{1}+\theta }^{B_{1}+\frac{m}{2%
\sqrt{U}}+\theta }\frac{e(2\sqrt{n}x)}{x^{\alpha }}dx+\delta
_{Q_{0}}(B_{1})+O(U^{-1/4+\varepsilon (U)}),
\end{eqnarray*}%
where $\delta _{Q_{0}}(B_{1})\ll U^{1/4+\varepsilon (U)}$ is (8.22). Hence
\begin{equation}
\mathrm{Re}\ Q_{0}(B_{1})=-Q_{10}(B_{1})+\delta
_{Q_{0}}(B_{1})+O(U^{-1/4+\varepsilon (U)}),
\end{equation}%
where $Q_{10}(B_{1})$ is (8.15), $\delta _{Q_{0}}(B_{1})$ is (8.22). It
follows from (8.17) and (8.23) that
\begin{eqnarray}
\mathrm{Re}\ Q(B_{1}) &=&\mathrm{Re}\ (Q_{0}(B_{1})+Q_{1}(B_{1}))  \notag \\
&=&-2\sqrt{U}\int_{0}^{\frac{m}{2\sqrt{U}}}(\sum\limits_{\frac{\sqrt{U}}{2}<%
\sqrt{n}\leq b_{0}\sqrt{U}}d(n)n^{-\frac{1}{4}}e^{-\frac{4\pi n}{A^{2}}}
\notag \\
&&\times \cos (4\pi \sqrt{n}(B_{1}+x)+\pi /4))dx+\delta _{Q}(B_{1})+O_{\eta
}(U^{-1/4+\varepsilon (U)}).
\end{eqnarray}%
Taking $\nu =\frac{j}{2\sqrt{U}}+x$ in Lemma 4.3, we have $\nu =O(L^{-1})$.
By (8.24) and Lemma 4.3 we obtain
\begin{equation*}
\mathrm{Re}\ Q(B_{1})=\ \ \ \ \ \ \ \ \ \ \ \ \ \ \ \ \ \ \ \ \ \ \ \ \ \ \
\ \ \ \ \ \ \ \ \ \ \ \ \ \ \ \ \ \ \ \ \ \ \ \ \ \ \ \ \ \ \ \ \ \ \ \ \ \
\ \ \ \ \ \ \ \ \ \ \ \ \ \ \ \ \ \ \ \ \ \ \
\end{equation*}%
\begin{eqnarray}
&=&2\sqrt{U}\int_{0}^{\frac{m}{2\sqrt{U}}}\sum\limits_{\sqrt{n}\leq \frac{%
\sqrt{U}}{2}}d(n)n^{-\frac{1}{4}}e^{-\frac{4\pi n}{A}}\sum\limits_{\alpha
=0}^{2}\frac{\gamma _{\alpha }}{(4\pi )^{\alpha }}\frac{\cos (4\pi \sqrt{n}%
(B_{1}+x)+\pi /4+\alpha \pi /2)dx}{(\sqrt{n}(\sqrt{n}+B_{1}+x))^{\alpha }}
\notag \\
&&+\mathrm{Re}\ \delta _{Q}(B_{1})+O_{\eta }(U^{-1/4+\varepsilon (U)})
\notag \\
&=&2\sqrt{U}\sum\limits_{\sqrt{n}\leq \frac{\sqrt{U}}{2}}d(n)n^{-\frac{1}{4}%
}e^{-\frac{4\pi n}{A^{2}}}\int_{0}^{\frac{m}{2\sqrt{U}}}\cos (4\pi \sqrt{n}%
(B_{1}+x)+\pi /4)dx  \notag \\
&&+\mathrm{Re}\ \delta _{q}+\mathrm{Re}\ \delta _{Q}(B_{1})+O_{\eta
}(U^{-1/4+\varepsilon (U)}),
\end{eqnarray}%
where $B_{1}$ is (8.2) and

\begin{eqnarray}
\delta _{Q}(B_{1}) &=&\delta _{Q_{0}}(B_{1})+\delta _{Q_{1}}(B_{1})\ll
U^{1/4+\varepsilon },  \notag \\
\delta _{q} &=&2\sqrt{U}\sum\limits_{\sqrt{n}\leq \frac{\sqrt{U}}{2}}d(n)n^{-%
\frac{1}{4}}\sum\limits_{\alpha =1}^{2}\frac{\gamma _{\alpha }}{(4\pi \sqrt{n%
})^{\alpha }}A\int_{-\frac{5\sqrt{L}}{A}}^{\frac{5\sqrt{L}}{A}}e^{-\pi
A^{2}\theta ^{2}}j_{\alpha n}(\theta )d\theta , \\
j_{\alpha n}(\theta ) &=&\int_{\theta }^{\frac{m}{2\sqrt{U}}+\theta }\frac{%
\cos (4\pi \sqrt{n}(B_{1}+x)+\pi /4+\alpha \pi /2)dx}{(\sqrt{n}%
+B_{1}+x)^{\alpha }}  \notag \\
&\ll &\frac{1}{\sqrt{nU}},\ \text{for }\alpha \geq 1.  \notag
\end{eqnarray}

Hence
\begin{equation}
\delta _{q}\ll \sum\limits_{\sqrt{n}\leq \frac{\sqrt{U}}{2}%
}d(n)n^{-5/4}=O(1).
\end{equation}

By (8.25) and (8.26) we obtain (8.3) and the proof of (8.3) is complete. \ \
\ \ \ \ \ \ \ \ \ \ \ \ \ \ \ \ \ \ \ \ \ \ \ \ \ \ \ \ \ \ \ \ \ \ \ \ \ \
\ \ \ \ \ \ \ \ \ \ \ \ \ \ \ \ \ \ \ \ \ \ \ \ \ \ \ \ \ \ \ \ \ \ $%
\boxempty$

\ \

\textbf{Note}

Like $\S 7$, we need to write down the accurate $\delta _{Q}(B_{1})$:
\begin{equation}
\delta _{Q}(B_{1})=\delta _{Q_{1}}(B_{1})+\mathrm{Re}\delta
_{Q_{0}}(B_{1})+\delta _{q}(B_{1}),
\end{equation}%
where $\delta _{Q_{1}}(B_{1})$ is (7.44), $\delta _{q}(B_{1})$ is (8.26), $%
\delta _{Q_{0}}(B_{1})$ is (8.22), i.e.
\begin{eqnarray}
\delta _{q}(B_{1}) &=&2\sqrt{U}\sum\limits_{\sqrt{n}\leq \frac{\sqrt{U}}{2}%
}d(n)n^{-\frac{1}{4}}\sum\limits_{\alpha =1}^{2}\frac{\gamma _{\alpha }}{%
(4\pi \sqrt{n})^{\alpha }}A\int_{-\frac{5\sqrt{L}}{A}}^{\frac{5\sqrt{L}}{A}%
}e^{-\pi A^{2}\theta ^{2}}  \notag \\
&&\times \int_{\theta }^{\frac{m}{2\sqrt{U}}+\theta }\frac{\cos (4\pi \sqrt{n%
}(B_{1}+x)+\pi /4+\alpha \pi /2)dx}{(\sqrt{n}+B_{1}+x)^{\alpha }}  \notag \\
&=&\frac{\sqrt{U}}{2\pi }\sum\limits_{\sqrt{n}\leq \frac{\sqrt{U}}{2}%
}d(n)n^{-\frac{3}{4}}e^{-\frac{4\pi n}{A^{2}}}  \notag \\
&&\times \sum\limits_{\alpha =1}^{2}\frac{\gamma _{\alpha }\cos (4\pi \sqrt{n%
}(B_{1}+u)-\pi /4+\alpha \pi /2)}{(4\pi \sqrt{n}(\sqrt{n}+B_{1}+u))^{\alpha }%
}\mbox {\LARGE \textbar
}_{u=0}^{u=\frac{m}{2\sqrt{U}}}  \notag \\
&&+O(U^{-1/4+\varepsilon (U)}),
\end{eqnarray}%
\begin{eqnarray}
\delta _{Q_{0}}(B_{1}) &=&\sum\limits_{\sqrt{n}\leq \frac{\sqrt{U}}{2}%
}d(n)n^{-\frac{1}{4}}\sum\limits_{\alpha =0}^{2}\gamma _{\alpha }(\frac{i}{%
4\pi \sqrt{n}})^{\alpha }A\int_{-\frac{5\sqrt{L}}{A}}^{\frac{5\sqrt{L}}{A}%
}e^{-\pi A^{2}\theta ^{2}}d\theta  \notag \\
&&\times e(2\sqrt{n}(B_{1}+\theta )+1/8)\rho _{\delta }(n,\alpha ,\theta ),
\end{eqnarray}

moreover,
\begin{eqnarray}
\rho _{\delta }(n,\alpha ,\theta ) &=&\sum\limits_{-U^{10}\leq b\leq
U^{10},b\neq 0}\frac{e(\frac{\sqrt{n}u}{\sqrt{U}})}{2\pi i(\frac{\sqrt{n}}{%
\sqrt{U}}-b)(B_{1}+\theta +\frac{u}{2\sqrt{U}})^{\alpha }}%
\mbox {\LARGE
\textbar }_{u=0}^{u=m}  \notag \\
&&+\sum\limits_{-U^{10}\leq b\leq U^{10},b\neq 0}\int_{0}^{m}\frac{\alpha e((%
\frac{\sqrt{n}}{\sqrt{U}}-b)u)du}{4\pi i\sqrt{U}(\frac{\sqrt{n}}{\sqrt{U}}%
-b)(B_{1}+\theta +\frac{u}{2\sqrt{U}})^{\alpha +1}}  \notag \\
&&+\frac{1}{2(B_{1}+\theta )^{\alpha }}+\frac{e(\frac{\sqrt{n}m}{\sqrt{U}})}{%
2(B_{1}+\theta +\frac{m}{2\sqrt{U}})^{\alpha }}  \notag \\
&=&\sum\limits_{-U^{10}\leq b\leq U^{10},b\neq 0}\frac{e(\frac{\sqrt{n}u}{%
\sqrt{U}})}{2\pi i(\frac{\sqrt{n}}{\sqrt{U}}-b)(B_{1}+\frac{u}{2\sqrt{U}}%
)^{\alpha }}\mbox {\LARGE \textbar
}_{u=0}^{u=m}  \notag \\
&&+\frac{1}{2(B_{1})^{\alpha }}+\frac{e(\frac{\sqrt{n}m}{\sqrt{U}})}{2(B_{1}+%
\frac{m}{2\sqrt{U}})^{\alpha }}+\rho _{\delta }^{\ast }(n,\alpha ,\theta ),
\end{eqnarray}%
and
\begin{eqnarray*}
&&\rho _{\delta }^{\ast }(n,\alpha ,\theta )= \\
&=&\sum\limits_{-U^{10}\leq b\leq U^{10},b\neq 0}\frac{e(\frac{\sqrt{n}u}{%
\sqrt{U}})}{2\pi i(\frac{\sqrt{n}}{\sqrt{U}}-b)}(\frac{1}{(B_{1}+\theta +%
\frac{u}{2\sqrt{U}})^{\alpha }}-\frac{1}{(B_{1}+\frac{u}{2\sqrt{U}})^{\alpha
}})\mbox {\LARGE
\textbar }_{u=0}^{u=m} \\
&&+\frac{1}{2}(\frac{1}{(B_{1}+\theta )^{\alpha }}-\frac{1}{B_{1}^{\alpha }}%
)+\frac{e(\frac{\sqrt{n}u}{\sqrt{U}})}{2}(\frac{1}{(B_{1}+\theta +\frac{u}{2%
\sqrt{U}})^{\alpha }}-\frac{1}{(B_{1}+\frac{u}{2\sqrt{U}})^{\alpha }}) \\
&&+\sum\limits_{-U^{10}\leq b\leq U^{10},b\neq 0}\int_{0}^{m}\frac{\alpha e((%
\frac{\sqrt{n}}{\sqrt{U}}-b)u)du}{4\pi i\sqrt{U}(\frac{\sqrt{n}}{\sqrt{U}}%
-b)(B_{1}+\theta +\frac{u}{2\sqrt{U}})^{\alpha +1}}.
\end{eqnarray*}%
By Lemma 3.5,
\begin{equation}
\sum\limits_{\sqrt{n}\leq \frac{\sqrt{U}}{2}}d(n)n^{-\frac{1}{4}%
}\sum\limits_{\alpha =0}^{2}\gamma _{\alpha }(\frac{i}{4\pi \sqrt{n}}%
)^{\alpha }e(2\sqrt{n}(B_{1}+\theta )+1/8)\rho _{\delta }^{\ast }(n,\alpha
,\theta )\ll U^{-1/4+\varepsilon (U)}.
\end{equation}%
Clearly, as $\alpha =1,2$, the sum (8.30) is $O(U^{-\frac{1}{4}+\varepsilon
(U)})$. Using
\begin{equation*}
A\int_{-\frac{5L}{A}}^{\frac{5L}{A}}e^{-\pi A^{2}\theta ^{2}}e(2\sqrt{n}%
\theta )d\theta =e^{-\frac{4\pi n}{A^{2}}}+O(U^{-10},
\end{equation*}%
it follows from (8.30), (8.31) and (8.32) that
\begin{eqnarray*}
&&\mathrm{Re}\ \delta _{Q_{0}}(B_{1})= \\
&=&\sum\limits_{\sqrt{n}\leq \frac{\sqrt{U}}{2}}d(n)n^{-\frac{1}{4}}e^{-%
\frac{4\pi n}{A^{2}}}(\sum\limits_{-U^{10}\leq b\leq U^{10},b\neq 0}\frac{%
\cos (4\pi \sqrt{n}(B_{1}+\frac{u}{2\sqrt{U}})-\frac{\pi }{4})}{2\pi (\frac{%
\sqrt{n}}{\sqrt{U}}-b)}\mbox {\LARGE
\textbar }_{u=0}^{u=m} \\
&&+\frac{\cos (4\pi \sqrt{n}B_{1}+\frac{\pi }{4})}{2}+\frac{\cos (4\pi \sqrt{%
n}(B_{1}+\frac{m}{2\sqrt{U}})+\frac{\pi }{4})}{2})+O(U^{-\frac{1}{4}%
+\varepsilon (U)}) \\
&=&\sum\limits_{\sqrt{n}\leq \frac{\sqrt{U}}{2}}d(n)n^{-\frac{1}{4}}e^{-%
\frac{4\pi n}{A^{2}}}((\frac{\cos \pi \frac{\sqrt{n}}{\sqrt{U}}}{2\sin \pi
\frac{\sqrt{n}}{\sqrt{U}}}-\frac{\sqrt{U}}{2\pi \sqrt{n}})\cos (4\pi \sqrt{n}%
(B_{1}+\frac{u}{2\sqrt{U}})-\frac{\pi }{4})\mbox
{\LARGE \textbar }_{u=0}^{u=m}
\end{eqnarray*}%
\begin{equation}
+\frac{\cos (4\pi \sqrt{n}B_{1}+\frac{\pi }{4})}{2}+\frac{\cos (4\pi \sqrt{n}%
(B_{1}+\frac{m}{2\sqrt{U}})+\frac{\pi }{4})}{2})+O(U^{-\frac{1}{4}%
+\varepsilon (U)}),
\end{equation}%
where the last equality is given by the following (9.25).

\ \

\section{EVALUATION OF $\Omega$}

By (5.13), (8.1) and (8.3),
\begin{eqnarray}
\sum\limits_{k=0}^{m}S(B) &=&2\sqrt{2}\mathrm{Re}\int_{\xi _{1}}^{\xi
_{2}}\sum\limits_{k=0}^{m}T(B+\xi )d\xi +O(U^{-1})  \notag \\
&=&2\sqrt{2}\int_{\xi _{1}}^{\xi _{2}}\mathrm{Re}\ Q(B_{1})d\xi +O(U^{-1})
\notag \\
&=&4\sqrt{2U}\int_{\xi _{1}}^{\xi _{2}}d\xi \sum\limits_{\sqrt{n}\leq \frac{%
\sqrt{U}}{2}}d(n)n^{-1/4}e^{-\frac{4\pi n}{A^{2}}}\int_{0}^{\frac{m}{2\sqrt{U%
}}}\cos (4\pi \sqrt{n}(B_{1}+x)+\pi /4)dx  \notag \\
&&+\mathrm{Re}\int_{\xi _{1}}^{\xi _{2}}\delta _{Q}(B_{1})d\xi +O_{\eta
}(U^{-3/4+\varepsilon (U)}),
\end{eqnarray}%
where $0<\varepsilon (U)\ll 1/\log \log U$, $\delta _{Q}(B_{1})\ll
U^{1/4+\varepsilon (U)}$, $\xi _{1}=1/16\sqrt{U}$, $\xi
_{2}=3/16\sqrt{U}$ and
\begin{equation*}
B_{1}=B+\xi -\frac{k}{2\sqrt{U}}=\sqrt{U}+\frac{j}{2\sqrt{U}}+\eta +\xi .
\end{equation*}%
Let
\begin{equation}
R(\eta ,m)=\frac{1}{\sqrt{V}}\sum\limits_{-\sqrt{V}L\leq j\leq \sqrt{V}L}e^{-%
\frac{\pi j^{2}}{V}}\sum\limits_{k=0}^{m}S(B),
\end{equation}%
then by (9.1),
\begin{eqnarray}
R(\eta ,m) &=&4\sqrt{2U}\sum\limits_{\sqrt{n}\leq \frac{\sqrt{U}}{2}%
}d(n)n^{-1/4}e^{-\frac{4\pi n}{A^{2}}}\int_{\xi _{1}}^{\xi _{2}}d\xi
\int_{0}^{\frac{m}{2\sqrt{U}}}v(\eta ,\xi ,x)dx  \notag \\
&&+\mathrm{Re}\ \delta _{R}(\eta ,m)+O_{\eta }(U^{-3/4+\varepsilon (U)}),
\end{eqnarray}%
where
\begin{equation}
\delta _{R}(\eta ,m)\ll U^{-1/4+\varepsilon }
\end{equation}%
or accurately,
\begin{equation}
\delta _{R}(\eta ,m)=\frac{1}{\sqrt{V}}\sum\limits_{-\sqrt{V}L\leq j\leq
\sqrt{V}L}e^{-\frac{\pi j^{-2}}{V}}\int_{\xi _{1}}^{\xi _{2}}\delta
_{Q}(B_{1})d\xi
\end{equation}%
((8.2) and (9.5) deduce (9.4)), moreover,
\begin{equation*}
v(\eta ,\xi ,x)=\frac{1}{\sqrt{V}}\sum\limits_{-\sqrt{V}L\leq j\leq \sqrt{V}%
L}e^{-\frac{\pi j^{-2}}{V}}\cos (4\pi n(\sqrt{U}+\frac{j}{2\sqrt{U}}+\xi
+\eta +x)+\pi /4)
\end{equation*}%
\begin{equation*}
=\mathrm{Re}\ e(2\sqrt{n}(\sqrt{U}+\xi +\eta +x)+1/8)\frac{1}{\sqrt{V}}%
\sum\limits_{j=-\infty }^{\infty }e^{-\frac{\pi j^{-2}}{V}}e(\frac{\sqrt{n}j%
}{\sqrt{U}})+0(U^{-10}).
\end{equation*}%
By Lemma 3.2,
\begin{equation*}
v(\eta ,\xi ,x)=\mathrm{Re}\ e(2\sqrt{n}(\sqrt{U}+\xi +\eta
+x)+1/8)\sum\limits_{j=-\infty }^{\infty }e^{-\pi V(\frac{\sqrt{n}}{\sqrt{U}}%
-j)^{2}}+0(U^{-10})
\end{equation*}%
\begin{eqnarray}
&=&e^{-\frac{\pi Vn}{U}}\cos (4\pi \sqrt{n}(\sqrt{U}+\xi +\eta +x)+\pi
/4)+O(\sum\limits_{j=-\infty ,j\neq 0}^{\infty }e^{-\pi v(\frac{\sqrt{n}}{%
\sqrt{U}}-j)^{2}})+0(U^{-10})  \notag \\
&=&e^{-\frac{\pi Vn}{U}}\cos (4\pi \sqrt{n}(\sqrt{U}+\xi +\eta +x)+\pi
/4)+0(U^{-10})
\end{eqnarray}%
for $\frac{\sqrt{n}}{\sqrt{U}}\leq 1/2$. Since
\begin{equation*}
e^{-\frac{\pi Vn}{U}}\ll e^{-L^{2}},n>\frac{UL^{2}}{V},
\end{equation*}%
then it follows from (9.3) and (9.6) that
\begin{eqnarray}
R(\eta ,m) &=&4\sqrt{2U}\sum\limits_{n\leq \frac{UL^{2}}{V}%
}d(n)n^{-1/4}e^{-\pi n(\frac{V}{U}+\frac{4}{A^{2}})}  \notag \\
&&\times \int_{\xi _{1}}^{\xi _{2}}d\xi \int_{0}^{\frac{m}{2\sqrt{U}}}\cos
(4\pi \sqrt{n}(\sqrt{U}+\xi +\eta +x)+\pi /4)dx  \notag \\
&&+\delta _{R}(\eta ,m)+O(U^{-3/4+\varepsilon })  \notag \\
&=&\frac{\sqrt{2U}}{\pi }\sum\limits_{n\leq \frac{UL^{2}}{V}%
}d(n)n^{-3/4}e^{-\pi n(\frac{V}{U}+\frac{4}{A^{2}})}  \notag \\
&&\times \int_{\xi _{1}}^{\xi _{2}}\cos (4\pi \sqrt{n}(\sqrt{U}+\xi +\eta
+x)-\pi /4)dx\mbox {\LARGE
\textbar }_{x=0}^{x=\frac{m}{2\sqrt{U}}}  \notag \\
&&+\delta _{R}(\eta ,m)+O_{\eta }(U^{-3/4+\varepsilon (U)}).
\end{eqnarray}%
It is obvious that
\begin{eqnarray}
f(\eta )\mbox {\LARGE \textbar
}_{\eta _{K}} &=&f(\eta _{1}+\cdot +\eta _{K})\mbox {\LARGE \textbar
}_{\eta _{1}=0}^{\eta _{1}=\frac{k_{1}}{\sqrt{U}}}\cdots
\mbox {\LARGE
\textbar }_{\eta _{K}=0}^{\eta _{K}=\frac{k_{1}}{\sqrt{U}}}  \notag \\
&=&\sum\limits_{a=0}^{K}(-1)^{K-a}\left( \frac{K}{a}\right) f(\frac{ak}{%
\sqrt{U}})  \notag \\
&=&(-1)^{K}f(0)+\sum\limits_{a=1}^{K}(-1)^{K-a}\left( \frac{K}{a}\right) f(%
\frac{ak_{1}}{\sqrt{U}}).
\end{eqnarray}%
By (9.7), (9.8) and definition (1.28) of $O_{\eta }$ we have
\begin{equation}
R(\eta ,m)\mbox {\LARGE \textbar
}_{\eta _{K}}=R_{0}+R_{m}+R_{mk}+\delta _{R}(\eta ,m)%
\mbox {\LARGE
\textbar }_{\eta _{K}}+O(U^{-3/4+\varepsilon (U)}),
\end{equation}%
where $0<\varepsilon (U)\ll 1/\log \log U,$
\begin{eqnarray}
R_{0} &=&(-1)^{K+1}\frac{\sqrt{2U}}{\pi }\sum\limits_{n\leq \frac{UL^{2}}{V}%
}d(n)n^{-3/4}e^{-\pi n(\frac{V}{U}+\frac{4}{A^{2}})}  \notag \\
&&\times \int_{\xi _{1}}^{\xi _{2}}\cos (4\pi \sqrt{n}(\sqrt{U}+\xi )-\pi
/4)d\xi , \\
R_{m} &=&\frac{(-1)^{K}\sqrt{2U}}{\pi }\sum\limits_{n\leq \frac{UL^{2}}{V}%
}d(n)n^{-3/4}e^{-\pi n(\frac{V}{U}+\frac{4}{A^{2}})}  \notag \\
&&\times \int_{\xi _{1}}^{\xi _{2}}\cos (4\pi \sqrt{n}(\sqrt{U}+\xi +\frac{m%
}{2\sqrt{U}})-\pi /4)d\xi , \\
R_{mk} &=&\frac{\sqrt{2U}}{\pi }\sum\limits_{n\leq \frac{UL^{2}}{V}%
}d(n)n^{-3/4}e^{-\pi n(\frac{V}{U}+\frac{4}{A^{2}})}\int_{\xi _{1}}^{\xi
_{2}}d\xi \sum\limits_{a=1}^{K}(-1)^{K-a}\left( \frac{K}{a}\right)  \notag \\
&&\times \cos (4\pi \sqrt{n}(\sqrt{U}+\xi +x+\frac{ak}{\sqrt{U}})-\pi /4)%
\mbox {\LARGE \textbar }_{x=0}^{x=\frac{m}{2\sqrt{U}}}.
\end{eqnarray}%
Therefore,
\begin{eqnarray}
\Omega &=&\frac{1}{\sqrt{V}}\sum\limits_{-\sqrt{V}L\leq j\leq \sqrt{V}L}e^{-%
\frac{\pi j^{2}}{V}}\sum\limits_{k,m=m_{0}}^{2m_{0}}S(B)%
\mbox {\LARGE \textbar
}_{\eta _{K}}=\sum\limits_{k,m=m_{0}}^{2m_{0}}R(\eta ,m)%
\mbox {\LARGE
\textbar }_{\eta _{K}}  \notag \\
&=&\Omega _{0}+\Omega _{1}+\Omega _{2}+\Omega _{\delta
}+O(U^{1/4+\varepsilon (U)}),
\end{eqnarray}%
\noindent where
\begin{eqnarray}
\Omega _{0} &=&\sum\limits_{k,m=m_{0}}^{2m_{0}}R_{0}=(m_{0}+1)^{2}R_{0}, \\
\Omega _{1}
&=&\sum\limits_{k,m=m_{0}}^{2m_{0}}R_{m}=(m_{0}+1)\sum%
\limits_{m=m_{0}}^{2m_{0}}R_{m}, \\
\Omega _{2} &=&\sum\limits_{k,m=m_{0}}^{2m_{0}}R_{mk}, \\
\Omega _{\delta } &=&\sum\limits_{k,m=m_{0}}^{2m_{0}}\delta _{R}(\eta ,m)%
\mbox
{\LARGE \textbar }_{\eta _{K}}
\end{eqnarray}%
(we have replaced $k_{1}\ by\ k$), $R_{0},R_{m},R_{mk}$ and $\delta
_{R}(\eta ,m)$ are (9.10), (9.11), (9.12) and (9.5), respectively. By (9.4),
\begin{equation}
\Omega _{\delta }\ll U^{3/4+\varepsilon (U)}.
\end{equation}

We first evaluate $\Omega _{1}$. By (9.11) and (9.15),
\begin{equation}
\Omega _{1}=C_{1}\sqrt{U}(m_{0}+1)\sum\limits_{\sqrt{n}\leq \frac{UL^{2}}{V}%
}d(n)n^{-3/4}e^{-\pi n(\frac{V}{U}+\frac{4}{A^{2}})}\int_{\xi _{1}}^{\xi
_{2}}w_{1n}(\xi )d\xi ,
\end{equation}%
where $C_{1}=O(1)$ is real,
\begin{equation}
w_{1n}(\xi )=\sum\limits_{m=m_{0}}^{2m_{0}}\cos (4\pi \sqrt{n}(\sqrt{U}+\xi +%
\frac{m}{2\sqrt{U}})-\pi /4).
\end{equation}%
By Lemma 3.3,
\begin{eqnarray}
w_{1n}(\xi ) &=&\int_{m_{0}}^{2m_{0}}\cos (4\pi \sqrt{n}(\sqrt{U}+\xi +\frac{%
u}{2\sqrt{U}})-\pi /4)du+w_{\delta }^{\prime }+w_{\delta }^{\prime \prime
}+O(U^{-10})  \notag \\
&=&\frac{\sqrt{U}}{2\pi \sqrt{n}}\cos (4\pi \sqrt{n}(\sqrt{U}+\xi +\frac{%
um_{0}}{2\sqrt{U}})-3\pi /4)\mbox
{\LARGE \textbar }_{u=1}^{u=2}  \notag \\
&&+w_{\delta }^{\prime }+w_{\delta }^{\prime \prime }+O(U^{-10})
\end{eqnarray}%
where
\begin{eqnarray}
w_{\delta }^{\prime } &=&\frac{1}{2}\sum\limits_{b=1}^{2}\cos (4\pi \sqrt{n}(%
\sqrt{U}+\xi +\frac{bm_{0}}{2\sqrt{U}})-\pi /4), \\
w_{\delta }^{\prime \prime } &=&\mathrm{Re}\sum\limits_{-U^{10}\leq b\leq
U^{10},b\neq 0}\int_{m_{0}}^{2m_{0}}e(2\sqrt{n}(\sqrt{U}+\xi +\frac{u}{2%
\sqrt{U}})-bu-1/8)du,  \notag \\
&=&\mathrm{Re}\frac{1}{2\pi i}\sum\limits_{-U^{10}\leq b\leq U^{10},b\neq 0}%
\frac{e(2\sqrt{n}(\sqrt{U}+\xi +\frac{u}{2\sqrt{U}})-1/8)}{(\frac{\sqrt{n}}{%
\sqrt{U}}-b)}\mbox {\LARGE \textbar }_{u=m_{0}}^{u=2m_{0}}
\end{eqnarray}%
for $e(-bu)=1,u=m_{0},2m_{0}$. It is easy to test Fourier's formula:
\begin{equation}
\cos (\alpha x)=\frac{1}{\pi }\sin (\alpha \pi )(\frac{1}{\alpha }%
+\sum\limits_{b=1}^{\infty }\frac{2\alpha (-1)^{b}\cos (bx)}{\alpha
^{2}-b^{2}}),
\end{equation}%
\begin{equation*}
0<\alpha <1,\ -\pi \leq x\leq \pi .
\end{equation*}%
Taking $x=\pi ,\alpha =\sqrt{n}/\sqrt{U}$, we have
\begin{eqnarray}
\sum\limits_{-U^{10}\leq b\leq U^{10},b\neq 0}\frac{1}{\frac{\sqrt{n}}{\sqrt{%
U}}-b} &=&\sum\limits_{b=1}^{\infty }\frac{2\frac{\sqrt{n}}{\sqrt{U}}}{(%
\frac{\sqrt{n}}{\sqrt{U}})^{2}-b^{2}}+O(U^{-10})  \notag \\
&=&\frac{\pi \cos (\pi \frac{\sqrt{n}}{\sqrt{U}})}{\sin (\pi \frac{\sqrt{n}}{%
\sqrt{U}})}-\frac{\sqrt{U}}{\sqrt{n}}  \notag \\
&=&-\frac{\pi ^{2}\sqrt{n}}{3\sqrt{U}}(1+\varphi (\frac{\sqrt{n}}{\sqrt{U}}%
))+O(U^{-10}),
\end{eqnarray}%
where $\varphi (\frac{\sqrt{n}}{\sqrt{U}})=C_{1}\frac{n}{U}+C_{2}(\frac{n}{U}%
)^{2}+\cdots $ for $\sqrt{n}\leq \frac{UL^{2}}{V}$. By Lemma 3.5,
\begin{equation*}
\sqrt{U}(m_{0}+1)\int_{\xi _{1}}^{\xi _{2}}d\xi \sum\limits_{\sqrt{n}\leq
\frac{UL^{2}}{V}}d(n)n^{-3/4}e^{-\pi n(\frac{V}{U}+\frac{4}{A^{2}}%
)}w_{\delta }^{\prime \prime }\ \ \ \ \ \ \ \ \ \ \ \ \ \ \ \ \ \ \ \ \ \ \
\ \ \ \ \ \ \
\end{equation*}%
\begin{equation}
\ll \sqrt{U}(m_{0}+1)\int_{\xi _{1}}^{\xi _{2}}U^{-1/2+1/4+\varepsilon
(U)}d\xi \ll U^{1/4+\varepsilon (U)}.
\end{equation}%
It follows from (9.19), (9.21), (9.22) and (9.26) that
\begin{eqnarray}
\Omega _{1} &=&C_{1}U(m_{0}+1)\int_{\xi _{1}}^{\xi _{2}}d\xi
\sum\limits_{n\leq \frac{UL^{2}}{V}}d(n)n^{-5/4}e^{-\pi n(\frac{V}{U}+\frac{4%
}{A^{2}})}  \notag \\
&&\times \cos (4\pi \sqrt{n}(\sqrt{U}+\xi +\frac{um_{0}}{2\sqrt{U}})-3\pi /4)%
\mbox
{\LARGE \textbar }_{u=1}^{u=2}  \notag \\
&&+\Omega _{1}\delta +O(X^{1/4+\varepsilon (X)}),
\end{eqnarray}%
where $U=[X]$,
\begin{eqnarray}
\Omega _{1\delta } &=&C_{2}\sqrt{U}(m_{0}+1)\sum\limits_{b=1}^{2}\int_{\xi
_{1}}^{\xi _{2}}\sum\limits_{n\leq \frac{UL^{2}}{V}}d(n)n^{-3/4}e^{-\pi n(%
\frac{V}{U}+\frac{4}{A^{2}})}  \notag \\
&&\times \cos (4\pi \sqrt{n}(\sqrt{U}+\xi +\frac{bm_{0}}{2\sqrt{U}})-\pi
/4)d\xi .
\end{eqnarray}%
Furthermore,
\begin{equation}
e^{-\pi n(\frac{V}{U}+\frac{4}{A^{2}})}-1=-\pi n(\frac{V}{U}+\frac{4}{A^{2}}%
)\int_{0}^{1}e^{-\theta \pi n(\frac{V}{U}+\frac{4}{A^{2}})}d\theta ,
\end{equation}%
and
\begin{eqnarray}
&&\cos (4\pi \sqrt{n}(\sqrt{U}+\xi +\frac{um_{0}}{2\sqrt{U}})-3\pi /4)
\notag \\
&=&\mathrm{Re}\ e(2\sqrt{n}(\sqrt{X}+\frac{um_{0}}{2\sqrt{X}})-3/8)e(2\sqrt{n%
}(\xi +\sqrt{U}-\sqrt{X}+\frac{um_{0}}{2}(\frac{1}{\sqrt{U}}-\frac{1}{\sqrt{X%
}})))  \notag \\
&=&\cos (4\pi \sqrt{n}(\sqrt{X}+\frac{um_{0}}{2\sqrt{X}})-3\pi /4)  \notag \\
&&-4\pi \sqrt{n}(\xi +\sqrt{U}-\sqrt{X}+\frac{um_{0}}{2}(\frac{1}{\sqrt{U}}-%
\frac{1}{\sqrt{X}}))\cos (4\pi \sqrt{n}(\sqrt{X}+\frac{um_{0}}{2\sqrt{U}}%
)-\pi /4)  \notag \\
&&+\mathrm{Re}(e(2\sqrt{n}(\sqrt{X}+\frac{um_{0}}{2\sqrt{X}})-3/8)\delta
_{nu}(\xi )),
\end{eqnarray}%
where
\begin{eqnarray}
\delta _{nu}(\xi ) =n(4\pi i(\xi +\sqrt{U}-\sqrt{X}+\frac{um_{0}}{2}(\frac{1%
}{\sqrt{U}}-\frac{1}{\sqrt{X}})))^{2} \ \ \ \ \ \ \ \ \ \ \ \ \ \ \ \ \ \ \
\   \notag \\
\times \int_{0}^{1}\int_{0}^{1}\theta e(2\theta \theta _{1}\sqrt{n}(\xi +%
\sqrt{U}-\sqrt{X}+\frac{um_{0}}{2}(\frac{1}{\sqrt{U}}-\frac{1}{\sqrt{X}}%
)))d\theta d\theta _{1}.
\end{eqnarray}%
Using (9.29) and Lemma 3.5,
\begin{equation*}
U(m_{0}+1)\sum\limits_{n\leq \frac{UL^{2}}{V}}d(n)n^{-5/4}(e^{-\pi n(\frac{V%
}{U}+\frac{4}{A^{2}})}-1)\int_{\xi _{1}}^{\xi _{2}}\cos (4\pi \sqrt{n}(\sqrt{%
U}+\xi +\frac{um_{0}}{2\sqrt{U}})-3\pi /4)d\xi
\end{equation*}%
\begin{equation}
\ll U^{1/4+\varepsilon (U)}.
\end{equation}%
Since $U=[X],\ \xi =O(U^{-\frac{1}{2}})$, then
\begin{equation}
\xi +\sqrt{U}-\sqrt{X}=\xi +\frac{[X]-X}{\sqrt{[X]}+\sqrt{X}}=\xi -\frac{%
\{X\}}{2\sqrt{X}}+O(\frac{1}{X}),
\end{equation}%
\begin{equation}
\frac{um_{0}}{2}(\frac{1}{\sqrt{U}}-\frac{1}{\sqrt{X}}))=O(\frac{1}{X}),
\end{equation}%
for $m_{0}\asymp \sqrt{U}L^{-2}$. It follows from (9.31), (9.33), (9.34) and
Lemma 3.5 that
\begin{equation}
U(m_{0}+1)\sum\limits_{n\leq \frac{UL^{2}}{V}}d(n)n^{-5/4}\int_{\xi
_{1}}^{\xi _{2}}e(2\sqrt{n}(\sqrt{X}+\frac{um_{0}}{2\sqrt{X}})-3/8)\delta
_{nu}(\xi )d\xi \ll X^{1/4+\varepsilon (X)}.
\end{equation}%
Therefore, by (9.27), (9.32), (9.33), (9.34) and (9.35) we obtain that
\begin{eqnarray}
\Omega _{1} &=&C_{1}U(m_{0}+1)\sum\limits_{n\leq \frac{UL^{2}}{V}%
}d(n)n^{-5/4}\int_{\xi _{1}}^{\xi _{2}}(\cos 4\pi \sqrt{n}(\sqrt{X}+\frac{%
um_{0}}{2\sqrt{X}})-3\pi /4)\ \ \ \ \ \ \ \ \ \ \ \   \notag \\
&&+C_{2}\sqrt{n}(\xi -\frac{\{X\}}{2\sqrt{X}}+O(\frac{1}{X}))\cos (4\pi
\sqrt{n}(\sqrt{X}+\frac{um_{0}}{2\sqrt{X}})-\pi /4)d\xi  \notag \\
&&+\Omega _{1\delta }+O(X^{1/4+\varepsilon (U)})  \notag \\
&=&C_{0}\sqrt{U}(m_{0}+1)\sum\limits_{n\leq \frac{UL^{2}}{V}%
}d(n)n^{-5/4}\cos (4\pi \sqrt{n}(\sqrt{X}+\frac{um_{0}}{2\sqrt{X}})-3\pi /4)%
\mbox {\LARGE \textbar
}_{u=1}^{u=2}  \notag \\
&&+(m_{0}+1)(C_{1}+C_{2}\{X\})\sum\limits_{\sqrt{n}\leq \frac{UL^{2}}{V}%
}d(n)n^{-3/4}\cos (4\pi \sqrt{n}(\sqrt{X}+\frac{um_{0}}{2\sqrt{X}})-\pi /4)%
\mbox {\LARGE \textbar
}_{u=1}^{u=2}  \notag \\
&&+\Omega _{1\delta }+O(X^{1/4+\varepsilon (X)})\
\end{eqnarray}%
for $\xi _{1}=1/16\sqrt{U}$, $\xi _{2}=3/16\sqrt{U}$. In the same way, by
(9.28),%
\begin{equation}
\Omega _{1\delta }=(m_{0}+1)\sum\limits_{a=1}^{2}C_{a}\sum\limits_{\sqrt{n}%
\leq \frac{UL^{2}}{V}}d(n)n^{-3/4}\cos (4\pi \sqrt{n}(\sqrt{X}+\frac{am_{0}}{%
2\sqrt{X}})-\pi /4)+O(X^{1/4+\varepsilon (X)}).
\end{equation}%
By Lemma 3.4,
\begin{equation}
\sum\limits_{\frac{UL^{2}}{V}<n\leq \frac{X_{2}L^{2}}{V}}d(n)n^{-5/4}\cos
(4\pi \sqrt{n}(\sqrt{X}+\frac{um_{0}}{2X})-3\pi /4)\ll X^{-3/4+\varepsilon }
\end{equation}%
and
\begin{equation}
\sum\limits_{\frac{UL^{2}}{V}<n\leq \frac{X_{2}L^{2}}{V}}d(n)n^{-3/4}\cos
(4\pi \sqrt{n}(\sqrt{X}+\frac{am_{0}}{2X})-\pi /4)\ll X^{-1/4+\varepsilon }.
\end{equation}%
Therefore, by (9.36), (9.37), (9.38) and (9.39), we may rewrite $\Omega _{1}$
in the form
\begin{eqnarray*}
\Omega _{1} &=&(m_{0}+1)\sqrt{X}\sum\limits_{a=1}^{2}C_{a}\sum\limits_{n\leq
\frac{X_{2}L^{2}}{V}}d(n)n^{-5/4}\cos (4\pi \sqrt{n}(\sqrt{X}+\frac{am_{0}}{2%
\sqrt{X}})-3\pi /4) \\
&&+(m_{0}+1)\sum\limits_{a=1}^{2}(C_{2}(a)+C_{3}(a)\{X\})
\end{eqnarray*}%
\begin{equation}
\ \ \ \ \ \ \ \ \ \ \ \times \sum\limits_{n\leq \frac{X_{2}L^{2}}{V}%
}d(n)n^{-3/4}\cos (4\pi \sqrt{n}(\sqrt{X}+\frac{am_{0}}{2\sqrt{X}})-\pi /4)%
\mbox {\LARGE \textbar }_{u=1}^{u=2}+O(X^{1/4+\varepsilon (U)})
\end{equation}%
for $X_{1}\leq X\leq X_{2}$, $U=[X]$, where $C_{j}(a)=O(1)$ is real and
independent of $X$.

Next we evaluate $\Omega _{2}$ in (9.16). By (9.16) and (9.12),
\begin{eqnarray}
\Omega _{2} &=&\frac{\sqrt{2U}}{\pi }\sum\limits_{n\leq \frac{UL^{2}}{V}%
}d(n)n^{-3/4}e^{-\pi n(\frac{V}{U}+\frac{4}{A^{2}})}  \notag \\
&&\times \sum\limits_{a=1}^{K}(-1)^{K}\left( \frac{K}{a}\right) \int_{\xi
_{1}}^{\xi _{2}}(w_{na}(\xi ,\frac{m}{2\sqrt{U}})-w_{na}(\xi ,0))d\xi  \notag
\\
&=&\Omega _{2}^{\prime }+\Omega _{20},\ \ \ \ \mathrm{say},\
\end{eqnarray}%
where
\begin{eqnarray}
w_{na}(\xi ,\frac{m}{2\sqrt{U}}) &=&\sum\limits_{k,m=m_{0}}^{2m_{0}}\cos
(4\pi \sqrt{n}(\sqrt{U}+\xi +\frac{ak}{\sqrt{U}}+\frac{m}{2\sqrt{U}})-\pi /4)
\\
w_{na}(\xi ,0) &=&\sum\limits_{k,m=m_{0}}^{2m_{0}}\cos (4\pi \sqrt{n}(\sqrt{U%
}+\xi +\frac{ak}{\sqrt{U}})-\pi /4)  \notag \\
&=&(m_{0}+1)\sum\limits_{k=m_{0}}^{2m_{0}}\cos (4\pi \sqrt{n}(\sqrt{U}+\xi +%
\frac{ak}{\sqrt{U}})-\pi /4).
\end{eqnarray}%
Repeating the evaluation of (9.40) we have
\begin{eqnarray}
\Omega _{20} &=&-\frac{\sqrt{2U}}{\pi }(m_{0}+1)\sum\limits_{n\leq \frac{%
UL^{2}}{V}}d(n)n^{-3/4}e^{-\pi n(\frac{V}{U}+\frac{4}{A^{2}})}  \notag \\
&&\times \sum\limits_{a=1}^{K}(-1)^{K}\left( \frac{K}{a}\right) \int_{\xi
_{1}}^{\xi _{2}}\sum\limits_{k=m_{0}}^{2m_{0}}\cos (4\pi \sqrt{n}(\sqrt{U}%
+\xi +\frac{ak}{\sqrt{U}})-\pi /4)d\xi  \notag \\
&=&(m_{0}+1)\sqrt{X}\sum\limits_{a=1}^{4K}C_{1}(a)\sum\limits_{n\leq \frac{%
X_{2}L^{2}}{V}}d(n)n^{-5/4}\cos (4\pi \sqrt{n}(\sqrt{X}+\frac{am_{0}}{2\sqrt{%
X}})-3\pi /4)  \notag \\
&&+(m_{0}+1)\sum\limits_{a=1}^{4K}(C_{2}(a)+C_{3}(a)\{X\})  \notag \\
&&\times \sum\limits_{n\leq \frac{X_{2}L^{2}}{V}}d(n)n^{-3/4}\cos (4\pi
\sqrt{n}(\sqrt{X}+\frac{am_{0}}{2\sqrt{X}})-\pi /4)+O(X^{1/4+\varepsilon
(U)}),
\end{eqnarray}%
where $C_{j}(a)$ is real, and
\begin{equation}
|C_{j}(a)|\ll \sup\limits_{0\leq a\leq K}\left( \frac{K}{a}\right) \ll
2^{K}\ll exp(O(\frac{L}{\log (L)})).
\end{equation}

Next we evaluate $\Omega _{2}^{\prime }$ in (9.41). By (9.42) and Lemma 3.3
\begin{equation}
w_{na}(\xi ,\frac{m}{2\sqrt{U}})=\mathrm{Re}\ (e(2\sqrt{n}(\sqrt{U}+\xi
)-1/8)\rho _{na}),
\end{equation}%
where
\begin{equation}
\rho _{na}=\sum\limits_{k,m=m_{0}}^{2m_{0}}e(2\sqrt{n}(\frac{ak}{\sqrt{U}}+%
\frac{m}{2\sqrt{U}}))=I_{1}I_{2}+I_{1}\sigma _{2}+I_{2}\sigma _{1}+\sigma
_{1}\sigma _{2}+O(U^{-9}),
\end{equation}%
moreover,
\begin{eqnarray}
I_{1} &=&\int_{m_{0}}^{2m_{0}}e(\frac{2\sqrt{n}au}{\sqrt{U}})du=\frac{\sqrt{U%
}}{2\pi ia\sqrt{n}}e(\frac{2\sqrt{n}aum_{0}}{\sqrt{U}})%
\mbox {\LARGE
\textbar }_{u=1}^{u=2}, \\
I_{2} &=&\int_{m_{0}}^{2m_{0}}e(\frac{\sqrt{n}u}{\sqrt{U}})du=\frac{\sqrt{U}%
}{2\pi i\sqrt{n}}e(\frac{\sqrt{n}um_{0}}{\sqrt{U}})\mbox {\LARGE
\textbar }_{u=1}^{u=2},  \notag \\
\sigma _{1} &=&\frac{1}{2}\sum\limits_{b=1}^{2}e(\frac{2abm_{0}}{\sqrt{U}}%
)+\sum\limits_{-U^{10}\leq b\leq U^{10},b\neq 0}\frac{e(\frac{2\sqrt{n}%
am_{0}u}{\sqrt{U}})}{2\pi i(\frac{2a\sqrt{n}}{\sqrt{U}}-b)}%
\mbox {\LARGE \textbar
}_{u=1}^{u=2}  \notag \\
&=&\frac{1}{2}\sum\limits_{b=1}^{2}e(\frac{2abm_{0}}{\sqrt{U}})-\frac{\pi
^{2}}{3}\frac{e(\frac{2\sqrt{n}am_{0}u}{\sqrt{U}})}{2\pi i}\frac{\sqrt{n}}{%
\sqrt{U}}(1+\varphi (\frac{2a\sqrt{n}}{\sqrt{U}}))\mbox
{\LARGE \textbar }_{u=1}^{u=2} \\
\sigma _{2} &=&\frac{1}{2}\sum\limits_{b=1}^{2}e(\frac{\sqrt{n}bm_{0}}{\sqrt{%
U}})-\frac{\pi ^{2}}{3}\frac{e(\frac{\sqrt{n}m_{0}u}{\sqrt{U}})}{2\pi i}%
\frac{\sqrt{n}}{\sqrt{U}}(1+\varphi (\frac{\sqrt{n}}{\sqrt{U}}))%
\mbox {\LARGE \textbar
}_{u=1}^{u=2}.
\end{eqnarray}%
The last equality is given by (9.25). Since $\sigma _{1}\sigma _{2}\ll 1,$
\noindent then
\begin{equation*}
\sqrt{U}\sum\limits_{\sqrt{n}\leq \frac{UL^{2}}{V}}d(n)n^{-3/4}e^{-\pi n(%
\frac{V}{U}+\frac{4}{A^{2}})}\ \ \ \ \ \ \ \ \ \ \ \ \ \ \ \ \ \ \ \ \ \ \ \
\ \ \ \ \ \ \ \ \ \ \ \ \ \ \ \ \ \ \ \ \ \ \ \ \ \ \ \ \ \ \ \ \ \ \ \ \
\end{equation*}%
\begin{equation}
\ \ \ \ \ \ \ \ \ \ \ \ \ \ \times \sum\limits_{a=1}^{K-a}(-1)^{K-a}\left(
\frac{K}{a}\right) \sigma _{1}\sigma _{2}\int_{\xi _{1}}^{\xi _{2}}e(2\sqrt{n%
}(\sqrt{U}+\xi )-1/8)d\xi \ll X^{1/4+\varepsilon (X)}
\end{equation}%
By (9.41), (9.46), (9.47) and (9.51), running the process of the evaluation
of (9.40), we obtain
\begin{equation*}
\Omega _{2}^{\prime }=C\sqrt{U}\ \mathrm{Re}\sum\limits_{n\leq \frac{UL^{2}}{%
V}}d(n)n^{-3/4}e^{-\pi n(\frac{V}{U}+\frac{4}{A^{2}})}\ \ \ \ \ \ \ \ \ \ \
\ \ \ \ \ \ \ \ \ \ \ \ \ \ \ \ \ \ \ \ \ \ \ \ \ \ \ \ \ \ \ \ \ \ \ \ \ \
\ \ \ \ \ \ \ \ \ \ \ \ \ \ \ \ \ \ \ \ \
\end{equation*}%
\begin{equation*}
\times \sum\limits_{a=1}^{K}(-1)^{K}\left( \frac{K}{a}\right)
(I_{1}I_{2}+I_{1}\sigma _{2}+I_{2}\sigma _{1})\int_{\xi _{1}}^{\xi _{2}}e(2%
\sqrt{n}(\sqrt{U}+\xi )-1/8)d\xi +O(X^{1/4+\varepsilon (X)}).
\end{equation*}%
Thus, we may rewrite in the form:
\begin{equation*}
\Omega _{2}^{\prime }=X\sum\limits_{a=1}^{4K+2}C_{1}(a)\sum\limits_{n\leq
\frac{X_{2}L^{2}}{V}}d(n)n^{-7/4}\cos (4\pi \sqrt{n}(\sqrt{X}+\frac{am_{0}}{2%
\sqrt{X}})-\pi /4)\ \ \ \ \ \ \ \ \ \ \ \ \ \ \ \ \ \ \ \ \ \ \ \ \ \ \ \ \
\ \ \ \ \ \ \ \ \ \ \ \ \ \ \ \ \ \ \ \
\end{equation*}%
\begin{eqnarray}
&&+\sqrt{X}\sum\limits_{a=1}^{4K+2}(C_{2}(a)+C_{3}(a)\{X\})\sum\limits_{n%
\leq \frac{X_{2}L^{2}}{V}}d(n)n^{-5/4}\cos (4\pi \sqrt{n}(\sqrt{X}+\frac{%
am_{0}}{2\sqrt{X}})-3\pi /4)  \notag \\
&&+O(X^{1/4+\varepsilon (X)}),
\end{eqnarray}%
where $C_{j}(a)$ satisfies (9.45).

Finally, by (9.13), (9.40), (9.44) and (9.52) we obtain that
\begin{eqnarray}
\Omega  &=&\Omega _{0}+\Omega _{1}+\Omega _{2}^{\prime }+\Omega _{20}+\Omega
_{\delta }+O(X^{1/4+\varepsilon (X)})  \notag \\
&=&X\sum\limits_{a=1}^{4K+2}C_{1}(a)\sum\limits_{n\leq \frac{X_{2}L^{2}}{V}%
}d(n)n^{-7/4}\cos (4\pi \sqrt{n}(\sqrt{X}+\frac{am_{0}}{2\sqrt{X}})-\pi /4)
\notag \\
&&+\sqrt{X}(m_{0}+1)\sum\limits_{a=1}^{4K}C_{2}(a)\sum\limits_{n\leq \frac{%
X_{2}L^{2}}{V}}d(n)n^{-5/4}\cos (4\pi \sqrt{n}(\sqrt{X}+\frac{am_{0}}{2\sqrt{%
X}})-3\pi /4)  \notag \\
&&+\sqrt{X}  \notag \\
&&\times \sum\limits_{a=1}^{4K+2}(C_{3}(a)+C_{4}(a)\{X\})\sum\limits_{n\leq
\frac{X_{2}L^{2}}{V}}d(n)n^{-5/4}\cos (4\pi \sqrt{n}(\sqrt{X}+\frac{am_{0}}{2%
\sqrt{X}})-3\pi /4)  \notag \\
&&+(m_{0}+1)  \notag \\
&&\times \sum\limits_{a=1}^{4K}(C_{5}(a)+C_{6}(a)\{X\})\sum\limits_{n\leq
\frac{X_{2}L^{2}}{V}}d(n)n^{-3/4}\cos (4\pi \sqrt{n}(\sqrt{X}+\frac{am_{0}}{2%
\sqrt{X}})-\pi /4)  \notag \\
&&+\Omega _{0}+\Omega _{\delta }+O(X^{1/4+\varepsilon (U)})
\end{eqnarray}%
for $U=[X],\ m_{0}\asymp \sqrt{X}L^{-2}$, where $C_{j}(a)$ is real and
independent of $X,$
\begin{equation}
|C_{j}(a)|\ll 2^{K}\ll \mathrm{exp}(O(\frac{L}{\log (L)})),
\end{equation}%
$\Omega _{0}$ is (9.14), $\Omega _{\delta }$ is (9.17) and (9.18).

\ \

\section{ PROOF OF THEOREM}

By (9.14) and (9.10),
\begin{equation*}
\Omega _{0}=\frac{(-1)^{K+1}\sqrt{2U}(m_{0}+1)^{2}}{\pi }\ \ \ \ \ \ \ \ \ \
\ \ \ \ \ \ \ \ \ \ \ \ \ \ \ \ \ \ \ \ \ \ \ \ \ \ \ \ \ \ \ \ \ \ \ \ \ \
\ \ \ \ \ \ \ \
\end{equation*}%
\begin{equation}
\times \sum\limits_{n\leq \frac{UL^{2}}{V}}d(n)n^{-3/4}e^{-\pi n(\frac{V}{U}+%
\frac{4}{A^{2}})}\int_{\xi _{1}}^{\xi _{2}}\cos (4\pi \sqrt{n}(\sqrt{U}+\xi
)-\pi /4)d\xi .
\end{equation}%
Since
\begin{equation*}
e^{-\pi n(\frac{V}{U}+\frac{4}{A^{2}})}\ll e^{-L^{2}}
\end{equation*}%
for $\frac{UL^{2}}{V}<n\leq \frac{X_{2}L^{2}}{V}$, then putting $\xi =\frac{1%
}{\sqrt{U}}\xi ^{\prime }$, we have
\begin{eqnarray}
\Omega _{0} &=&\frac{2(-1)^{K+1}(m_{0}+1)^{2}}{\sqrt{2}\pi }%
\sum\limits_{n\leq \frac{X_{2}L^{2}}{V}}d(n)n^{-3/4}e^{-\pi n(\frac{V}{U}+%
\frac{4}{A^{2}})}  \notag \\
&&\times \int_{1/16}^{3/16}\cos (4\pi \sqrt{n}(\sqrt{U}+\xi /\sqrt{U})-\pi
/4)d\xi +O(U^{-10})  \notag \\
&=&\frac{(-1)^{K+1}(m_{0}+1)^{2}}{4X^{1/4}}(\frac{X^{1/4}}{\sqrt{2}\pi }%
\sum\limits_{n\leq \frac{X_{2}L^{2}}{V}}d(n)n^{-3/4}\cos (4\pi \sqrt{nX}-\pi
/4)  \notag \\
&&+\delta _{01}(X)+\delta _{02}(X)+O(U^{-1})),
\end{eqnarray}%
\newline
where
\begin{equation*}
\delta _{01}(X)=C_{1}X^{1/4}\ \ \ \ \ \ \ \ \ \ \ \ \ \ \ \ \ \ \ \ \ \ \ \
\ \ \ \ \ \ \ \ \ \ \ \ \ \ \ \ \ \ \ \ \ \ \ \ \ \ \ \ \ \ \ \ \ \ \ \ \ \
\ \ \ \ \ \ \ \ \ \ \ \ \ \ \ \ \ \ \ \ \ \ \ \ \
\end{equation*}%
\begin{equation}
\times \sum\limits_{n\leq \frac{X_{2}L^{2}}{V}}d(n)n^{-3/4}(e^{-\pi n(\frac{V%
}{U}+\frac{4}{A^{2}})}-1)\int_{1/16}^{3/16}\cos (4\pi \sqrt{n}(\sqrt{U}+\xi /%
\sqrt{U})-\pi /4)d\xi ,
\end{equation}%
\begin{equation*}
\delta _{02}(X)=C_{2}X^{1/4}\ \ \ \ \ \ \ \ \ \ \ \ \ \ \ \ \ \ \ \ \ \ \ \
\ \ \ \ \ \ \ \ \ \ \ \ \ \ \ \ \ \ \ \ \ \ \ \ \ \ \ \ \ \ \ \ \ \ \ \ \ \
\ \ \ \ \ \ \ \ \ \ \ \ \ \ \ \ \ \ \ \ \ \ \ \ \
\end{equation*}%
\begin{equation}
\times \sum\limits_{n\leq \frac{X_{2}L^{2}}{V}}d(n)n^{-3/4}%
\int_{1/16}^{3/16}(\cos (4\pi \sqrt{n}(\sqrt{U}+\xi /\sqrt{U})-\pi /4)-\cos
(4\pi \sqrt{nX}-\pi /4))d\xi ,\ \ \ \ \ \ \ \ \ \ \ \ \ \ \ \ \ \ \
\end{equation}%
with $C_{1},C_{2}=O(1)$ be real and independent of $X$. Denote
\begin{equation}
\delta _{00}(X)=\Delta (X)-\frac{X^{1/4}}{\pi \sqrt{2}}\sum\limits_{n\leq
\frac{X_{2}L^{2}}{V}}d(n)n^{-3/4}\cos (4\pi \sqrt{nX}-\pi /4).
\end{equation}%
By (10.2),
\begin{equation}
\Omega _{0}=\frac{(-1)^{K+1}(m_{0}+1)^{2}}{4X^{1/4}}(\Delta (X)-\delta
_{0}(X)+O(U^{-1})),
\end{equation}%
where
\begin{equation}
\delta _{0}(X)=\delta _{00}(X)-\delta _{01}(X)-\delta _{02}(X).
\end{equation}%
For $X_{1}\leq X\leq X_{2}$, $X_{1}\asymp X_{2}$, $V=X_{2}^{\varepsilon
_{0}} $, $\varepsilon _{0}=1/\log L$, (2.59) implicates
\begin{equation}
\delta _{00}(X)\ll X^{\varepsilon }.
\end{equation}%
Using (9.29) and taking $\alpha =-1/4$ in Lemma 3.4, we have
\begin{equation}
\delta _{01}(X)\ll X^{\varepsilon }.
\end{equation}%
Since
\begin{equation*}
\cos (4\pi \sqrt{n}(\sqrt{U}+\xi /\sqrt{U})-\pi /4)-\cos (4\pi \sqrt{nX}-\pi
/4)=\ \ \ \ \ \ \ \ \ \ \ \ \ \ \ \ \ \ \ \ \ \ \ \ \ \ \ \ \ \ \ \ \ \ \ \
\ \ \ \ \ \ \ \ \ \ \ \ \ \ \ \ \ \
\end{equation*}%
\begin{equation*}
=\mathrm{Re}\ e(2\sqrt{nX}-1/8)(e(2\sqrt{n}(\sqrt{U}+\xi /\sqrt{U}-\sqrt{X}%
))-1)\ \ \ \ \ \ \ \ \ \ \ \ \ \ \ \ \ \ \ \ \ \ \ \ \ \ \ \ \ \ \ \ \ \ \ \
\ \ \ \ \ \ \ \
\end{equation*}%
\begin{equation*}
=\sqrt{n}\mathrm{Re}\ e(2\sqrt{nX}-1/8)4\pi i(\sqrt{U}-\sqrt{X}+\xi /\sqrt{U}%
)\int_{0}^{1}e(2\theta \sqrt{n}(\sqrt{U}-\sqrt{X}+\xi /\sqrt{U}))d\theta ,
\end{equation*}%
then using Lemma 3.4, we have
\begin{equation*}
\delta _{02}(X)\ll X^{\varepsilon }.
\end{equation*}%
Thus,
\begin{equation}
\delta _{0}(X)=\delta _{00}(X)-\delta _{01}(X)-\delta _{02}(X)\ll
X^{\varepsilon }.
\end{equation}%
Moreover, by (9.53) and (10.6),
\begin{equation}
\Omega =\frac{(-1)^{K+1}(m_{0}+1)^{2}}{4X^{1/4}}(\Delta (X)-X^{1/4}\Lambda
(X)-\delta (X)+O(X^{-1/2+\varepsilon (X)})),
\end{equation}%
where
\begin{equation*}
\Lambda (X) =\frac{X}{(m_{0}+1)^{2}}\sum\limits_{a=1}^{4K+2}C_{1}(a)\sum%
\limits_{n\leq \frac{X_{2}L^{2}}{V}}d(n)n^{-7/4}\cos (4\pi \sqrt{n}(\sqrt{X}+%
\frac{am_{0}}{2\sqrt{X}})-\pi /4)
\end{equation*}
\begin{equation}
+\frac{\sqrt{X}}{m_{0}+1}\sum\limits_{a=1}^{4K}C_{2}(a)\sum\limits_{n\leq
\frac{X_{2}L^{2}}{V}}d(n)n^{-5/4}\cos (4\pi \sqrt{n}(\sqrt{X}+\frac{am_{0}}{2%
\sqrt{X}})-3\pi /4),
\end{equation}
\begin{equation}
\delta (X) =\delta _{0}(X)+\delta _{1}(X)+\delta _{2}(X),
\end{equation}%
The three terms in the right side of the last equality are the following:

\ \

\noindent $\delta _{0}(X)$ is (10.10),
\begin{equation*}
\delta _{1}(X) \ \ \ \ \ \ \ \ \ \ \ \ \ \ \ \ \ \ \ \ \ \ \ \ \ \ \ \ \ \ \
\ \ \ \ \ \ \ \ \ \ \ \ \ \ \ \ \ \ \ \ \ \ \ \ \ \ \ \ \ \ \ \ \ \ \ \ \ \
\ \ \ \ \ \ \ \ \ \ \ \ \ \ \ \ \ \ \ \ \ \ \ \ \ \ \ \ \ \ \ \ \ \ \ \ \ \
\ \ \ \ \
\end{equation*}
\begin{equation*}
=\frac{X^{3/4}}{(m_{0}+1)^{2}}\sum\limits_{a=1}^{4K+2}(C_{3}(a)+C_{4}(a)\{X%
\})\sum\limits_{n\leq \frac{X_{2}L^{2}}{V}}d(n)n^{-5/4}\cos (4\pi \sqrt{n}(%
\sqrt{X}+\frac{am_{0}}{2\sqrt{X}})-3\pi /4)
\end{equation*}%
\begin{equation*}
+\frac{X^{1/4}}{(m_{0}+1)}\sum\limits_{a=1}^{4K}(C_{5}(a)+C_{6}(a)\{X\})\sum%
\limits_{n\leq \frac{X_{2}L^{2}}{V}}d(n)n^{-3/4}\cos (4\pi \sqrt{n}(\sqrt{X}+%
\frac{am_{0}}{2\sqrt{X}})-\pi /4),
\end{equation*}%
\begin{equation}
\delta _{2}(X)=\frac{X^{1/4}}{(m_{0}+1)^{2}}\Omega _{\delta }=\frac{X^{1/4}}{%
(m_{0}+1)^{2}}\sum\limits_{k,m=m_{0}}^{2m_{0}}\delta _{R}(\eta ,m)|_{\eta
_{K}}\ll X^{\varepsilon },\ \ \ \ \ \ \ \ \ \ \ \ \ \ \ \ \ \ \ \ \ \ \ \ \
\ \ \ \
\end{equation}%
where $\Omega _{\delta }$ is (9.17), $\delta _{R}(\eta ,m)$ is (9.5), $%
C_{j}(a),\delta _{2}(X)$ contain new constant $C=O(1)$ to be independent of $%
X$. \newline
Therefore, by (1.26) and (10.11) we obtain that
\begin{equation}
\Delta (X)=X^{1/4}\Lambda (X)+\delta (X)+O(X^{-1/2+\varepsilon (X)}),\
0<\varepsilon (X)=O(\frac{1}{\log \log X}).
\end{equation}

The proof of Theorem is complete. \ \ \ \ \ \ \ \ \ \ \ \ \ \ \ \ \ \ \ \ \
\ \ \ \ \ \ \ \ \ \ \ \ \ \ \ \ \ \ \ \ \ \ \ \ \ \ \ \ \ \ \ \ \ \ \ \ \ \
\ \ \ \ \ $\boxempty$

\ \

\textbf{NOTE}

For the convenient of the other application afterwards we write down the
accurate $\delta (X)$ of (10.15):
\begin{equation*}
\delta (X) =\Delta (X)-\frac{X^{1/4}}{\sqrt{2}\pi }\sum\limits_{n\leq \frac{%
X_{2}L^{2}}{V}}d(n)n^{-3/4}\cos (4\pi \sqrt{nX}-\pi /4) \ \ \ \ \ \ \ \ \ \
\ \ \ \ \ \ \ \ \ \ \ \ \ \ \ \ \ \ \ \ \ \ \ \ \ \
\end{equation*}
\begin{equation*}
+C_{1}X^{1/4}\sum\limits_{n\leq \frac{X_{2}L^{2}}{V}}d(n)n^{-3/4}(e^{-\pi n(%
\frac{V}{U}+\frac{4}{A^{2}})}-1)\int_{1/16}^{3/16}\cos (4\pi \sqrt{n}(\sqrt{U%
}+\xi /\sqrt{U})-\pi /4)d\xi \ \ \ \ \ \ \ \
\end{equation*}
\begin{equation*}
+C_{2}X^{1/4}\sum\limits_{n\leq \frac{X_{2}L^{2}}{V}}d(n)n^{-3/4}%
\int_{1/16}^{3/16}(\cos (4\pi \sqrt{n}(\sqrt{U}+\xi /\sqrt{U})-\pi /4)-\cos
(4\pi \sqrt{nX}-\pi /4))d\xi \ \ \ \ \
\end{equation*}
\begin{equation*}
+\frac{X^{3/4}}{(m_{0}+1)^{2}}\sum\limits_{a=1}^{4K+2}(C_{3}(a)+C_{4}(a)\{X%
\})\sum\limits_{n\leq \frac{X_{2}L^{2}}{V}}d(n)n^{-5/4}\cos (4\pi \sqrt{n}(%
\sqrt{X}+\frac{am_{0}}{2\sqrt{X}})-3\pi /4) \ \ \ \ \
\end{equation*}
\begin{equation*}
+\frac{X^{1/4}}{(m_{0}+1)}\sum\limits_{a=1}^{4K}(C_{5}(a)+C_{6}(a)\{X\})\sum%
\limits_{n\leq \frac{X_{2}L^{2}}{V}}d(n)n^{-3/4}\cos (4\pi \sqrt{n}(\sqrt{X}+%
\frac{am_{0}}{2\sqrt{X}})-\pi /4) \ \ \ \ \
\end{equation*}
\begin{equation}
+\frac{X^{1/4}}{(m_{0}+1)^{2}\sqrt{V}}\sum\limits_{-\sqrt{V}L\leq j\leq
\sqrt{V}L}e^{-\frac{\pi j^{2}}{V}}\sum\limits_{k,m=m_{0}}^{2m_{0}}\int_{\xi
_{1}}^{\xi _{2}}\delta _{Q}(B_{1})d\xi |_{\eta _{K}}, \ \ \ \ \ \ \ \ \ \ \
\ \ \ \ \ \ \ \ \ \ \ \
\end{equation}

where $\delta _{Q}(B_{1})$ is (8.28), $X_{1}\leq X\leq X_{2},\ U=[X]$, $%
m_{0}\asymp \sqrt{X_{2}}L^{-2},\ X_{1}\asymp X_{2}$, $L=\log X_{2},\ \xi
_{1}=1/16\sqrt{U},\ \xi _{2}=3/16\sqrt{U},\ A=C_{0}\sqrt{UL}$ $C_{j}(a)$ is
independent of $X$, and
\begin{equation*}
|C_{j}(a)|\ll exp(O(\frac{L}{\log L})).
\end{equation*}%
We can take $V=X_{2}^{\varepsilon _{0}}$, $\varepsilon _{0}=1/\log L$, $%
m_{0}=[\sqrt{X_{2}}L^{-2}]$, such that they are independent of $X$.

\ \

{\it Acknowledgements.} I would like to thank Professor Liang Zhang
for pointing out some errors and checking the paper.

\ \


\begin{thebibliography}{99}
\bibitem{[1]} E.C.Titchmarsh, (revised by Heatu-Brown D.R.) Theory of The
Riemann Zeta-Function, Second edition, Oxford, 1986.

\bibitem{[2]} $Ivic^{\prime}$,A., The Riemann Zeta-Function, John Wiley \ \&
\ Sons, New York, 1985.

\bibitem{[3]} Raymond Ayoub, An Introduction To The Analytic Theory Of
Number, American Mathematical Society, 1963.

\bibitem{[4]} Dirichlet,P.G.L. \"{U}ber die Bestimmung der mittleren werte
in der Zahlentheorie, Abh.Akad.Berlin(werke 2, 49-

66)1849, Math. Abh., 69-83

\bibitem{[5]} Voronoi,G., Sur un problem du calcul des functions
asymptotiques, J. f\"{u}r math. 126(1903)241-282

\bibitem{[6]} Corput,J.G. Van der,Versch\"{a}rfung der Absch\"{a}tzung beim
Teilerproblem, M. A. 87(1922)39-65

\bibitem{[7]} Corput,J.G. Van der, Zum Teilerproblem, M.A. 98(1928)697-716

\bibitem{[8]} $Chih,T.T.$ A divisor problem, Acad. Sinica Sci. Record,
3(1950)177-182

\bibitem{[9]} Richert,H.E., Versch\"{a}rfung der Absch\"{a}tzung beim
Dirichletschen Teilerproblem, Math. Z, 58, (1953)204-

218

\bibitem{[10]} Yin Wenlin, Dirichlet Divisor Problem, J. Peking
University(Natural Science), 2(1959), 103-126

\bibitem{[11]} Kolesnik,G., The improvement of the error term in the divisor
problem, Mat.Zametki, 6(1969), 545-

554

\bibitem{[12]} Kolesnik,G., On the estimation of certain trigonometric sum,
Acta Arith. 25(1973)7-30

\bibitem{[13]} Kolesnik,G., On the order of $\zeta(1/2+it)$ and $\Delta(R)$,
Pacific J. Math. 82(1982), 107-122

\bibitem{[14]} Kolesnik,G., On the method of expanent pairs, Acta Arith.
45(1985), 115-143

\bibitem{[15]} Iwaniec,H. and Mozzochi,C.J., On the divisor and circle
problem, J.Number Theory 29(1988), no.1, 60-

93

\bibitem{[16]} Hardy,G.H., On Dirichlet's divisor problem,Proc. London Math.
Soc. 15(1916),1-25

\bibitem{[17]} Matti Jutila, Transformation formule for Dirichlet
polynomials, Journal of number 18, 135-

156(1984)

\bibitem{[18]} M.N.Huxley, Exponential Sums and Lattice Points III, Proc.
London Math. Soc. (3)87:591-

609(2003)
\end{thebibliography}
\end{document}